\documentclass[11pt,reqno]{amsart}
\usepackage[a4paper]{geometry}

\usepackage{amsmath,amssymb,amsfonts,mathrsfs}
\usepackage{graphicx}
\usepackage{xcolor}
\usepackage[colorlinks=true,linkcolor=blue,citecolor=blue,urlcolor=blue]{hyperref}

\usepackage{cite}

\newtheorem{theorem}{Theorem}[section]

\newtheorem{lemma}{Lemma}[section]
\newtheorem{proposition}[lemma]{Proposition}
\newtheorem{corollary}[theorem]{Corollary}

\theoremstyle{definition}
\newtheorem{definition}[lemma]{Definition}

\theoremstyle{remark}
\newtheorem{remark}{Remark}[section]

\numberwithin{equation}{section}

\begin{document}

\title[INVERSE SCATTERING THEORY IN HYPERBOLIC SPACE]{Inverse Scattering from Conformal Infinity for Totally Geodesic Defects in Hyperbolic Space}

\author{Lu Chen}
\address[Lu Chen]{Key Laboratory of Algebraic Lie Theory and Analysis, Ministry of Education, School of Mathematics and Statistics, Beijing Institute of Technology, Beijing 100081, PR China}
\email{chenlu5818804@163.com}

\author{Hongyu Liu}
\address[Hongyu Liu]{Department of Mathematics, City University of Hong Kong, Hong Kong SAR, China}
\email{hongyu.liuip@gmail.com, hongyliu@cityu.edu.hk}


\author{Longyue Tao}
\address[Longyue Tao]{Department of Mathematics, City University of Hong Kong, Hong Kong SAR, China}
\email{sdyctly@163.com, longyue.tao@my.cityu.edu.hk}



\date{\today}

\begin{abstract}

We study inverse scattering from conformal infinity for impenetrable topological defects whose interaction surfaces are totally geodesic in hyperbolic space \(\mathbb H^n\), \(n \geq 2\).
For a fixed spectral parameter \(\lambda_0 > 0\), the prescribed inputs are boundary labels \(\xi \in \partial_\infty \mathbb H^n\).
Each label selects an incoming Helgason mode in the hyperbolic interior.
Given a defect \(\mathcal P \Subset \mathbb H^n\), this mode generally fails to satisfy the homogeneous trace condition on the interaction surface and hence generates an outgoing correction.
The leading coefficient of this correction at conformal infinity defines the measured far-field pattern.
The central question is whether \(\mathcal P\) can be recovered from the far-field patterns corresponding to one or finitely many prescribed boundary labels.
This gives a formally determined inverse problem at one fixed spectral parameter.

Our first main result establishes the unique determination of a class of totally geodesic defects that includes both bulk components and hypersurface-supported ones.
For Dirichlet-type defects, a single boundary label suffices.
For Neumann-type defects, \(n+1\) boundary labels are sufficient and in general necessary, provided they satisfy the natural affine-independence condition at conformal infinity.

Our second main result provides {sharp} quantitative stability estimates within the same framework.
The hyperbolic Hausdorff distance between two defects is controlled by the discrepancy of their far-field patterns at conformal infinity.
The proof combines a delicate continuation from conformal infinity with quantitative hyperbolic reflection across totally geodesic hypersurfaces.
Taken together, these results yield a comprehensive qualitative and quantitative far-field inverse scattering theory at conformal infinity, based on formally determined data.

\end{abstract}

\maketitle

\tableofcontents

\section{Introduction}

\subsection{Background and motivation}

Scattering theory on hyperbolic spaces and asymptotically hyperbolic manifolds is naturally organized by the geometry at \emph{conformal infinity}.
Under conformal compactification, the boundary at infinity becomes a boundary of the compactified space and carries the asymptotic coefficients of bulk fields.
This point of view has its roots in spectral and automorphic scattering on hyperbolic quotients, including the work of Lax--Phillips and Patterson on scattering theory and Eisenstein series \cite{LaxPhillips1976,Patterson1975}.
It was further developed in automorphic and spectral scattering by Faddeev--Pavlov and Perry \cite{FaddeevPavlov1972,Perry1989}.
Friedlander introduced the radiation-field approach in the Euclidean setting \cite{Friedlander1980}, and S\'a Barreto subsequently adapted it to asymptotically hyperbolic manifolds \cite{SaBarreto2005}.
The resolvent and scattering matrix theory on asymptotically hyperbolic manifolds was developed by Mazzeo--Melrose, Graham--Zworski, and Guillarmou \cite{MazzeoMelrose1987,GrahamZworski2003,Guillarmou2005}.
In this setting, scattering data are objects at conformal infinity, and inverse scattering asks what information about the hyperbolic interior is encoded there.
Hyperbolic geometry also arises naturally in black-hole physics: near a static nondegenerate Killing horizon, the associated optical geometry is asymptotically hyperbolic \cite{GibbonsWarnick2009}.

The inverse scattering theory on asymptotically hyperbolic manifolds makes this principle precise.
In the stationary theory, Joshi--S\'a Barreto showed that, away from exceptional energies, the scattering matrix determines boundary asymptotics of the metric, and analogously of a potential \cite{JoshiSaBarreto2000}.
Radiation-field and boundary-control methods later led to global and partial-data recovery results from scattering data at conformal infinity \cite{SaBarreto2005,HoraSaBarreto2015}.
S\'a Barreto--Wang showed that the scattering operator is a Fourier integral operator associated with the scattering relation \cite{SaBarretoWang2019}.

These results use global operators at conformal infinity, such as scattering matrices, scattering operators, or radiation-field maps.
Such operators encode families of boundary inputs and outputs, often depending on a spectral or time parameter.
The present paper asks a complementary finite-label question at one fixed spectral parameter: what can be recovered from a single prescribed boundary label, or from at most finitely many such labels?

To formulate this question, we use the hyperbolic Sommerfeld--Rellich framework of Chen--Liu \cite{ChenLiu2023}.
This framework provides the hyperbolic radiation condition, the Rellich theorem, and the far-field expansion needed to identify the leading outgoing coefficient at conformal infinity.
After fixing a spectral parameter \(\lambda_0>0\), we prescribe a finite family of boundary labels
\[
    \Xi=(\xi_1,\ldots,\xi_m)
    \in
    \bigl(\partial_\infty\mathbb H^n\bigr)^m .
\]
Each label \(\xi_\ell\) selects an incoming Helgason mode \(u^i_{\xi_\ell}\), defined below by the Helgason--Fourier kernel.
The data considered here are the corresponding far-field patterns at conformal infinity.
For each prescribed label, the observation variable still ranges over all of \(\partial_\infty\mathbb H^n\).
We ask whether these prescribed labels and far-field patterns determine a compact impenetrable defect in the hyperbolic interior.
This is the far-field inverse problem at conformal infinity studied in this paper.

The problem thus exhibits a holographic boundary-to-bulk character in a precise geometric sense: the input labels and observations reside on \(\partial_\infty\mathbb H^n\), while the unknown defect lies in the interior \(\mathbb H^n\). This viewpoint is strongly reminiscent of the role of conformal infinity in hyperbolic and AdS-type geometries, where boundary data may encode bulk information \cite{Maldacena1998,GubserKlebanovPolyakov1998,Witten1998}. This naturally motivates the recovery of a compact hyperbolic defect from formally determined far-field data.

The present paper studies inverse scattering for defects whose interaction surfaces are finite unions of totally geodesic faces.
This class is intrinsic to hyperbolic geometry and includes both bulk components and components supported on hypersurfaces.
A supporting totally geodesic hypersurface \(V\) has its own ideal boundary
\[
    \partial_\infty V \subset \partial_\infty\mathbb H^n .
\]
The incidence of the prescribed boundary labels with these ideal boundaries is the geometric mechanism behind the finite-label uniqueness threshold, especially for the Neumann trace condition. In the Neumann-type case, the labels must not all lie in the ideal boundary of one supporting totally geodesic hypersurface; this is minimally ensured by $n+1$ affinely independent boundary labels. We prove that such defects can be uniquely and stably determined from far-field patterns generated by a single, or at most finitely many, prescribed boundary labels at conformal infinity. The stability estimates control the hyperbolic Hausdorff distance between defects by the discrepancy of their far-field patterns, with {sharp} quantitative moduli. The precise assumptions and main results, including the admissible boundary labels and the obstruction in the Neumann case, are stated below.

\subsection{The far-field inverse problem at conformal infinity}

Let \(\mathbb H^n\), \(n\geq 2\), be the \(n\)-dimensional hyperbolic space.
Throughout this paper, we identify \(\mathbb H^n\) with the Poincar\'e ball model \(\mathbb B^n\).
Its conformal infinity is
\[
    \partial_\infty\mathbb H^n
    =
    \partial\mathbb B^n
    =
    \mathbb S^{n-1}.
\]
The direct and inverse problems below are formulated in the hyperbolic Sommerfeld--Rellich framework of \cite{ChenLiu2023}.
In this framework, outgoing solutions are selected by a radiation condition at conformal infinity, and the leading outgoing coefficient is the far-field pattern.

Let \(\mathcal P\Subset\mathbb B^n\) be an \emph{impenetrable topological defect}, namely a compact set in the hyperbolic interior across which exterior solutions are not continued, with connected exterior
\[
    G:=\mathbb B^n\setminus\mathcal P.
\]
Here, the hyperbolic interior is the interior of the conformal compactification, while the far-field data are recorded at \(\partial_\infty\mathbb H^n\).
The \emph{defect} interacts with the exterior solution through a homogeneous trace condition on its interaction surface.
For notational simplicity, we write \(\partial\mathcal P\) for this interaction surface, with the trace interpretation understood for lower-dimensional components.

A component of \(\mathcal P\) with nonempty interior is called a \emph{bulk component}.
A component with empty interior is called a \emph{hypersurface-supported component}.
In this general structural class, both types of components may occur.
For hypersurface-supported components, the trace condition is imposed from each local side belonging to \(G\).

Fix a spectral parameter \(\lambda_0>0\).
For a boundary label \(\xi\in\partial_\infty\mathbb H^n\), we prescribe the \emph{incoming Helgason mode}
\begin{equation}\label{eq:incoming-helgason-mode}
    u^i_\xi(x):=e_{2\lambda_0,\xi}(x),
    \qquad x\in\mathbb B^n .
\end{equation}
When the boundary label \(\xi\) is fixed, we simply write \(u^i\) for \(u^i_\xi\).
Here, \(e_{\lambda,\xi}\) denotes the Helgason--Fourier kernel on the Poincar\'e ball, given by
\[
    e_{\lambda,\xi}(x)
    =
    \left(
    \frac{\sqrt{1-|x|^2}}{|x-\xi|}
    \right)^{n-1+\mathrm i\lambda},
    \qquad x\in\mathbb B^n,\quad \xi\in\mathbb S^{n-1}.
\]
Equivalently, \(e_{\lambda,\xi}\) is a generalized eigenfunction of the hyperbolic Laplacian selected by the boundary label \(\xi\).
Thus, \(u^i_\xi\) is the \emph{incoming Helgason mode} associated with the \emph{boundary label \(\xi\)}.
Set
\[
    \mathcal L_{\lambda_0}
    :=
    -\Delta_{\mathbb H}
    -
    \frac{(n-1)^2}{4}
    -
    \lambda_0^2 .
\]
By the eigenvalue relation for the Helgason kernel, one has
\begin{equation}\label{eq:ui}
	\mathcal L_{\lambda_0}u^i_\xi=0
    \qquad
    \text{in }\mathbb B^n .
\end{equation}
    
For the pair \((\mathcal P,\xi)\), let \(u=u_{\mathcal P,\xi}\) denote the exterior solution.
The incoming Helgason mode \(u^i_\xi\) solves \eqref{eq:ui} in the whole Poincar\'e ball, but it does not generally satisfy the constraints imposed by the geometric configuration of the defect \(\mathcal P\), its interaction surface, and the trace condition on \(\partial\mathcal P\).
Here and below, \(\mathscr B_{\mathcal P}\) denotes the trace operator associated with the chosen \(\mathsf D\)- or \(\mathsf N\)-type condition, specified below.
The \emph{outgoing correction} generated by the defect is defined by
\[
    u^s_{\mathcal P,\xi}:=u_{\mathcal P,\xi}-u^i_\xi .
\]
Equivalently, \(u^s_{\mathcal P,\xi}\) is characterized by
\[
\begin{cases}
\mathcal L_{\lambda_0}u^s_{\mathcal P,\xi}=0,
& \text{in }G, \medskip\\
\mathscr B_{\mathcal P} u^s_{\mathcal P,\xi}
=
-\mathscr B_{\mathcal P} u^i_\xi,
& \text{on }\partial\mathcal P, \medskip\\
u^s_{\mathcal P,\xi}\text{ is outgoing at }\partial_\infty\mathbb H^n.
\end{cases}
\]
When \((\mathcal P,\xi)\) is fixed, we simply write \(u^s\) for \(u^s_{\mathcal P,\xi}\).

The direct problem for the exterior solution \(u=u_{\mathcal P,\xi}\),
with outgoing correction \(u^s=u^s_{\mathcal P,\xi}\), is
\begin{equation}\label{eq:hhelm1}
\begin{cases}
\mathcal L_{\lambda_0} u=0,
& \text{in } G, \medskip\\
\mathscr B_{\mathcal P}u=0,
& \text{on } \partial\mathcal P, \medskip\\
\dfrac{\partial u^s}{\partial \rho}
-\left(\lambda_0\mathrm i-\dfrac{n-1}{2}\right)
\tanh\left(\dfrac{\rho}{2}\right)u^s
=o\Bigl(\sinh^{-(n-1)/2}(\rho)\Bigr),
& \text{as } \rho\to+\infty.
\end{cases}
\end{equation}
Here, \(\rho=\rho(x)\) denotes the hyperbolic distance from the origin.
The last condition in \eqref{eq:hhelm1} is the hyperbolic Sommerfeld radiation condition.
It selects the outgoing branch of the correction at conformal infinity.

The hyperbolic scattering picture is therefore the following.
A boundary label \(\xi\) selects an incoming Helgason mode in the interior.
The defect changes the corresponding exterior solution through a trace condition on \(\partial\mathcal P\).
The outgoing correction is then read off at \(\partial_\infty\mathbb H^n\) through its leading asymptotic coefficient.

The trace condition in \eqref{eq:hhelm1} is taken in one of the following two forms.
In the \emph{Dirichlet-type}, or \(\mathsf D\)-type, case,
\[
    u_{\mathcal P,\xi}=0
    \qquad
    \text{on }\partial\mathcal P .
\]
In the \emph{Neumann-type}, or \(\mathsf N\)-type, case,
\[
    \partial_{\nu_{\mathbb H}}u_{\mathcal P,\xi}=0
    \qquad
    \text{on }\partial\mathcal P ,
\]
{
Here \(\partial_{\nu_{\mathbb H}}\) denotes the hyperbolic normal
derivative defined in \eqref{eq:outnormal}.  For hypersurface-supported
components, the homogeneous condition is imposed from both local sides
whenever both sides belong to \(G\); its vanishing is independent of the sign
chosen for the unit hyperbolic normal.
}

For each defect considered in this paper, the well-posedness of \eqref{eq:hhelm1},
the radiation condition, the Rellich theorem, and the asymptotic expansion are understood
in the hyperbolic Sommerfeld--Rellich framework of Chen--Liu~\cite{ChenLiu2023};
see also the boundary-radiation viewpoint in~\cite{SaBarreto2005}.
In particular, the outgoing correction has the expansion at conformal infinity
\begin{equation}\label{eq:us}
u^s_{\mathcal P,\xi}(x)
=
\left(\cosh\left(\frac{\rho(x)}{2}\right)\right)^{2\lambda_0\mathrm i-(n-1)}
u_{\infty,\mathcal P,\xi}(\hat x)
+
\mathcal O\left(\sinh^{-\frac{n+1}{2}}(\rho(x))\right),
\qquad \rho(x)\to+\infty,
\end{equation}
where \(\hat x=x/|x|\in\mathbb S^{n-1}=\partial_\infty\mathbb H^n\).
The coefficient \(u_{\infty,\mathcal P,\xi}\) is the far-field pattern at conformal infinity generated by the defect \(\mathcal P\) and the incoming mode \(u^i_\xi\).
In the sequel, every far-field pattern is understood at conformal infinity, and the fixed spectral parameter \(\lambda_0\) is suppressed from the notation.

\medskip
The inverse problem considered in this paper is the recovery of the defect \(\mathcal P\) from its far-field patterns at conformal infinity.
Let \(u^i_{\xi_1},\ldots,u^i_{\xi_m}\) be prescribed incoming Helgason modes of the form \eqref{eq:incoming-helgason-mode}, with fixed spectral parameter \(\lambda_0\) and boundary labels \(\xi_\ell\in\partial_\infty\mathbb H^n\).
For a defect \(\mathcal P\), we denote by \(u_{\infty,\mathcal P,\xi_{\ell}}\) the far-field pattern at conformal infinity generated by \(\mathcal P\) and \(u^i_{\xi_\ell}\), \(\ell=1,\ldots,m\).
The inverse problem is formulated through the far-field map at conformal infinity
\begin{equation}\label{eq:hip1}
    \mathscr F(\mathcal P)
    :=
    \bigl(
    u_{\infty,\mathcal P,\xi_1},\ldots,
    u_{\infty,\mathcal P,\xi_m}
    \bigr)
    \in \bigl(L^2(\mathbb S^{n-1})\bigr)^m.
\end{equation}
Both the labels \(\xi_\ell\) and the output variable \(\hat x\) belong to conformal infinity \(\partial_\infty\mathbb H^n\), while \(\mathcal P\) lies in the hyperbolic interior.
Thus \eqref{eq:hip1} is an interior recovery problem from far-field data at conformal infinity.

For the inverse problem \eqref{eq:hip1}, we study two aspects: qualitative determination and quantitative determination.

\begin{itemize}
    \item \emph{Unique determination.}
    The qualitative issue asks whether the far-field map at conformal infinity determines \(\mathcal P\) uniquely in the prescribed defect class.
    Namely, whether
    \begin{equation}\label{eq:intro-uniqueness}
        u_{\infty,\mathcal P,\xi_\ell}(\hat x)
        =
        u_{\infty,\mathcal P',\xi_\ell}(\hat x)
        \quad
        \text{for all } \hat x\in\mathbb S^{n-1}
        \text{ and } \ell=1,\ldots, m
        \quad
        \Longleftrightarrow
        \quad
        \mathcal P=\mathcal P' .
    \end{equation}

    \item \emph{Stable determination.}
    The quantitative issue asks whether errors in the far-field patterns at conformal infinity control the hyperbolic geometric discrepancy of the defects.
    Namely, whether there exists a modulus of continuity \(\psi\), with \(\psi(t)\to0\) as \(t\to0^+\), such that
    \begin{equation}\label{eq:intro-stability}
        d_{\mathcal H}^{\mathbb H}(\mathcal P,\mathcal P')
        \leq
        \psi\!\left(
        \max_{\ell=1,\ldots,m}
        \|u_{\infty,\mathcal P,\xi_\ell}-u_{\infty,\mathcal P',\xi_\ell}\|_{L^2(\mathbb S^{n-1})}
        \right),
    \end{equation}
    where \(d_{\mathcal H}^{\mathbb H}\) is the hyperbolic Hausdorff distance between defects.
\end{itemize}

For the minimal-data formulation, we take \(m=1\) for Dirichlet-type defects and \(m=n+1\) for Neumann-type defects; the \(\mathsf N\)-type stability result below also allows larger admissible families of boundary labels.

\subsection{Main results}

We now state the qualitative and quantitative results for the inverse problems
\eqref{eq:intro-uniqueness} and \eqref{eq:intro-stability}.

A compact impenetrable defect \(\mathcal P\Subset\mathbb B^n\) is called a
\emph{totally geodesic defect} if its exterior \(G=\mathbb B^n\setminus\mathcal P\)
is connected and its interaction surface is a finite union of totally geodesic faces:
\[
    \partial\mathcal P
    =
    \bigcup_{j=1}^{J}\mathcal F_j,
    \qquad
    \mathcal F_j
    \text{ is a totally geodesic face for } j=1,\ldots,J.
\]
Here, \(\partial\mathcal P\) is understood in the interaction-surface sense introduced above.
The number \(J\) of faces is not assumed to be known a priori.
The class allows both bulk components and hypersurface-supported components.
The precise definitions of totally geodesic hypersurfaces, totally geodesic faces, and totally geodesic defects are given in Section~\ref{sec:reflection-principles}.

We consider two trace types: the \(\mathsf D\)-type case, corresponding to the homogeneous Dirichlet trace condition, and the \(\mathsf N\)-type case, corresponding to the homogeneous hyperbolic Neumann trace condition.
The totally geodesic structure is used through hyperbolic reflection and the ideal-boundary incidence of the supporting hypersurfaces.

\subsubsection{Unique determination of totally geodesic defects}

We first state the qualitative uniqueness results for the far-field map at conformal infinity \eqref{eq:hip1}.
They concern the determination of totally geodesic defects from far-field patterns at conformal infinity, in the sense of \eqref{eq:intro-uniqueness}.

\medskip
The first result treats the \(\mathsf D\)-type case.
It shows that one far-field pattern at conformal infinity, generated by one incoming Helgason mode with fixed spectral parameter and fixed boundary label, determines the \(\mathsf D\)-type totally geodesic defect.

\begin{theorem}\label{thm:mainip1}
Let \(\mathcal P\) and \(\mathcal P'\) be two \(\mathsf D\)-type totally geodesic defects in \(\mathbb B^n\).
Let \(\lambda_0>0\) and \(\xi\in\mathbb S^{n-1}=\partial_\infty\mathbb H^n\) be fixed.
Let
\[
    u^i_{\xi}=e_{2\lambda_0,\xi}
\]
be the incoming Helgason mode.
If
\[
    u_{\infty,\mathcal P,\xi}(\hat x)
    =
    u_{\infty,\mathcal P',\xi}(\hat x)
    \qquad
    \text{for all }\hat x\in\mathbb S^{n-1},
\]
then
\[
    \mathcal P=\mathcal P' .
\]
\end{theorem}

\medskip
The second result treats the \(\mathsf N\)-type case, where the incidence geometry at conformal infinity gives the optimal threshold of \(n+1\) affinely independent prescribed boundary labels.

\begin{theorem}\label{thm:mainip2}
Let \(\mathcal P\) and \(\mathcal P'\) be two \(\mathsf N\)-type totally geodesic defects in \(\mathbb B^n\).
Let \(\lambda_0>0\) be fixed.
Choose boundary labels
\[
    \xi_1,\ldots,\xi_{n+1}\in\mathbb S^{n-1}=\partial_\infty\mathbb H^n
\]
such that
\[
    \xi_2-\xi_1,\ldots,\xi_{n+1}-\xi_1
\]
are linearly independent in \(\mathbb R^n\).
For \(\ell=1,\ldots,n+1\), let
\[
    u^i_{\xi_\ell}=e_{2\lambda_0,\xi_\ell}
\]
be the corresponding incoming Helgason mode.
If
\[
    u_{\infty,\mathcal P,\xi_\ell}(\hat x)
    =
    u_{\infty,\mathcal P',\xi_\ell}(\hat x)
    \qquad
    \text{for all }\hat x\in\mathbb S^{n-1}
    \text{ and } \ell=1,\ldots,n+1,
\]
then
\[
    \mathcal P=\mathcal P' .
\]
\end{theorem}

\begin{remark}\label{rem:affine-independence}
The condition that
\[
    \xi_2-\xi_1,\ldots,\xi_{n+1}-\xi_1
\]
are linearly independent means that the boundary labels
\(\xi_1,\ldots,\xi_{n+1}\in\mathbb S^{n-1}\subset\mathbb R^n\)
are affinely independent.
Equivalently, they are not contained in any common affine hyperplane of
\(\mathbb R^n\).

By Definition~\ref{def:asymptotic-boundary}, \(\partial_\infty V\) denotes the ideal boundary of a totally geodesic hypersurface \(V\), namely its set of limiting points at conformal infinity.
Together with Remark~\ref{rem:asymptotic-boundary}, the affine independence assumption is equivalent to
\[
    \{\xi_1,\ldots,\xi_{n+1}\}
    \not\subset
    \partial_\infty V
\]
for every totally geodesic hypersurface \(V\subset\mathbb B^n\).
\end{remark}

\begin{remark}\label{rem:hard-sharp-general-class}
The number \(n+1\) in Theorem~\ref{thm:mainip2} is optimal for the general \(\mathsf N\)-type totally geodesic defect class.
Indeed, if only \(m\leq n\) boundary labels are used, then these labels are contained in \(\partial_\infty V\) for some totally geodesic hypersurface \(V\subset\mathbb B^n\), by Remark~\ref{rem:affine-independence}.
For each such label, the corresponding incoming Helgason mode is even with respect to the reflection across \(V\), and hence
\[
    \partial_{\nu_{\mathbb H}} e_{2\lambda_0,\xi_\ell}=0
    \qquad\text{on }V .
\]
{
Choose a hyperbolic isometry taking \(V\) to
\(V_0=\{x_n=0\}\).  The transformed label \(\eta\) satisfies \(\eta_n=0\),
and the explicit Helgason kernel is therefore even in \(x_n\).  Its
hyperbolic normal derivative vanishes on \(V_0\), and
\eqref{eq:normal-covariance-isometry} transports this condition back to \(V\).
}
Thus, if \(\Gamma\subset V\) is an \(\mathsf N\)-type hypersurface-supported component, then \(u=e_{2\lambda_0,\xi_\ell}\) already satisfies the homogeneous hyperbolic Neumann trace condition on \(\Gamma\).
The outgoing correction is therefore identically zero for these boundary labels.
Consequently, any two distinct \(\mathsf N\)-type hypersurface-supported defects contained in \(V\) produce the same far-field patterns at conformal infinity for the corresponding \(m\) incoming Helgason modes.
Hence uniqueness in the general \(\mathsf N\)-type defect class fails for every \(m\leq n\).
Therefore \(n+1\) incoming Helgason modes are necessary and sufficient for uniqueness in this class, and Theorem~\ref{thm:mainip2} is sharp at this level of generality.
\end{remark}

\subsubsection{Stable determination of totally geodesic defects}

\medskip
We next state the stability results for the far-field map at conformal infinity \eqref{eq:hip1}.
Throughout this subsection, every far-field pattern is understood at conformal infinity.

For \(0<h\leq h_0\), let \(\mathcal A_{\mathbb H}^h\) denote the \(\mathsf D\)-type admissible class, and let \(\mathcal B_{\mathbb H}^h\) denote the \(\mathsf N\)-type admissible class.
Here \(h\) controls the minimal size of the totally geodesic faces in the interaction surface.
The precise admissibility conditions are given in Definition~\ref{def:admissible-tg-defects} for \(\mathcal A_{\mathbb H}^h\) and in Definition~\ref{def:hard-admissible-tg-defects} for \(\mathcal B_{\mathbb H}^h\).
The following theorems give quantitative estimates of the form \eqref{eq:intro-stability}: the hyperbolic Hausdorff distance between two admissible defects is controlled by the discrepancy of their far-field patterns.

In the proof, it is convenient to use the modified boundary distance \(d(\mathcal P,\mathcal P')\), adapted to the hyperbolic-reflection argument.
This distance is compared with the hyperbolic Hausdorff distance in Proposition~\ref{prop:distance-equivalence-hyperbolic}, yielding
\[
    d_{\mathcal H}^{\mathbb H}(\mathcal P,\mathcal P')
    \leq
    \delta^{-1}\!\left(d(\mathcal P,\mathcal P')\right),
\]
where \(\delta\) is the exterior connectedness modulus of the admissible class.

\medskip
The first stability theorem concerns the \(\mathsf D\)-type case.
It shows that, within the admissible class, a \(\mathsf D\)-type totally geodesic defect is stably determined by one far-field pattern generated by a single incoming Helgason mode with fixed spectral parameter and fixed boundary label.

\begin{theorem}\label{thm:main-soft-farfield}
Let \(0<h\leq h_0\), and let
\(
    \mathcal P,\mathcal P'\in\mathcal A_{\mathbb H}^h
\)
be two \(\mathsf D\)-type admissible totally geodesic defects.
Fix \(\lambda_0>0\) and a boundary label
\[
    \xi\in\partial_\infty\mathbb H^n=\mathbb S^{n-1}.
\]
Let \(u_{\infty,\mathcal P,\xi}\) and \(u_{\infty,\mathcal P',\xi}\) be the corresponding far-field patterns generated by the same incoming Helgason mode
\[
    u^i_\xi=e_{2\lambda_0,\xi}.
\]
Set
\[
    \varepsilon
    :=
    \bigl\|
        u_{\infty,\mathcal P,\xi}
        -
        u_{\infty,\mathcal P',\xi}
    \bigr\|_{L^2(\mathbb S^{n-1})}.
\]
There exists \(\varepsilon_{\mathsf D}(h)>0\), depending on \(h\) and on the a priori data, such that if
\[
    0<\varepsilon<\varepsilon_{\mathsf D}(h),
\]
then
\begin{equation}\label{eq:main-soft-modified-distance-stability}
    d(\mathcal P,\mathcal P')
    \leq
    C\bigl(\log\log(1/\varepsilon)\bigr)^{-\alpha},
\end{equation}
and consequently
\begin{equation}\label{eq:main-soft-final-stability}
    d_{\mathcal H}^{\mathbb H}(\mathcal P,\mathcal P')
    \leq
    \delta^{-1}
    \left(
        C\bigl(\log\log(1/\varepsilon)\bigr)^{-\alpha}
    \right),
\end{equation}
where \(\delta\) is the nondecreasing exterior connectedness modulus fixed in the admissible class.
Here, \(C>0\) and \(\alpha>0\) depend only on \(n\), \(\lambda_0\), the \(\mathsf D\)-type a priori data given in Definition~\ref{def:admissible-tg-defects}, and the boundary label \(\xi\).
They are independent of \(\mathcal P,\mathcal P'\), \(h\), and \(\varepsilon\).
\end{theorem}

\medskip
For the \(\mathsf N\)-type case, we impose a quantitative nondegeneracy condition on the boundary labels.

\begin{definition}\label{def:hard-directions}
Let
\[
    \Xi=(\xi_1,\ldots,\xi_m)
    \in
    \bigl(\partial_\infty\mathbb H^n\bigr)^m=\bigl(\mathbb S^{n-1}\bigr)^m.
\]
For a totally geodesic hypersurface \(V\subset\mathbb B^n\), let \(\partial_\infty V\) denote its ideal boundary in the sense of Definition~\ref{def:asymptotic-boundary}.
We say that \(\Xi\) is admissible if
\begin{equation*}
    a_0(\Xi)
    :=
    \inf_V
    \max_{\ell=1,\ldots,m}
    \operatorname{dist}_{\mathbb S^{n-1}}(\xi_\ell,\partial_\infty V)
    >0,
\end{equation*}
where the infimum is taken over all totally geodesic hypersurfaces \(V\subset\mathbb B^n\).
Here \(\operatorname{dist}_{\mathbb S^{n-1}}\) is the geodesic distance on \(\mathbb S^{n-1}\).
We call \(a_0(\Xi)\) the label nondegeneracy constant of \(\Xi\).
\end{definition}

\begin{remark}\label{rem:hard-directions-section-six}
The condition \(a_0(\Xi)>0\) is the quantitative form of the requirement that the boundary labels are not all contained in \(\partial_\infty V\) for any totally geodesic hypersurface \(V\subset\mathbb B^n\).
Equivalently, for every such \(V\), at least one label in \(\Xi\) stays a uniformly positive distance away from \(\partial_\infty V\).
In the \(\mathsf N\)-type stability theorem, \(a_0(\Xi)\) is part of the a priori data of the label configuration.
In particular, by Remark~\ref{rem:affine-independence}, any \(n+1\) affinely independent points \(\xi_1,\ldots,\xi_{n+1}\in\mathbb S^{n-1}\subset\mathbb R^n\) form an admissible family.
\end{remark}

\medskip
The second stability theorem concerns the \(\mathsf N\)-type case.
It gives a stability result from an admissible finite family of boundary labels, at fixed spectral parameter, for \(\mathsf N\)-type totally geodesic defects.

\begin{theorem}\label{thm:main-hard-farfield-simple}
Let \(0<h\leq h_0\), and let
\(
    \mathcal P,\mathcal P'\in\mathcal B_{\mathbb H}^h
\)
be two \(\mathsf N\)-type admissible totally geodesic defects.
Fix \(\lambda_0>0\).
Let
\[
    \Xi=(\xi_1,\ldots,\xi_m)
    \in
    \bigl(\partial_\infty\mathbb H^n\bigr)^m=\bigl(\mathbb S^{n-1}\bigr)^m
\]
be an admissible family of boundary labels in the sense of Definition~\ref{def:hard-directions}.
For each \(\ell=1,\ldots,m\), let
\(u_{\infty,\mathcal P,\xi_\ell}\) and
\(u_{\infty,\mathcal P',\xi_\ell}\) be the corresponding far-field patterns generated by the same incoming Helgason mode
\[
    u^i_{\xi_\ell}=e_{2\lambda_0,\xi_\ell}.
\]
Set
\begin{equation*}
    \varepsilon
    :=
    \max_{\ell=1,\ldots,m}
    \bigl\|
        u_{\infty,\mathcal P,\xi_\ell}
        -
        u_{\infty,\mathcal P',\xi_\ell}
    \bigr\|_{L^2(\mathbb S^{n-1})}.
\end{equation*}
There exists \(\varepsilon_{\mathsf N}(h)>0\), depending on \(h\), on the \(\mathsf N\)-type a priori data, and on the admissible label family \(\Xi\), such that if
\[
    0<\varepsilon<\varepsilon_{\mathsf N}(h),
\]
then
\begin{equation}\label{eq:main-hard-modified-distance-stability}
    d(\mathcal P,\mathcal P')
    \leq
    C\bigl(\log\log(1/\varepsilon)\bigr)^{-\alpha},
\end{equation}
and consequently
\begin{equation}\label{eq:main-hard-final-stability}
    d_{\mathcal H}^{\mathbb H}(\mathcal P,\mathcal P')
    \leq
    \delta^{-1}
    \left(
        C\bigl(\log\log(1/\varepsilon)\bigr)^{-\alpha}
    \right),
\end{equation}
where \(\delta\) is the nondecreasing exterior connectedness modulus fixed in the admissible class.
Here, \(C>0\) and \(\alpha>0\) depend only on \(n\), \(\lambda_0\), the \(\mathsf N\)-type a priori data given in Definition~\ref{def:hard-admissible-tg-defects}, and \(\Xi\) through \(m\) and \(a_0(\Xi)\).
They are independent of \(\mathcal P,\mathcal P'\), \(h\), and \(\varepsilon\).
\end{theorem}

\begin{remark}\label{rem:hard-stability-minimal-measurements}
The \(\mathsf N\)-type stability theorem is optimal with respect to the number of boundary labels.
Indeed, by Remark~\ref{rem:hard-sharp-general-class}, uniqueness in the general \(\mathsf N\)-type totally geodesic defect class fails for every family of \(m\leq n\) boundary labels.
Consequently, no stability estimate of the form \eqref{eq:intro-stability} can hold uniformly with \(m\leq n\) boundary labels, since such an estimate would imply uniqueness when the far-field discrepancy is zero.
Thus \(n+1\) affinely independent boundary labels give the minimal label configuration for uniform stability in the general \(\mathsf N\)-type class.
For larger admissible families \(\Xi\), Theorem~\ref{thm:main-hard-farfield-simple} remains valid, with constants depending on the quantitative nondegeneracy \(a_0(\Xi)\).
\end{remark}

\begin{corollary}\label{cor:bulk-defects-double-log}
If, in Theorems~\ref{thm:main-soft-farfield} and
\ref{thm:main-hard-farfield-simple}, the admissible class is restricted to defects with only bulk components and the exterior connectedness modulus \(\delta\) is linear, then, for \(\varepsilon>0\) sufficiently small,
\[
    d_{\mathcal H}^{\mathbb H}(\mathcal P,\mathcal P')
    \leq
    C\bigl(\log\log(1/\varepsilon)\bigr)^{-\alpha}.
\]
\end{corollary}

\begin{remark}\label{rem:double-log-sharpness}
The stability estimates in Theorems~\ref{thm:main-soft-farfield}--\ref{thm:main-hard-farfield-simple} and Corollary~\ref{cor:bulk-defects-double-log} have double-logarithmic dependence for the present far-field problem at conformal infinity.
The first logarithmic loss comes from quantitative continuation of the far-field data from conformal infinity into the hyperbolic exterior.
The second logarithmic loss comes from propagating this smallness to the interaction surface and converting it into geometric control through hyperbolic reflection.
Thus the double-logarithmic modulus comes from this two-step process of recovering a compact hyperbolic defect from finite far-field data at conformal infinity.
In addition, Remark~\ref{rem:hard-stability-minimal-measurements} shows that the number of boundary labels in the \(\mathsf N\)-type stability theorem is minimal.
\end{remark}

\subsection{Discussion and main contributions}

We emphasize four main contributions of the paper.
They are organized around the role of conformal infinity in the data and the ideal-boundary geometry of totally geodesic hypersurfaces.

\begin{itemize}
\item Motivated by the holographic boundary-to-bulk viewpoint at conformal infinity, we introduce and study a novel class of finite-label inverse scattering problems for compact impenetrable defects in hyperbolic space at one fixed spectral parameter.
The prescribed inputs are a single boundary label, or at most finitely many boundary labels, at \(\partial_\infty\mathbb H^n\).
Each label selects an incoming Helgason mode, and the data are the corresponding far-field patterns at conformal infinity.

   \item We identify the role of ideal-boundary geometry in determining visibility for totally geodesic defects.
The admissible defects have interaction surfaces consisting of finitely many totally geodesic faces.
The class allows both bulk components and hypersurface-supported components.
The incidence relation between the prescribed boundary labels and the ideal boundaries of supporting totally geodesic hypersurfaces determines which components can be detected.

  \item We prove uniqueness results with the optimal number of prescribed boundary labels, equivalently incoming Helgason modes.
In the \(\mathsf D\)-type case, one label is sufficient.
In the general \(\mathsf N\)-type case, \(n+1\) affinely independent labels are necessary and sufficient.
The necessity follows from the invisibility of \(\mathsf N\)-type hypersurface-supported components under degenerate label configurations.

    \item We prove {sharp} quantitative stability estimates in admissible classes.
The hyperbolic Hausdorff distance between two defects is controlled by a double-logarithmic modulus of the discrepancy of their far-field patterns at conformal infinity.
The double-logarithmic modulus reflects the two logarithmic losses in the continuation argument, and the \(\mathsf N\)-type estimate is obtained under the minimal label requirement.
\end{itemize}

The qualitative and quantitative parts of the paper are developed under different levels of admissibility.
For uniqueness, the assumptions are structural: the defect is impenetrable, its exterior is connected, and its interaction surface consists of finitely many totally geodesic faces.
For stability, these assumptions are strengthened to quantitative admissible classes.
The quantitative data include localization in the hyperbolic metric, lower size control of the totally geodesic faces, controlled incidence of faces, and uniform exterior connectedness.
These conditions provide the compactness and uniform constants needed to pass from unique determination to stability for the hyperbolic Hausdorff distance.

The totally geodesic structure enters through the geometry of the problem.
It links the hyperbolic interior to conformal infinity.
The ideal boundary of a supporting totally geodesic hypersurface determines whether a boundary label is transverse or degenerate with respect to a face.
This is why the \(\mathsf D\)-type and \(\mathsf N\)-type cases have different data requirements.
In the \(\mathsf D\)-type case, a single label is enough to force a contradiction through reflection.
In the \(\mathsf N\)-type case, degenerate labels may fail to see hypersurface-supported components, and the admissibility condition on the labels rules out this obstruction.

Another point is the uniform control of the direct problem over varying defects.
For a fixed defect, the radiation condition, Rellich theorem, and far-field expansion are supplied by the hyperbolic Sommerfeld--Rellich framework of Chen--Liu~\cite{ChenLiu2023}.
For the inverse stability problem, one needs estimates that are uniform over the whole admissible class.
We obtain such estimates by combining the hyperbolic radiation theory at conformal infinity with compactness on bounded hyperbolic regions, Mosco convergence of Sobolev spaces, and uniform Sobolev inequalities on varying exterior domains~\cite{mosco1969convergence,fornoni2023mosco}.
This uniform direct theory is what allows the far-field error at conformal infinity to be propagated into the exterior region and then converted into hyperbolic geometric control.

The paper contributes to the hyperbolic inverse scattering program in two ways.
It gives a far-field recovery theory at conformal infinity from a prescribed finite family of boundary labels, rather than an operator-level inverse theory.
It also shows that the data requirements are determined by ideal-boundary incidence, a geometric phenomenon specific to hyperbolic space.
The results therefore provide both a qualitative and a quantitative form of recovery from data at conformal infinity for compact totally geodesic defects.


\medskip
The rest of the paper is organized as follows.
Section~\ref{sec 2} collects the analytic preliminaries for hyperbolic scattering, including the Poincar\'e ball model, the Helgason--Fourier transform, outgoing Green functions, the hyperbolic radiation condition, far-field patterns, and the hyperbolic Rellich theorem.
Section~\ref{sec:reflection-principles} introduces totally geodesic hypersurfaces and totally geodesic defects, and proves reflection principles for solutions of \(\mathcal L_{\lambda_0}u=0\).
Section~\ref{sec:unique-determination} proves the uniqueness results for \(\mathsf D\)-type and \(\mathsf N\)-type totally geodesic defects.
Section~\ref{sec:quant-prelim} introduces the quantitative admissible classes and establishes uniform estimates for the corresponding direct problems.
Section~\ref{sec:quantitative-rellich} develops the quantitative ingredients for stability, including a far-field-to-near-field estimate at conformal infinity, propagation of smallness through hyperbolic chains, and quantitative reflection across totally geodesic hypersurfaces.
Finally, Section~\ref{sec:quantitative-stability} proves the stability estimates for \(\mathsf N\)-type and \(\mathsf D\)-type totally geodesic defects.

\section{Analytic preliminaries for hyperbolic scattering}\label{sec 2}

This section introduces the analytic framework for hyperbolic scattering used throughout the paper.
We work in the Poincar\'e ball model \(\mathbb B^n\) of \(\mathbb H^n\), fix the geometric notation and M\"obius isometries, and present the basic Fourier, Green function, radiation, far-field, and Rellich tools, as well as {the hyperbolic normal derivative and the associated Green identities}.

\subsection{Hyperbolic space and M\"obius transformations}\label{sect:2.1}

We work in the \emph{Poincar\'e ball} model of hyperbolic space $\mathbb{H}^n$, namely the unit ball
\[
    \mathbb B^n=\{x\in\mathbb R^n: |x|<1\},
\]
endowed with the Poincar\'e metric
\begin{equation}\label{eq:hyper metric}
    g_{\mathbb H}=\left(\frac{2}{1-|x|^2}\right)^2g_e,
\end{equation}
where \(g_e\) denotes the Euclidean metric.
The associated volume element is
\begin{equation}\label{eq:hyvolume}
	dV_{\mathbb H}=\left(\frac{2}{1-|x|^2}\right)^n dx.
\end{equation}
The hyperbolic distance from the origin is
\[
    \rho(x)=\log\frac{1+|x|}{1-|x|}.
\]
For \(R>0\), we denote by \(B_{\mathbb H}(x,R)\) the hyperbolic ball centered at \(x\) with hyperbolic radius \(R\).
When the center is omitted, we write \(B_{\mathbb H}(R):=B_{\mathbb H}(0,R)\).

In Euclidean coordinates, the Laplace--Beltrami operator and the hyperbolic gradient are given by
\begin{equation}\label{eq:Laplace}
    \Delta_{\mathbb H}
    =
    \frac{1-|x|^2}{4}
    \left(
        (1-|x|^2)\Delta_{\mathbb R^n}
        +
        2(n-2)\sum_{i=1}^n x_i\frac{\partial}{\partial x_i}
    \right),
    \qquad
    \nabla_{\mathbb H}
    =
    \left(\frac{1-|x|^2}{2}\right)^2\nabla_{\mathbb R^n}.
\end{equation}
We shall also use hyperbolic polar coordinates, given by \(x=\tanh(\rho/2)\xi\), where \(\rho\in(0,+\infty)\) and \(\xi\in\mathbb S^{n-1}\).
In these coordinates, the metric takes the form
\[
    g_{\mathbb H}=d\rho^2+\sinh^2\rho\,g_{\mathbb S^{n-1}},
\]
where \(g_{\mathbb S^{n-1}}\) is the standard metric on \(\mathbb S^{n-1}\).
Moreover,
\begin{equation*}
    \Delta_{\mathbb H}
    =
    \frac{\partial^2}{\partial \rho^2}
    +(n-1)\coth\rho\,\frac{\partial}{\partial \rho}
    +\frac{1}{\sinh^2\rho}\Delta_{\mathbb S^{n-1}},
    \qquad
    \nabla_{\mathbb H}
    =
    \left(
        \frac{\partial}{\partial \rho},
        \frac{1}{\sinh\rho}\nabla_{\mathbb S^{n-1}}
    \right).
\end{equation*}
Here the expression for \(\nabla_{\mathbb H}\) in polar coordinates is understood with respect to the orthonormal polar frame.
Whenever the integrals are well defined, we have
\begin{equation*}
\begin{split}
    \int_{\mathbb B^n} f(x)\,dV_{\mathbb H}
    &=
    \int_0^1\int_{\mathbb S^{n-1}}
    f(r\xi)r^{n-1}
    \left(\frac{2}{1-r^2}\right)^n
    d\sigma(\xi)\,dr                                     \\
    &=
    \int_0^{+\infty}\int_{\mathbb S^{n-1}}
    f\left(\tanh\left(\frac{\rho}{2}\right)\xi\right)
    (\sinh\rho)^{n-1}
    d\sigma(\xi)\,d\rho .
\end{split}
\end{equation*}

For each \(a\in \mathbb B^n\), we define the \emph{M\"obius transformation} \(T_a\) (see \cite{Ahlfors}) by
\[
    T_a(x)
    =
    \frac{|x-a|^2a-(1-|a|^2)(x-a)}
    {1-2x\cdot a+|x|^2|a|^2},
\]
where \(x\cdot a\) denotes the Euclidean scalar product in \(\mathbb R^n\).
A direct calculation gives
\[
    T_a(0)=a,
    \qquad
    T_a(a)=0,
    \qquad
    T_a\circ T_a=\mathrm{Id}.
\]
In particular, \(T_a^{-1}=T_a\).

The Poincar\'e metric is invariant under \(T_a\), and hence \(T_a\) is a hyperbolic isometry.
Consequently, the volume element \(dV_{\mathbb H}\) is invariant under \(T_a\), and for any \(\varphi\in L^1(\mathbb B^n)\),
\[
    \int_{\mathbb B^n}|\varphi\circ T_a|\,dV_{\mathbb H}
    =
    \int_{\mathbb B^n}|\varphi|\,dV_{\mathbb H}.
\]
The Laplace--Beltrami operator commutes with pullback by \(T_a\):
\[
    \Delta_{\mathbb H}(\phi\circ T_a)
    =
    (\Delta_{\mathbb H}\phi)\circ T_a,
    \qquad \phi\in C_c^\infty(\mathbb B^n).
\]
Equivalently,
\[
    \int_{\mathbb B^n}
    -\Delta_{\mathbb H}(\phi\circ T_a)(\phi\circ T_a)\,dV_{\mathbb H}
    =
    \int_{\mathbb B^n}
    -\Delta_{\mathbb H}\phi\cdot \phi\,dV_{\mathbb H}.
\]

The geodesic distance between two points \(x,y\in\mathbb B^n\) can be written as
\begin{equation}\label{eq:distance}
	\rho(x,y)
    =
    \rho(0,T_x(y))
    =
    \log\frac{1+|T_x(y)|}{1-|T_x(y)|}.
\end{equation}
We use the following notation for the hyperbolic distance from a point to a set.
For \(x\in\mathbb B^n\) and a nonempty set \(E\subset\mathbb B^n\), we set
\begin{equation}\label{eq:distH}
	   \rho(x,E)
    :=
    \inf_{y\in E}\rho(x,y).
\end{equation}
Using the M\"obius transformations, we define the convolution of measurable functions \(f\) and \(g\) on \(\mathbb B^n\) (see \cite{liu}) by
\[
    (f\ast g)(x)
    =
    \int_{\mathbb B^n}
    f(y)g(T_x(y))\,dV_{\mathbb H}(y),
\]
whenever the integral is well defined.

\subsection{The Helgason--Fourier transform on hyperbolic space}

We record the basic facts about the Helgason--Fourier transform that will be used below.
For further background on Helgason--Fourier analysis on hyperbolic spaces, we refer to \cite{Helgason1984,Helgason2008,LuYangQ1,LuYangQ2,LuYangQ3}.

Set
\begin{equation}\label{eq:e}
    e_{\lambda,\xi}(x)
    =
    \left(
        \frac{\sqrt{1-|x|^2}}{|x-\xi|}
    \right)^{n-1+\mathrm i\lambda},
    \qquad
    x\in \mathbb B^n,\quad
    \lambda\in\mathbb R,\quad
    \xi\in\mathbb S^{n-1}.
\end{equation}
Then \(e_{\lambda,\xi}\) is a generalized eigenfunction of \(-\Delta_{\mathbb H}\), namely
\[
    -\Delta_{\mathbb H}e_{\lambda,\xi}
    =
    \frac{(n-1)^2+\lambda^2}{4}e_{\lambda,\xi}.
\]
Moreover, \(e_{-\lambda,\xi}=\overline{e_{\lambda,\xi}}\).
In particular, the incoming Helgason mode in \eqref{eq:incoming-helgason-mode} is \(u^i_\xi=e_{2\lambda_0,\xi}\).

For \(f\in C_c^\infty(\mathbb B^n)\), its Helgason--Fourier transform is defined by
\begin{equation}\label{eq:FT}
	  \widehat f(\lambda,\xi)
    =
    \int_{\mathbb B^n}
    f(x)e_{-\lambda,\xi}(x)\,dV_{\mathbb H}(x).
\end{equation}
The inversion formula is
\begin{equation}\label{eq:IFT}
	   f(x)
    =
    D_n
    \int_{-\infty}^{\infty}
    \int_{\mathbb S^{n-1}}
    \widehat f(\lambda,\xi)e_{\lambda,\xi}(x)
    |c(\lambda)|^{-2}\,d\sigma(\xi)\,d\lambda,
\end{equation}
where
\(
    D_n
    =
    \left(2^{3-n}\pi|\mathbb S^{n-1}|\right)^{-1},
\)
and \(c(\lambda)\) is the Harish--Chandra function (see \cite{liu}).

The transform extends to \(L^2(\mathbb B^n,dV_{\mathbb H})\) in the Plancherel sense.
For \(g,h\in L^2(\mathbb B^n,dV_{\mathbb H})\), one has
\[
    \int_{\mathbb B^n}
    g(x)\overline{h(x)}\,dV_{\mathbb H}(x)
    =
    D_n
    \int_{-\infty}^{\infty}
    \int_{\mathbb S^{n-1}}
    \widehat g(\lambda,\xi)\overline{\widehat h(\lambda,\xi)}
    |c(\lambda)|^{-2}\,d\sigma(\xi)\,d\lambda .
\]
In particular,
\[
    \|g\|_{L^2(\mathbb B^n,dV_{\mathbb H})}^2
    =
    D_n
    \int_{-\infty}^{\infty}
    \int_{\mathbb S^{n-1}}
    |\widehat g(\lambda,\xi)|^2
    |c(\lambda)|^{-2}\,d\sigma(\xi)\,d\lambda .
\]

\subsection{Heat kernel and Green function of the shifted Laplacian}

Consider the heat equation
\begin{equation}\label{eq:heat}
\begin{cases}
    \partial_t u(x,t)-\Delta_{\mathbb H}u(x,t)=0,
    & (x,t)\in \mathbb B^n\times \mathbb R^+, \\
    u(x,0)=f(x),
    & f\in C_c^\infty(\mathbb B^n).
\end{cases}
\end{equation}
Applying the Helgason--Fourier transform \eqref{eq:FT} to \eqref{eq:heat} gives
\[
    \partial_t\widehat u(\lambda,\xi,t)
    +
    \frac{(n-1)^2+\lambda^2}{4}
    \widehat u(\lambda,\xi,t)
    =
    0,
    \qquad
    \widehat u(\lambda,\xi,0)=\widehat f(\lambda,\xi).
\]
Solving this ODE, we obtain
\[
    \widehat u(\lambda,\xi,t)
    =
    e^{-\frac{(n-1)^2+\lambda^2}{4}t}
    \widehat f(\lambda,\xi).
\]
Using the inversion formula \eqref{eq:IFT}, we obtain
\[
    u(x,t)
    =
    e^{t\Delta_{\mathbb H}}f(x)
    =
    \int_{\mathbb B^n}
    P_t(T_x(y))f(y)\,dV_{\mathbb H}(y).
\]
Here, \(P_t\) is the \emph{heat kernel} on hyperbolic space and admits the spectral representation
\[
    P_t(x)
    =
    D_n
    \int_{-\infty}^{\infty}
    \int_{\mathbb S^{n-1}}
    e^{-\frac{(n-1)^2+\lambda^2}{4}t}
    e_{\lambda,\xi}(x)
    |c(\lambda)|^{-2}
    d\sigma(\xi)\,d\lambda.
\]
The kernel \(P_t\) is radial, and hence we write \(P_t(x)=P_t(\rho(x))\).
The explicit formula for \(P_t\) is
\[
P_t(\rho)
=
(2\pi)^{-\frac{n+1}{2}}t^{-\frac12}
e^{-\frac{(n-1)^2}{4}t}
\int_{\rho}^{+\infty}
\frac{\sinh r}{\sqrt{\cosh r-\cosh \rho}}
\left(
    -\frac{1}{\sinh r}\frac{\partial}{\partial r}
\right)^m
e^{-\frac{r^2}{4t}}\,dr
\]
when \(n=2m\), and
\[
P_t(\rho)
=
2^{-m-1}\pi^{-m-\frac12}t^{-\frac12}
e^{-\frac{(n-1)^2}{4}t}
\left(
    -\frac{1}{\sinh \rho}\frac{\partial}{\partial \rho}
\right)^m
e^{-\frac{\rho^2}{4t}}
\]
when \(n=2m+1\) (see \cite{Davis, Gri}).

Since \(\operatorname{Spec}(-\Delta_{\mathbb H})=[(n-1)^2/4,+\infty)\), the Mellin formula gives the Bessel--Green--Riesz representation of the shifted resolvent:
\begin{equation}\label{eq:Mellin}
\begin{split}
\left(-\Delta_{\mathbb H}-\frac{(n-1)^2}{4}+k^2\right)^{-1}
&=
\int_{0}^{+\infty}
e^{\left(\frac{(n-1)^2}{4}-k^2\right)t}
e^{t\Delta_{\mathbb H}}\,dt .
\end{split}
\end{equation}
The corresponding integral kernel of \eqref{eq:Mellin} is radial by the homogeneity and isotropy of hyperbolic space.
We denote this \emph{Green function} by \(G_k(\rho)\), where \(\rho=\rho(x,y)\).
For \(n\geq 3\), \(G_k\) is given by (see \cite{li, LuYangQ2, LuYangQ3, Mat}):
\[
G_k(\rho)
=
(2\pi)^{-\frac{n}{2}}
(\sinh \rho)^{-\frac{n-2}{2}}
e^{-\frac{(n-2)\pi}{2}\mathrm i}
Q_{k-\frac{1}{2}}^{\frac{n-2}{2}}(\cosh \rho),
\]
where \(Q_{k-\frac{1}{2}}^{\frac{n-2}{2}}\) is the Legendre function of the second kind.
Moreover, it satisfies (see \cite{Erd})
\[
e^{-\frac{(n-2)\pi}{2}\mathrm i}
Q_{k-\frac{1}{2}}^{\frac{n-2}{2}}(\cosh \rho)
=
\frac{\Gamma(\frac{n-1}{2}+k)}
{2^{k+\frac{1}{2}}\Gamma(k+\frac{1}{2})\sinh^{\frac{n-2}{2}}\rho}
\int_{0}^{\pi}
(\cosh \rho+\cos t)^{\frac{n-3}{2}-k}
(\sin t)^{2k}\,dt .
\]
Hence, one has
\begin{equation}\label{eq:shifted-green-kernel}
	G_k(\rho)
=
\frac{A_{n,k}}{(\sinh \rho)^{n-2}}
\int_{0}^{\pi}
(\cosh \rho+\cos t)^{\frac{n-3}{2}-k}
(\sin t)^{2k}\,dt ,
\end{equation}
where
\[
A_{n,k}
=
(2\pi)^{-\frac{n}{2}}
\frac{\Gamma(\frac{n-1}{2}+k)}
{2^{k+\frac{1}{2}}\Gamma(k+\frac{1}{2})}.
\]

\subsection{Green function of the shifted Helmholtz operator on hyperbolic space}

Let \(\mu>0\). We define the outgoing and incoming Green functions of the shifted Helmholtz operator
\[
    -\Delta_{\mathbb H}
    -
    \frac{(n-1)^2}{4}
    -
    \mu^2
\]
by the limiting resolvents
\[
\begin{split}
G_{-\mu\mathrm i}
&:=
\lim_{\epsilon\to0^+}
\left(
-\Delta_{\mathbb H}
-\frac{(n-1)^2}{4}
-(\mu+\epsilon\mathrm i)^2
\right)^{-1},                                      \\
G_{\mu\mathrm i}
&:=
\lim_{\epsilon\to0^+}
\left(
-\Delta_{\mathbb H}
-\frac{(n-1)^2}{4}
-(\mu-\epsilon\mathrm i)^2
\right)^{-1}.
\end{split}
\]
{
For \(n\geq3\), these Green functions are obtained from \(G_k\) in
\eqref{eq:shifted-green-kernel} by analytic continuation to
\(k=-\mu\mathrm i\) and \(k=\mu\mathrm i\), respectively.
}

For \(n\geq 3\), writing \(\rho=\rho(x)\), the outgoing Green function with pole at the origin is
\begin{equation}\label{eq:Green-}
\begin{split}
G_{-\mu \mathrm i}(\rho)
&=
A_{n,-\mu \mathrm i}
\frac{
    (\cosh \rho)^{\frac{n-3}{2}+\mu \mathrm i}
}{
    (\sinh \rho)^{n-2}
}                                                  \\
&\quad \times
\int_{0}^{\pi}
\left(
    1+\frac{\cos t}{\cosh \rho}
\right)^{\frac{n-3}{2}+\mu \mathrm i}
(\sin t)^{-2\mu \mathrm i}\,dt .
\end{split}
\end{equation}
The incoming Green function is
\begin{equation}\label{eq:Green+}
\begin{split}
G_{\mu \mathrm i}(\rho)
&=
A_{n,\mu \mathrm i}
\frac{
    (\cosh \rho)^{\frac{n-3}{2}-\mu \mathrm i}
}{
    (\sinh \rho)^{n-2}
}                                                  \\
&\quad \times
\int_{0}^{\pi}
\left(
    1+\frac{\cos t}{\cosh \rho}
\right)^{\frac{n-3}{2}-\mu \mathrm i}
(\sin t)^{2\mu \mathrm i}\,dt .
\end{split}
\end{equation}
Here, \(A_{n,\pm\mu\mathrm i}\) is obtained from \(A_{n,k}\) in
\eqref{eq:shifted-green-kernel} by replacing \(k\) with \(\pm\mu\mathrm i\).
{
For \(n=2\), the same limiting absorption gives (see \cite{Mat,DLMF})
\begin{equation}\label{eq:Green-two-dimensional}
    G_{\pm\mu\mathrm i}(\rho)
    =
    \frac{1}{2\pi}
    Q_{-\frac12\pm\mu\mathrm i}(\cosh\rho).
\end{equation}
The large-\(\rho\) expansion of \(Q\) yields the asymptotic and radiation
formulas below and, for \(j=0,1,2\),
\[
    \left|\partial_\rho^jG_{\pm\mu\mathrm i}(\rho)\right|
    \leq Ce^{-\rho/2},
    \qquad \rho\geq1.
\]
}

By M\"obius invariance, the Green functions with pole at \(y\in\mathbb B^n\)
are obtained by replacing \(\rho(x)\) with \(\rho(x,y)\).
Thus \(G_{-\mu\mathrm i}(\rho(x,y))\) and \(G_{\mu\mathrm i}(\rho(x,y))\) are Green functions of the shifted Helmholtz operator on \(\mathbb H^n\), with singularity at \(y\).
In the scattering problem below, \(G_{-\mu\mathrm i}\) is the outgoing Green function.

\subsection{Radiation condition, far-field pattern and hyperbolic Rellich theorem}\label{subsec:radition}

For fixed \(y\in \mathbb{B}^n\), write \(x=\tanh(\rho/2)\hat{x}\), where \(\rho=\rho(x)\) and \(\hat{x}\in \mathbb S^{n-1}\).
Then the Green function \(G_{-\mu \mathrm i}(\rho(x,y))\) may be regarded as a function \(G(\rho,\hat{x},y)\).
Set
\[
    I_\mu:=\int_{0}^{\pi}(\sin t)^{-2\mu \mathrm i}\,dt.
\]
As \(\rho\rightarrow +\infty\), the outgoing Green function \(G_{-\mu\mathrm i}\) admits the asymptotic expansion
\begin{equation}
\begin{split}
G_{-\mu \mathrm i}(\rho(x,y))
=G(\rho,\hat{x},y)
&=
\frac{A_{n,-\mu \mathrm i}}{2^{\frac{n-1}{2}-\mu\mathrm i}}
\frac{\left(\cosh(\frac{\rho}{2})\right)^{2\mu \mathrm i}}
{\cosh^{n-1}(\frac{\rho}{2})}
\frac{\left(\cosh(\frac{\rho(y)}{2})\right)^{2\mu \mathrm i}}
{\cosh^{n-1}(\frac{\rho(y)}{2})}                                      \\
&\quad \times
\left(1-2\hat{x}\cdot y+|y|^2\right)^{\mu\mathrm i-\frac{n-1}{2}}
I_\mu
+
O\left(\frac{1}{\sinh^{\frac{n+1}{2}}(\rho)}\right).
\end{split}
\end{equation}
Here the \(O\)-term is uniform for \(\hat{x}\in\mathbb S^{n-1}\) and for \(y\) in compact subsets of \(\mathbb B^n\).
A direct computation gives
\begin{equation*}
\begin{split}
\frac{\partial G}{\partial \rho}
=
\left(\mu \mathrm i-\frac{n-1}{2}\right)
\tanh\left(\frac{\rho}{2}\right)G
+
O\left(\frac{1}{\sinh^{\frac{n+1}{2}}(\rho)}\right)
\end{split}
\end{equation*}
as \(\rho(x)\rightarrow +\infty\).
Hence the Green function \(G_{-\mu \mathrm i}(\rho,\hat{x},y)\) satisfies
\begin{equation}\label{radiation}
\frac{\partial G}{\partial \rho}
-
\left(\mu \mathrm i-\frac{n-1}{2}\right)
\tanh\left(\frac{\rho}{2}\right)G
=
O\left(\frac{1}{\sinh^{\frac{n+1}{2}}(\rho)}\right)
\end{equation}
as \(\rho\rightarrow +\infty\).
We call the condition \eqref{radiation} the \emph{hyperbolic radiation condition}.

Consider the shifted Helmholtz equation on \(\mathbb H^n\) with source \(f\in C_c^\infty(\mathbb{B}^n)\)
\begin{equation}\label{eq:source-shifted-helmholtz}
\left(-\Delta_{\mathbb{H}}-\frac{(n-1)^2}{4}-\mu^2\right)u=f(x),
\qquad x\in \mathbb{B}^n.
\end{equation}
If \(u\) satisfies the hyperbolic radiation condition, then the outgoing solution is given by
\[
    u(x)
    =
    \int_{\mathbb{B}^n}
    G_{-\mu \mathrm i}(\rho(x,y))f(y)\,dV_{\mathbb{H}}(y).
\]

Define
\[
    C_{n,\mu}
    :=
    \frac{A_{n,-\mu \mathrm i}}{2^{\frac{n-1}{2}-\mu\mathrm i}}I_\mu .
\]
The \emph{far-field pattern} of \(u\) is
\begin{equation}\label{eq:ffp}
\begin{split}
u_{\infty}(\hat{x}, \mu)
&=
C_{n,\mu}
\int_{\mathbb{B}^n}
\frac{\left(\cosh(\frac{\rho(y)}{2})\right)^{2\mu \mathrm i}}
{\cosh^{n-1}(\frac{\rho(y)}{2})}
\left(1-2\hat{x}\cdot y+|y|^2\right)^{\mu \mathrm i-\frac{n-1}{2}}
f(y)\,dV_{\mathbb{H}}(y)\\
&=
C_{n,\mu}
\int_{\mathbb{B}^n}
e_{-2\mu, \hat{x}}(y)f(y)\,dV_{\mathbb{H}}(y)\\
&=
C_{n,\mu}\widehat{f}(2\mu,\hat{x}).
\end{split}
\end{equation}
Consequently, at the fixed spectral parameter \(\mu>0\), the far-field pattern determines the Helgason--Fourier transform of \(f\) on the spectral shell \(\lambda=2\mu\):
\[
    \widehat{f}(2\mu,\hat{x})
    =
    C_{n,\mu}^{-1}u_{\infty}(\hat{x},\mu).
\]

The radiation condition \eqref{radiation} singles out the outgoing solution, while the far-field pattern \eqref{eq:ffp} gives its leading term at conformal infinity.
We shall use the following hyperbolic Rellich theorem, proved in \cite{ChenLiu2023}.

\begin{lemma}\label{thm3}
Let \(B_{\mathbb{H}}(0,R_0)\) denote the hyperbolic ball centered at the origin with radius \(R_0\).
Assume that \(u\in C^{2}(\mathbb{B}^n\setminus B_{\mathbb{H}}(0,R_0))\) satisfies
\[
    -\Delta_{\mathbb{H}}u-\frac{(n-1)^2}{4}u-\mu^2 u=0
    \quad
    \text{in } \mathbb{B}^n\setminus B_{\mathbb{H}}(0,R_0).
\]
If
\[
    \lim_{R\rightarrow +\infty}
    \int_{\partial B_{\mathbb{H}}(0,R)}
    |u|^2\,d\sigma_{\mathbb{H}}
    =
    0,
\]
then \(u=0\) in \(\mathbb{B}^n\setminus B_{\mathbb{H}}(0,R_0)\).
\end{lemma}

{
\subsection{Hyperbolic normal derivatives and Green identities}\label{subsec:out nor}

Let \(\Omega\) be a bounded domain in \(\mathbb B^n\), and set
\[
    \alpha(x):=\frac{2}{1-|x|^2},
    \qquad
    \widetilde u:=\alpha^{\frac n2-1}u,
    \qquad
    \widetilde v:=\alpha^{\frac n2-1}v.
\]
Let \(\nu\) be the Euclidean outward unit normal to \(\partial\Omega\).
We derive the hyperbolic normal derivative along \(\partial\Omega\) by
integration by parts.  Using the relation between the conformal Laplacian
\(-\Delta_{\mathbb H}-n(n-2)/4\) and the Euclidean Laplacian gives
\begin{equation*}
\begin{split}
\int_\Omega(-\Delta_{\mathbb H}u)v\,dV_{\mathbb H}
&=
\int_\Omega
\left(-\Delta_{\mathbb H}u-\frac{n(n-2)}4u\right)v\,dV_{\mathbb H}
+\frac{n(n-2)}4\int_\Omega uv\,dV_{\mathbb H}
\\
&=
\int_\Omega(-\Delta_{\mathbb R^n}\widetilde u)\widetilde v\,dx
+\frac{n(n-2)}4\int_\Omega uv\,dV_{\mathbb H}
\\
&=
\int_\Omega(-\Delta_{\mathbb H}v)u\,dV_{\mathbb H}
+
\int_{\partial\Omega}
\left(
\frac{\partial\widetilde v}{\partial\nu}\widetilde u
-
\frac{\partial\widetilde u}{\partial\nu}\widetilde v
\right)d\sigma.
\end{split}
\end{equation*}
The boundary term satisfies
\[
\frac{\partial\widetilde v}{\partial\nu}\widetilde u
-
\frac{\partial\widetilde u}{\partial\nu}\widetilde v
=
\alpha^{n-2}
\left(
\frac{\partial v}{\partial\nu}u
-
\frac{\partial u}{\partial\nu}v
\right).
\]
If we define
\begin{equation}\label{eq:outnormal}
    \frac{\partial u}{\partial\nu_{\mathbb H}}
    :=
    \left\langle\nabla_{\mathbb H}u,\nu_{\mathbb H}\right\rangle_{\mathbb H}
    =
    \alpha^{-1}\frac{\partial u}{\partial\nu}
    =
    \frac{1-|x|^2}{2}\frac{\partial u}{\partial\nu},
\end{equation}
then, since \(d\sigma_{\mathbb H}=\alpha^{n-1}d\sigma\),
\[
\int_{\partial\Omega}
\left(
\frac{\partial\widetilde v}{\partial\nu}\widetilde u
-
\frac{\partial\widetilde u}{\partial\nu}\widetilde v
\right)d\sigma
=
\int_{\partial\Omega}
\left(
\frac{\partial v}{\partial\nu_{\mathbb H}}u
-
\frac{\partial u}{\partial\nu_{\mathbb H}}v
\right)d\sigma_{\mathbb H},
\]
which gives the second Green identity.

To verify the first Green identity, we compute
\begin{equation*}
\begin{split}
\int_\Omega(-\Delta_{\mathbb H}u)v\,dV_{\mathbb H}
&=
\int_\Omega
\left(-\Delta_{\mathbb H}u-\frac{n(n-2)}4u\right)v\,dV_{\mathbb H}
+\frac{n(n-2)}4\int_\Omega uv\,dV_{\mathbb H}
\\
&=
\int_\Omega\nabla\widetilde u\cdot\nabla\widetilde v\,dx
+\frac{n(n-2)}4\int_\Omega uv\,dV_{\mathbb H}
-
\int_{\partial\Omega}
\frac{\partial\widetilde u}{\partial\nu}\widetilde v\,d\sigma.
\end{split}
\end{equation*}
Expanding \(\widetilde u\) and \(\widetilde v\) gives
\[
\begin{split}
&\int_\Omega\nabla\widetilde u\cdot\nabla\widetilde v\,dx
+\frac{n(n-2)}4\int_\Omega uv\,dV_{\mathbb H}
\\
&=
\int_\Omega\alpha^{n-2}\nabla u\cdot\nabla v\,dx
+\frac{n-2}{2}\int_\Omega
\alpha^{n-3}\nabla\alpha\cdot\nabla(uv)\,dx
\\
&\quad+
\frac{(n-2)^2}{4}\int_\Omega
\alpha^{n-4}|\nabla\alpha|^2uv\,dx
+\frac{n(n-2)}4\int_\Omega\alpha^nuv\,dx.
\end{split}
\]
Since \(\alpha=2/(1-|x|^2)\),
\[
\alpha\Delta\alpha+\frac{n-4}{2}|\nabla\alpha|^2
=\frac n2\alpha^4.
\]
Therefore,
\begin{equation*}\tag{i}
\begin{split}
&\int_\Omega\nabla\widetilde u\cdot\nabla\widetilde v\,dx
+\frac{n(n-2)}4\int_\Omega uv\,dV_{\mathbb H}
\\
&=
\int_\Omega\alpha^{n-2}\nabla u\cdot\nabla v\,dx
+\frac{n-2}{2}\int_\Omega
\operatorname{div}\!\left(\alpha^{n-3}uv\nabla\alpha\right)dx
\\
&=
\int_\Omega
\left\langle\nabla_{\mathbb H}u,\nabla_{\mathbb H}v\right\rangle_{\mathbb H}
\,dV_{\mathbb H}
\\
&\quad+
\frac{n-2}{2}\int_{\partial\Omega}
\alpha^{n-3}\frac{\partial\alpha}{\partial\nu}uv\,d\sigma,
\end{split}
\end{equation*}
Moreover,
\begin{equation*}\tag{ii}
\frac{\partial\widetilde u}{\partial\nu}\widetilde v
=
\alpha^{n-2}\frac{\partial u}{\partial\nu}v
+
\frac{n-2}{2}\alpha^{n-3}
\frac{\partial\alpha}{\partial\nu}uv.
\end{equation*}
Substituting (i) and (ii) into the preceding identity, the two
terms containing \(\partial\alpha/\partial\nu\) cancel.  Moreover,
\eqref{eq:outnormal} gives
\[
\int_{\partial\Omega}
\alpha^{n-2}\frac{\partial u}{\partial\nu}v\,d\sigma
=
\int_{\partial\Omega}
\frac{\partial u}{\partial\nu_{\mathbb H}}v\,d\sigma_{\mathbb H}.
\]
Therefore,
\begin{equation}\label{eq:first-green-hyperbolic}
\begin{split}
\int_\Omega
(-\Delta_{\mathbb H}u)v\,dV_{\mathbb H}
&=
\int_\Omega
\left\langle\nabla_{\mathbb H}u,\nabla_{\mathbb H}v\right\rangle_{\mathbb H}
\,dV_{\mathbb H}
\\
&\quad-
\int_{\partial\Omega}
\frac{\partial u}{\partial\nu_{\mathbb H}}v\,d\sigma_{\mathbb H}.
\end{split}
\end{equation}
}

\section{Totally geodesic defects and hyperbolic reflections}\label{sec:reflection-principles}

This section fixes the geometry of totally geodesic defects in hyperbolic space.
We also establish the reflection principles that will be used in the uniqueness proofs.

\medskip
Hyperbolic space admits several standard models, including the Poincar\'e ball model, the upper half-space model, and the hyperboloid model.
Since the scattering theory used in this paper has been established in the Poincar\'e ball model in \cite{ChenLiu2023}, we continue to work in this model for the inverse scattering problem.
We now introduce the geometric objects needed in the sequel, namely totally geodesic hypersurfaces and totally geodesic defects.

Although the scattering analysis is carried out in the Poincar\'e ball model, it is convenient to describe totally geodesic hypersurfaces first in the hyperboloid model, where they are given by intersections with Lorentzian hyperplanes and where the corresponding foliations and reflections have explicit linear forms.
We then transfer these objects to the Poincar\'e ball model by the isometry \(\phi\).

\subsection{Totally geodesic hypersurfaces in the hyperboloid model}

Let \(\mathbb{R}^{n,1}=(\mathbb{R}^{n+1},g_{\rm L})\) be the Minkowski space, where the metric is given by
\[
    ds^2=dx_1^2+\cdots+dx_n^2-dx_{n+1}^2 .
\]
The hyperboloid model of hyperbolic space \(\mathbb{H}^n\) is the submanifold
\begin{equation}\label{eq:hyperboloid}
	 \mathbb H^n
    =
    \left\{
    x\in \mathbb{R}^{n,1}:
    x_1^2+\cdots+x_n^2-x_{n+1}^2=-1,\quad x_{n+1}>0
    \right\}.
\end{equation}
The totally geodesic hypersurface with normal in the \(x_n\)-direction and passing through the point \(o=(0,\ldots,0,1)\) is defined by
\[
    U_{x_n}
    =
    \{x\in \mathbb{H}^n:\ x_n=0\}.
\]
The family of totally geodesic hypersurfaces with the same normal direction can be generated by hyperbolic rotations.
More precisely, define
\[
    A_t^{x_n}
    =
    \operatorname{Id}_{\mathbb R^{n-1}}\oplus \widetilde A_t,
\]
where \(\widetilde A_t\) is the hyperbolic rotation on \(\mathbb R^{1,1}\) given by
\[
    \widetilde A_t
    =
    \begin{pmatrix}
    \cosh t & \sinh t\\
    \sinh t & \cosh t
    \end{pmatrix}.
\]
Here \(\widetilde A_t\) acts on the \((x_n,x_{n+1})\)-variables.
Then \(U_{x_n}^{t}:=A_t^{x_n}(U_{x_n})\), \( t\in\mathbb R\), is a family of totally geodesic hypersurfaces with normal in the \(x_n\)-direction.
These hypersurfaces are pairwise disjoint and foliate the whole hyperbolic space \(\mathbb H^n\).

Similarly, for any \(\nu\in \mathbb S^{n-1}\), the totally geodesic hypersurface with normal direction \(\nu\) and passing through \(o\) is defined by
\begin{equation}\label{eq:Unu}
	U_{\nu}
    =
    \{x\in \mathbb{H}^n:\ (x_1,\ldots,x_n)\cdot\nu=0\}.
\end{equation}
For \(x'=(x_1,\ldots,x_n)\in\mathbb R^n\), write \(x'=(x'\cdot\nu)\nu+y'\), \(y'\perp\nu.\)
Let \(A_t^\nu\) be the hyperbolic rotation acting on the two-dimensional Lorentzian plane spanned by \(\nu\) and \(x_{n+1}\), and acting as the identity on \(\nu^\perp\).
Equivalently,
\[
    A_t^\nu(x',x_{n+1})
    =
    \left(
    \big((x'\cdot\nu)\cosh t+x_{n+1}\sinh t\big)\nu+y',
    (x'\cdot\nu)\sinh t+x_{n+1}\cosh t
    \right).
\]
We define
\(
    U_\nu^t:=A_t^\nu(U_\nu).
\)
Since \(x'\cdot\nu=0\) on \(U_\nu\), a direct calculation gives
\[
    U_\nu^t
    =
    \left\{
    \left(
    \sinh t\,x_{n+1}\nu+y',
    \cosh t\,x_{n+1}
    \right):
    (y',x_{n+1})\in U_\nu
    \right\}.
\]
The hypersurfaces \(U_\nu^t\), \(t\in\mathbb R\), are pairwise disjoint and foliate \(\mathbb H^n\).

\subsection{Totally geodesic hypersurfaces in the Poincar\'e ball model}

Let \(\phi\) be the isometric map from the hyperboloid model $\mathbb H^n$ defined in \eqref{eq:hyperboloid} to the Poincar\'e ball model $\mathbb B^n$.
The map \(\phi\) is obtained by stereographic projection from the hyperboloid to the plane \(\{x_{n+1}=0\}\), with the projection vertex chosen as \((0,0,\ldots,-1)\).
More precisely, we write
\[
    \phi:x\in\mathbb H^n\mapsto \frac{x'}{x_{n+1}+1}\in\mathbb B^n .
\]
Under the map \(\phi\), a direct calculation shows that the totally geodesic hypersurface \(U_{\nu}\) in the hyperboloid model given in \eqref{eq:Unu} is mapped to
\[
    \phi(U_{\nu})=\{x\in\mathbb B^n:x\cdot\nu=0\}.
\]
In the two-dimensional Poincar\'e disk \(\mathbb B^2\), \(\phi(U_{\nu})\) is a geodesic line.

We recall the M\"obius transformations in the Poincar\'e ball.
For each \(a\in\mathbb B^n\), define
\begin{equation}\label{eq:Mobius}
	T_a(x)=\frac{|x-a|^2a-(1-|a|^2)(x-a)}{1-2x\cdot a+|x|^2|a|^2},
\end{equation}
where \(x\cdot a\) denotes the Euclidean scalar product in \(\mathbb R^n\).
The maps \(T_a\) are hyperbolic isometries of \(\mathbb B^n\).
Euclidean rotations are also hyperbolic isometries of the Poincar\'e ball.

\begin{definition}\label{def:tg-hypersurface-ball}
A subset \(V\subset\mathbb B^n\) is called a \emph{totally geodesic hypersurface} in \(\mathbb B^n\) if there exist \(a\in\mathbb B^n\) and \(\nu\in\mathbb S^{n-1}\) such that
\begin{equation}\label{eq:V}
	V=T_a(\phi(U_\nu)).
\end{equation}
Here, \(U_\nu\) denotes a totally geodesic hypersurface in the hyperboloid model and \(\phi\) is the isometry from the hyperboloid model to the Poincar\'e ball model.
\end{definition}

Since Euclidean rotations are hyperbolic isometries, the above definition is independent of the choice of the reference direction.
Equivalently, the family
\[
    \{T_a(\phi(U_\nu)):a\in\mathbb B^n,\ \nu\in\mathbb S^{n-1}\}
\]
gives the class of all totally geodesic hypersurfaces in \(\mathbb B^n\).

\begin{remark}
A totally geodesic hypersurface in \(\mathbb B^n\) is, in particular, a minimal hypersurface in the hyperbolic metric.
\end{remark}

\begin{definition}\label{def:asymptotic-boundary}
Let \(V\subset\mathbb B^n\) be a totally geodesic hypersurface.
We define its \emph{ideal boundary}, also called its asymptotic boundary, by
\[
    \partial_\infty V
    :=
    \overline V^{\,\overline{\mathbb B^n}}\cap\mathbb S^{n-1},
\]
where the closure is taken in the closed unit ball.
\end{definition}

\begin{remark}\label{rem:asymptotic-boundary}
For the standard totally geodesic hypersurface
\[
    V_\nu=\{x\in\mathbb B^n:x\cdot\nu=0\},
\]
one has
\[
    \partial_\infty V_\nu
    =
    \{\eta\in\mathbb S^{n-1}:\eta\cdot\nu=0\}.
\]
Hence \(\partial_\infty V_\nu\) is an \((n-2)\)-dimensional great sphere in \(\mathbb S^{n-1}\).
For a general totally geodesic hypersurface \(V=T_a(\phi(U_\nu))\), its ideal boundary \(\partial_\infty V\) is the image of such a great sphere under the boundary action of the M\"obius transformation \(T_a\).
\end{remark}

We now give the precise geometric definition of the defects considered below.
A bulk component is an impenetrable compact component with nonempty interior.
A hypersurface-supported component is an impenetrable hypersurface-type compact component with empty interior.
In both cases, the trace condition is imposed on the interaction surface of the defect.

\begin{definition}\label{def:totally-geodesic-defect}
A subset \(\mathcal F\subset\mathbb B^n\) is called a \emph{totally geodesic face} if there exist \(a\in\mathbb B^n\) and \(\nu\in\mathbb S^{n-1}\) such that, with
\[
    V=T_a(\phi(U_\nu)),
\]
where \(V\) is a totally geodesic hypersurface in the sense of Definition~\ref{def:tg-hypersurface-ball}, the set \(\mathcal F\) is the relative closure in \(V\) of a nonempty, relatively compact, simply connected open subset of \(V\).

A nonempty compact set \(\mathcal P\subset\mathbb B^n\) is called a \emph{totally geodesic defect} if its exterior
\(\mathbb B^n\setminus\mathcal P\) is connected and its interaction surface, denoted by \(\partial\mathcal P\), is the union of finitely many totally geodesic faces \(\mathcal F_j\), \(j=1,\ldots,J\), with pairwise disjoint relative interiors.
That is,
\[
    \partial\mathcal P
    =
    \bigcup_{j=1}^{J}\mathcal F_j,
    \qquad
    J\geq1,
    \qquad
    \mathcal F_j
    \text{ is a totally geodesic face for }j=1,\ldots,J .
\]
A connected component of \(\mathcal P\) with nonempty interior is called a bulk component.
A connected component of \(\mathcal P\) with empty interior is called a hypersurface-supported component.
Thus this definition includes both bulk components and hypersurface-supported components.
\end{definition}

\subsection{\texorpdfstring{Reflection about the totally geodesic hypersurface in hyperbolic space \(\mathbb B^n\)}{Reflection about the totally geodesic hypersurface in hyperbolic space Bn}}

For the totally geodesic hypersurface
\[
    \phi(U_{x_n})
    =
    \{x\in\mathbb B^n:\ x_n=0\}
\]
along the \(x_n\)-direction in \(\mathbb B^n\), we define the reflection \(I_{x_n}\) about \(\phi(U_{x_n})\) by
\[
    I_{x_n}(x_1,x_2,\ldots,x_n)
    =
    (x_1,x_2,\ldots,-x_n),
    \qquad
    x=(x_1,x_2,\ldots,x_n)\in\mathbb B^n.
\]
Clearly, \(I_{x_n}(\phi(U_{x_n}))=\phi(U_{x_n})\).
In the two-dimensional case, using the equation of geodesic lines, one can directly verify that \(I_{x_n}\) maps geodesic lines to geodesic lines.

In the \(n\)-dimensional case, \(I_{x_n}\) maps totally geodesic hypersurfaces to totally geodesic hypersurfaces in \(\mathbb B^n\).
Indeed, \(I_{x_n}\) is an orthogonal transformation of \(\mathbb R^n\), and hence a hyperbolic isometry of the Poincar\'e ball.
Moreover, for M\"obius transformation \eqref{eq:Mobius}, one has
\[
    I_{x_n}\circ T_a
    =
    T_{I_{x_n}(a)}\circ I_{x_n}.
\]
Therefore, if \(V=T_a(\phi(U_{\nu_1}))\), then
\[
    I_{x_n}(V)
    =
    T_{I_{x_n}(a)}
    \bigl(I_{x_n}(\phi(U_{\nu_1}))\bigr).
\]
Since \(\phi(U_{\nu_1})\) is the Euclidean hyperplane through the origin orthogonal to \(\nu_1\), the reflection \(I_{x_n}\) maps it onto another Euclidean hyperplane through the origin.
Thus there exists \(\nu_2\in\mathbb S^{n-1}\) such that
\[
    I_{x_n}(\phi(U_{\nu_1}))=\phi(U_{\nu_2}),
\]
and hence
\[
    I_{x_n}(V)
    =
    T_{I_{x_n}(a)}(\phi(U_{\nu_2})).
\]
Thus \(I_{x_n}(V)\) is again a totally geodesic hypersurface.

Similarly, one can define the reflection \(I_\nu\) about \(\phi(U_\nu)\) for any \(\nu\in\mathbb S^{n-1}\).
If \(x=(x\cdot\nu)\nu+y\), where \(y\in\nu^\perp\), then
\[
    I_\nu(x)
    =
    -(x\cdot\nu)\nu+y.
\]
Equivalently,
\[
    I_\nu(x)=x-2(x\cdot\nu)\nu.
\]
Clearly, \(I_\nu(\phi(U_\nu))=\phi(U_\nu)\).
The same argument as above shows that \(I_\nu\) maps totally geodesic hypersurfaces to totally geodesic hypersurfaces.

\begin{definition}\label{def:reflection-general-hypersurface}
Let \(V\subset\mathbb B^n\) be a totally geodesic hypersurface.
By \eqref{eq:V}, \(V\) can be written as \(V=T_a(\phi(U_\nu))\) for some \(a\in\mathbb B^n\) and some \(\nu\in\mathbb S^{n-1}\).
The reflection about \(V=T_a(\phi(U_\nu))\) is defined by
\begin{equation}
\label{eq:reflection}
    I_\nu^a(x)
    :=
    (T_a\circ I_\nu\circ T_a)(x),
\end{equation}
where we have used \(T_a^{-1}=T_a\).
Then \(I_\nu^a\) fixes \(V\) pointwise.
Indeed, if \(x=T_a(z)\in V\) with \(z\in\phi(U_\nu)\), then
\[
    I_\nu^a(x)
    =
    T_a\circ I_\nu\circ T_a(T_a(z))
    =
    T_a\circ I_\nu(z)
    =
    T_a(z)
    =
    x.
\]
Moreover, \(I_\nu^a\) maps totally geodesic hypersurfaces to totally geodesic hypersurfaces, since \(T_a\) and \(I_\nu\) both map totally geodesic hypersurfaces to totally geodesic hypersurfaces.
Equivalently, it is the unique hyperbolic isometry that fixes \(V\) pointwise
and reverses its unit normal.  Hence the reflection depends only on \(V\), not
on the representation \(V=T_a(\phi(U_\nu))\), and we denote it by \(I_V:=I_\nu^a\).
\end{definition}

\subsection{\texorpdfstring{Reflection principles in \(\mathbb B^n\)}{Reflection principles in the Poincare ball}}\label{subsec reflection}

{
Recall from \eqref{eq:outnormal} that the hyperbolic normal derivative along a smooth
hypersurface \(S\subset\mathbb B^n\) is
\[
    \frac{\partial u}{\partial\nu_{\mathbb H}}
    :=
    \bigl\langle\nabla_{\mathbb H}u,\nu_{\mathbb H}\bigr\rangle_{\mathbb H}
    =du(\nu_{\mathbb H}),
\]
where \(\nu_{\mathbb H}\) is a unit hyperbolic normal to \(S\).
The following covariance formula gives the transformation law needed for
reflection.

\begin{lemma}\label{lem:out-normal}
Let \(M:\mathbb B^n\to\mathbb B^n\) be a hyperbolic isometry, let \(S\) be a
smooth hypersurface, and set \(S':=M(S)\).  If \(\nu_{\mathbb H}\) is a unit
hyperbolic normal along \(S\), orient \(S'\) by
\[
    \nu'_{\mathbb H}(Mx):=dM_x\nu_{\mathbb H}(x),
    \qquad x\in S.
\]
Then, for every \(u\in C^1\) in a neighborhood of \(S'\),
\begin{equation}\label{eq:normal-covariance-isometry}
    \left(
        \frac{\partial u}{\partial\nu'_{\mathbb H}}
    \right)\circ M
    =
    \frac{\partial(u\circ M)}{\partial\nu_{\mathbb H}}
    \qquad\text{on }S.
\end{equation}
In particular, if \(V\) is a totally geodesic hypersurface and \(I_V\) is its
hyperbolic reflection, then \(I_V\) fixes \(V\) pointwise and
\[
    d(I_V)_x\nu_{\mathbb H}(x)=-\nu_{\mathbb H}(x),
    \qquad x\in V.
\]
Consequently,
\begin{equation}\label{eq:normal-reflection}
    \frac{\partial(u\circ I_V)}{\partial\nu_{\mathbb H}}
    =
    -\frac{\partial u}{\partial\nu_{\mathbb H}}
    \qquad\text{on }V.
\end{equation}
\end{lemma}

\begin{proof}
Since \(M\) is a hyperbolic isometry, \(dM_x\nu_{\mathbb H}(x)\) is a unit
hyperbolic normal to \(S'\) at \(Mx\).  The chain rule gives
\[
\begin{split}
    \frac{\partial(u\circ M)}{\partial\nu_{\mathbb H}}(x)
    & =d(u\circ M)_x\bigl(\nu_{\mathbb H}(x)\bigr) \\
    & =du_{Mx}\bigl(dM_x\nu_{\mathbb H}(x)\bigr) \\
    & =\frac{\partial u}{\partial\nu'_{\mathbb H}}(Mx),
\end{split}
\]
which proves \eqref{eq:normal-covariance-isometry}.  For a hyperbolic
reflection, \(d(I_V)_x\) is the identity on \(T_xV\) and multiplication by
\(-1\) on the normal line.  Applying the covariance identity with \(M=I_V\)
and using \(I_Vx=x\) on \(V\) proves \eqref{eq:normal-reflection}.
\end{proof}
}

\begin{lemma}\label{lem3}
Let \(a\in\mathbb B^n\). Assume that \(\Omega\) is a domain in \(\mathbb{B}^n\) and \(u\) satisfies
\[
    -\Delta_{\mathbb{H}}u=f(x),
    \qquad x\in \Omega.
\]
Then \(u\circ I_{x_n}^a\) satisfies
\[
    -\Delta_{\mathbb{H}}(u\circ I_{x_n}^a)
    =
    f\circ I_{x_n}^a,
    \qquad
    x\in (I_{x_n}^a)^{-1}(\Omega).
\]
\end{lemma}

\begin{proof}
The map \(I_{x_n}\) is an isometry of the Poincar\'e ball model, and \(T_a\) is also a hyperbolic isometry.
Hence
\[
    I_{x_n}^a=T_a\circ I_{x_n}\circ T_a
\]
is a hyperbolic isometry.
Therefore the Laplace--Beltrami operator commutes with pullback by \(I_{x_n}^a\):
\[
    \Delta_{\mathbb H}(u\circ I_{x_n}^a)
    =
    (\Delta_{\mathbb H}u)\circ I_{x_n}^a.
\]
The desired equation follows immediately.

The proof is complete.
\end{proof}

\begin{lemma}\label{lem:reflection-dirichlet}
Assume that \(\Omega\) is a connected domain, bounded or unbounded, in
\(\mathbb B^n\), and is invariant under \(I_{x_n}^a\), the hyperbolic
reflection across the totally geodesic hypersurface
\[
    V=T_a(\phi(U_{x_n})).
\]
Let \(\Gamma\subset V\cap\Omega\) be a nonempty relatively open subset.
Assume that \(u\) satisfies
\[
    -\Delta_{\mathbb H}u-\frac{(n-1)^2}{4}u=\lambda_0^2u
    \qquad \text{in }\Omega,
\]
and
\[
    u=0
    \qquad \text{on }\Gamma .
\]
Then \(u\) is odd under \(I_{x_n}^a\), that is,
\[
    u(x)=-(u\circ I_{x_n}^a)(x),
    \qquad x\in\Omega .
\]
\end{lemma}

\begin{proof}
Since \(\Omega\) is invariant under \(I_{x_n}^a\), the function
\(u\circ I_{x_n}^a\) is defined in \(\Omega\).
By Lemma~\ref{lem3}, \(u\circ I_{x_n}^a\) satisfies the same shifted Helmholtz equation in \(\Omega\).
Set
\[
    w:=u+u\circ I_{x_n}^a .
\]
Then \(w\) satisfies
\[
    -\Delta_{\mathbb H}w-\frac{(n-1)^2}{4}w=\lambda_0^2w
    \qquad \text{in }\Omega .
\]
Since \(I_{x_n}^a\) fixes \(V\) pointwise, we have \(w=2u=0\) on \(\Gamma\).
Moreover, by \eqref{eq:normal-reflection},
\[
    \frac{\partial w}{\partial\nu_{\mathbb H}}
    =
    \frac{\partial u}{\partial\nu_{\mathbb H}}
    +
    \frac{\partial(u\circ I_{x_n}^a)}{\partial\nu_{\mathbb H}}
    =
    0
    \qquad \text{on }\Gamma .
\]
By unique continuation from Cauchy data on the nonempty analytic hypersurface piece \(\Gamma\), we obtain \(w=0\) in \(\Omega\).
Therefore
\[
    u+u\circ I_{x_n}^a=0
    \qquad \text{in }\Omega ,
\]
which proves the odd symmetry.
\end{proof}

\begin{lemma}\label{lem:reflection-neumann}
Assume that \(\Omega\) is a connected domain, bounded or unbounded, in
\(\mathbb B^n\), and is invariant under \(I_{x_n}^a\), the hyperbolic
reflection across the totally geodesic hypersurface
\[
    V=T_a(\phi(U_{x_n})).
\]
Let \(\Gamma\subset V\cap\Omega\) be a nonempty relatively open subset.
Assume that \(u\) satisfies
\[
    -\Delta_{\mathbb H}u-\frac{(n-1)^2}{4}u=\lambda_0^2u
    \qquad \text{in }\Omega,
\]
and
\[
    \frac{\partial u}{\partial\nu_{\mathbb H}}=0
    \qquad \text{on }\Gamma .
\]
Then \(u\) is even under \(I_{x_n}^a\), that is,
\[
    u(x)=(u\circ I_{x_n}^a)(x),
    \qquad x\in\Omega .
\]
\end{lemma}

\begin{proof}
Since \(\Omega\) is invariant under \(I_{x_n}^a\), the function
\(u\circ I_{x_n}^a\) is defined in \(\Omega\).
By Lemma~\ref{lem3}, \(u\circ I_{x_n}^a\) satisfies the same shifted Helmholtz equation in \(\Omega\).
Set
\[
    w:=u-u\circ I_{x_n}^a .
\]
Then \(w\) satisfies
\[
    -\Delta_{\mathbb H}w-\frac{(n-1)^2}{4}w=\lambda_0^2w
    \qquad \text{in }\Omega .
\]
Since \(I_{x_n}^a\) fixes \(V\) pointwise, we have \(w=0\) on \(\Gamma\).
Moreover, by \eqref{eq:normal-reflection},
\[
    \frac{\partial(u\circ I_{x_n}^a)}{\partial\nu_{\mathbb H}}
    =
    -
    \frac{\partial u}{\partial\nu_{\mathbb H}}
    \qquad \text{on }\Gamma .
\]
Hence
\[
    \frac{\partial w}{\partial\nu_{\mathbb H}}
    =
    2\frac{\partial u}{\partial\nu_{\mathbb H}}
    =
    0
    \qquad \text{on }\Gamma .
\]
By unique continuation from Cauchy data on the nonempty analytic hypersurface piece \(\Gamma\), we obtain \(w=0\) in \(\Omega\).
Therefore
\[
    u-u\circ I_{x_n}^a=0
    \qquad \text{in }\Omega ,
\]
which proves the even symmetry.
\end{proof}

{
\begin{lemma}
\label{lem:one-sided-reflection-extension}
Let \(V\subset\mathbb B^n\) be a totally geodesic hypersurface, let \(I_V\) be
its hyperbolic reflection, and let \(\Omega\) be a domain invariant under
\(I_V\).  Set \(\Gamma:=\Omega\cap V\), and suppose that
\[
    \Omega=\Omega_+\cup\Gamma\cup\Omega_-,
    \qquad
    \Omega_-:=I_V(\Omega_+),
\]
where \(\Omega_+\) and \(\Omega_-\) are the two sides of \(V\) in \(\Omega\).
Let
\[
    \Lambda_0:=\frac{(n-1)^2}{4}+\lambda_0^2.
\]

Assume first that
\(u\in H^1_{\mathbb H,\mathrm{loc}}(\Omega_+\cup\Gamma)\) is a weak solution
of
\[
    -\Delta_{\mathbb H}u-\Lambda_0u=0
    \qquad\text{in }\Omega_+,
\]
and has zero trace on \(\Gamma\).  Then its odd extension
\[
    \mathcal E_{\mathsf D}u(x):=
    \begin{cases}
        u(x), & x\in\Omega_+,\\
        0, & x\in\Gamma,\\
        -u(I_Vx), & x\in\Omega_-,
    \end{cases}
\]
belongs to \(H^1_{\mathbb H,\mathrm{loc}}(\Omega)\) and is a weak solution of
\[
    -\Delta_{\mathbb H}(\mathcal E_{\mathsf D}u)
    -\Lambda_0\mathcal E_{\mathsf D}u=0
    \qquad\text{in }\Omega.
\]

Assume instead that \(u\) satisfies the homogeneous hyperbolic Neumann condition
on \(\Gamma\) in the weak sense, that is,
\[
    \int_{\Omega_+}
    \left(
        \langle\nabla_{\mathbb H}u,\nabla_{\mathbb H}\psi\rangle_{\mathbb H}
        -\Lambda_0u\psi
    \right)dV_{\mathbb H}=0
\]
for every \(H^1\)-test function \(\psi\) compactly supported in
\(\Omega_+\cup\Gamma\).  Then its even extension
\[
    \mathcal E_{\mathsf N}u(x):=
    \begin{cases}
        u(x), & x\in\Omega_+\cup\Gamma,\\
        u(I_Vx), & x\in\Omega_-,
    \end{cases}
\]
belongs to \(H^1_{\mathbb H,\mathrm{loc}}(\Omega)\) and satisfies the same
shifted Helmholtz equation weakly in \(\Omega\).  In either case, interior
elliptic regularity implies that the reflected extension is smooth, and
indeed real analytic, in \(\Omega\).
\end{lemma}

\begin{proof}
The zero trace in the Dirichlet case makes the two traces of the odd extension
agree on \(\Gamma\), while the two traces of the even extension agree there
automatically.  Hence the corresponding extension belongs to
\(H^1_{\mathbb H,\mathrm{loc}}(\Omega)\).

Let \(\varphi\in C_c^\infty(\Omega)\).  Since \(I_V\) is a hyperbolic
isometry, changing variables \(x=I_Vy\) on \(\Omega_-\) shows in the
Dirichlet case that
\[
\begin{split}
&\int_{\Omega}
\left(
    \langle\nabla_{\mathbb H}(\mathcal E_{\mathsf D}u),
    \nabla_{\mathbb H}\varphi\rangle_{\mathbb H}
    -\Lambda_0(\mathcal E_{\mathsf D}u)\varphi
\right)dV_{\mathbb H} \\
&\qquad=
\int_{\Omega_+}
\left(
    \left\langle\nabla_{\mathbb H}u,
    \nabla_{\mathbb H}(\varphi-\varphi\circ I_V)\right\rangle_{\mathbb H}
    -\Lambda_0u(\varphi-\varphi\circ I_V)
\right)dV_{\mathbb H}.
\end{split}
\]
The test function \(\varphi-\varphi\circ I_V\) vanishes on \(\Gamma\), so the
last integral is zero by the weak Dirichlet equation in \(\Omega_+\).

For the even extension, the same change of variables gives
\[
\begin{split}
&\int_{\Omega}
\left(
    \langle\nabla_{\mathbb H}(\mathcal E_{\mathsf N}u),
    \nabla_{\mathbb H}\varphi\rangle_{\mathbb H}
    -\Lambda_0(\mathcal E_{\mathsf N}u)\varphi
\right)dV_{\mathbb H} \\
&\qquad=
\int_{\Omega_+}
\left(
    \left\langle\nabla_{\mathbb H}u,
    \nabla_{\mathbb H}(\varphi+\varphi\circ I_V)\right\rangle_{\mathbb H}
    -\Lambda_0u(\varphi+\varphi\circ I_V)
\right)dV_{\mathbb H}.
\end{split}
\]
This vanishes by the weak formulation of the homogeneous hyperbolic Neumann
condition.  Thus both extensions
solve the shifted Helmholtz equation weakly in \(\Omega\), and the final
assertion follows from interior elliptic regularity and analyticity for the
analytic coefficients of the hyperbolic Laplacian.
\end{proof}
}

The Dirichlet and Neumann reflection principles above are stated for
\(V=T_a(\phi(U_{x_n}))\) only to simplify notation.
By conjugating with an orthogonal transformation sending \(\nu\) to \(e_n\), the same statements hold for every totally geodesic hypersurface
\[
    V=T_a(\phi(U_\nu)), \qquad \nu\in\mathbb S^{n-1},
\]
with \(I_{x_n}^a\) replaced by \(I_\nu^a\). 
We shall use this general form below.

\section{Unique determination of totally geodesic defects}\label{sec:unique-determination}

In this section we prove the uniqueness results for totally geodesic defects.
Let \(\mathcal P\) be a totally geodesic defect in the sense of
Definition~\ref{def:totally-geodesic-defect}, endowed with either the Dirichlet
or the Neumann boundary condition.
We write
\[
    G:=\mathbb B^n\setminus \mathcal P,
\]
and let \(u\) be the exterior solution of \eqref{eq:hhelm1} in \(G\).

We first introduce the Dirichlet and Neumann sets associated with totally geodesic
hypersurfaces.

\begin{definition}\label{def:dirichlet-neumann-set}
Let \(\mathscr V_x\) denote the family of all totally geodesic hypersurfaces in
\(\mathbb B^n\) passing through \(x\).

The \emph{Dirichlet set} of \(u\) in \(G\) is defined by
\[
\mathcal D_u
:=
\left\{
x\in G:\ 
u \text{ vanishes locally near }x\text{ on }V\cap G
\text{ for some }V\in\mathscr V_x
\right\}.
\]

Let \(\mathbf U=(u_1,\ldots,u_m)\) be a finite family of solutions in \(G\).
The \emph{Neumann set} of \(\mathbf U\) in \(G\) is defined by
\[
\mathcal N_{\mathbf U}
:=
\left\{
x\in G:\ 
\partial_{\nu_{\mathbb H}}\mathbf U=0
\text{ locally near }x\text{ on }V\cap G
\text{ for some }V\in\mathscr V_x
\right\},
\]
where
\[
\partial_{\nu_{\mathbb H}}\mathbf U=0
\quad\text{means}\quad
\partial_{\nu_{\mathbb H}}u_\ell=0,\qquad \ell=1,\ldots,m.
\]
For a single solution \(u\), we write \(\mathcal N_u\).
\end{definition}

\begin{lemma}\label{lem:analytic-continuation-hypersurface}
Let \(V\) be a totally geodesic hypersurface in \(\mathbb B^n\), and let \(C\) be a connected component of \(V\cap G\).
If \(u=0\) on a nonempty relatively open subset of \(C\), then
\[
u=0
\qquad\text{on } C.
\]
Similarly, if \(\partial_{\nu_{\mathbb H}}\mathbf U=0\) on a nonempty relatively open subset of \(C\), then
\[
\partial_{\nu_{\mathbb H}}\mathbf U=0
\qquad\text{on } C.
\]
\end{lemma}

\begin{proof}
Totally geodesic hypersurfaces are real analytic, and solutions of the shifted Helmholtz equation are real analytic in \(G\).
Thus the restrictions \(u|_C\) and \((\partial_{\nu_{\mathbb H}}u_\ell)|_C\), \(\ell=1,\ldots,m\), are real analytic on \(C\).
The conclusions follow from the identity theorem for real analytic functions.
\end{proof}

If \(x\in\mathcal D_u\), then by Lemma~\ref{lem:analytic-continuation-hypersurface}, \(u\) vanishes on the connected component of \(V\cap G\) containing \(x\), where \(V\in\mathscr V_x\) is the hypersurface appearing in Definition~\ref{def:dirichlet-neumann-set}.
We call this connected component a \emph{Dirichlet totally geodesic hypersurface} of \(u\).

Similarly, if \(x\in\mathcal N_{\mathbf U}\), then the corresponding connected component of \(V\cap G\) is called a \emph{Neumann totally geodesic hypersurface} of \(\mathbf U\).
\subsection{Proof of Theorem~\ref{thm:mainip1}}

We first prove a boundedness lemma for the Dirichlet set.

\begin{lemma}\label{lem:mainip1}
The Dirichlet set \(\mathcal D_u\) and all Dirichlet totally geodesic hypersurfaces are bounded.
\end{lemma}

\begin{proof}
Since
\[
u^i=e_{2\lambda_0,\xi}(x)
=
\left(\cosh\frac{\rho}{2}\right)^{-2\lambda_0\mathrm i-(n-1)}
\left|\tanh\frac{\rho}{2}\hat{x}-\xi\right|^{-2\lambda_0\mathrm i-(n-1)},
\]
a direct calculation gives
\[
\begin{split}
\frac{\partial u^{i}}{\partial \rho}
&=
\left(-\lambda_0\mathrm i-\frac{n-1}{2}\right)
\left(\cosh\frac{\rho}{2}\right)^{-2\lambda_0\mathrm i-(n-1)}
\tanh\frac{\rho}{2}
\left|\tanh\frac{\rho}{2}\hat{x}-\xi\right|^{-2\lambda_0\mathrm i-(n-1)}
\\
&\quad+
\left(\cosh\frac{\rho}{2}\right)^{-2\lambda_0\mathrm i-(n+1)}
\left(-\lambda_0\mathrm i-\frac{n-1}{2}\right)
\left|\tanh\frac{\rho}{2}\hat{x}-\xi\right|^{-2\lambda_0\mathrm i-(n+1)}
\left(\tanh\frac{\rho}{2}\hat{x}-\xi\right)\cdot\hat{x}.
\end{split}
\]
Hence
\[
\begin{split}
&\frac{\partial u^{i}}{\partial \rho}
-
\left(-\lambda_0\mathrm i-\frac{n-1}{2}\right)
\tanh\frac{\rho}{2}\,u^i
\\
&\quad =
\left(\cosh\frac{\rho}{2}\right)^{-2\lambda_0\mathrm i-(n+1)}
\left(-\lambda_0\mathrm i-\frac{n-1}{2}\right)
\left|\tanh\frac{\rho}{2}\hat{x}-\xi\right|^{-2\lambda_0\mathrm i-(n+1)}
\left(\tanh\frac{\rho}{2}\hat{x}-\xi\right)\cdot\hat{x}.
\end{split}
\]
This yields
\begin{equation}\label{eq:incoming-radial-identity}
\frac{\partial u^{i}}{\partial \rho}
-
\left(-\lambda_0\mathrm i-\frac{n-1}{2}\right)
\tanh\frac{\rho}{2}\,u^i
=
O\left(\sinh^{-\frac{n+1}{2}}\rho\right)
\end{equation}
for any \(\hat{x}\neq\xi\).
On the other hand, the outgoing correction satisfies the outgoing radiation condition
\[
\frac{\partial u^{s}}{\partial \rho}
-
\left(\lambda_0\mathrm i-\frac{n-1}{2}\right)
\tanh\frac{\rho}{2}\,u^s
=
O\left(\sinh^{-\frac{n+1}{2}}\rho\right).
\]

Suppose that there exists an unbounded Dirichlet totally geodesic hypersurface \(\Pi\), and let \(V\) be the totally geodesic hypersurface supporting \(\Pi\).
Since \(\Pi\) is an unbounded connected portion of \(V\), it contains a geodesic ray \(x=x(\rho)\) with \(\rho(x)=\rho\to+\infty\) and endpoint \(\eta\in\partial_\infty V\subset\mathbb S^{n-1}\).
On \(\Pi\), we have
\[
u^i=-u^s.
\]
If \(\eta=\xi\), then
\[
|u^i(x)|\asymp e^{\frac{n-1}{2}\rho},
\qquad
u^s(x)=O\left(e^{-\frac{n-1}{2}\rho}\right),
\]
which is impossible.
If \(\eta\neq\xi\), then \eqref{eq:incoming-radial-identity} and the outgoing expansion of \(u^s\) give, along this ray,
\[
u^i(x)
=
e^{-\frac{n-1}{2}\rho}e^{-\mathrm i\lambda_0\rho}(a_\eta+o(1)),
\qquad
u^s(x)
=
e^{-\frac{n-1}{2}\rho}e^{\mathrm i\lambda_0\rho}(b_\eta+o(1)),
\]
where \(a_\eta\neq0\) and \(b_\eta\) is finite.
Thus \(u^i+u^s=0\) implies
\[
a_\eta e^{-\mathrm i\lambda_0\rho}
+
b_\eta e^{\mathrm i\lambda_0\rho}
+
o(1)
=0,
\]
which is impossible as \(\rho\to+\infty\).
Therefore every Dirichlet totally geodesic hypersurface is bounded.

Finally, choose \(R>0\) such that
\[
\mathcal P\subset B_{\mathbb{H}}(R).
\]
If there existed \(x\in\mathcal D_u\) with \(\rho(x)>R\), then the local vanishing of \(u\) on some totally geodesic hypersurface \(V\) through \(x\) would extend, by Lemma~\ref{lem:analytic-continuation-hypersurface}, to the connected component of \(V\cap G\) containing \(x\).
Since
\[
V\setminus B_{\mathbb{H}}(R)\subset V\cap G,
\]
this connected component is unbounded, contradicting the boundedness proved above.
Hence
\[
\mathcal D_u\subset B_{\mathbb{H}}(R).
\]

The proof is complete.
\end{proof}

\begin{proof}[Proof of Theorem~\ref{thm:mainip1}]
We argue by contradiction.
Assume that there exist two \(\mathsf D\)-type totally geodesic defects \(\mathcal P\) and \(\mathcal P'\) such that \(\mathcal P\neq \mathcal P'\), but they generate the same far-field pattern for one fixed incoming Helgason mode \(u^i=e_{2\lambda_0,\xi}\).
Let \(u,u'\), \(u^s,(u^s)'\), and
\(u_{\infty,\mathcal P,\xi},u_{\infty,\mathcal P',\xi}\)
denote respectively the exterior solutions, the outgoing corrections, and the far-field patterns associated with
\(\mathcal P,\mathcal P'\).
Thus
\begin{equation}\label{eq:arg1}
u_{\infty, \mathcal P, \xi}(\hat x)
=
u_{\infty, \mathcal P', \xi}(\hat x),
\qquad
\hat x\in\mathbb S^{n-1}.
\end{equation}

\medskip
\noindent\emph{Step 1. Equality of the exterior solutions in the common exterior.}

Set
\[
    G:=\mathbb B^n\setminus \mathcal P,
    \qquad
    G':=\mathbb B^n\setminus \mathcal P'.
\]
Let \(\Omega_0\) be the connected component of
\[
    \mathbb B^n\setminus
    \left(
    \overline{\mathcal P}\cup\overline{\mathcal P'}
    \right)
\]
which is unbounded with respect to the hyperbolic distance, namely
\(\sup_{x\in\Omega_0}\rho(x)=+\infty\).
Since the incoming Helgason mode is the same for both defects, one has
\[
    u-u'=u^s-(u^s)'
    \qquad
    \text{in } \Omega_0 .
\]
By \eqref{eq:arg1} and the hyperbolic Rellich theorem in Lemma~\ref{thm3}, we have \(u^s=(u^s)'\) in \(\Omega_0\).
Hence \(u=u'\) in \(\Omega_0\).

\medskip
\noindent\emph{Step 2. Production of the first Dirichlet totally geodesic hypersurface.}

We show that the difference between the two defects produces a Dirichlet totally geodesic hypersurface in the exterior of one of them.
If every point of \(\partial\Omega_0\) belonged to \(\overline{\mathcal P}\cap \overline{\mathcal P'}\), then, by the connectedness of \(G\) and \(G'\), one would have \(G=\Omega_0=G'\).
Consequently,
\[
    \overline{\mathcal P}
    =
    \mathbb B^n\setminus G
    =
    \mathbb B^n\setminus G'
    =
    \overline{\mathcal P'},
\]
contradicting \(\mathcal P\neq\mathcal P'\).
Therefore, after interchanging \(\mathcal P\) and \(\mathcal P'\) if necessary, there exist a totally geodesic face \(\mathcal F_1\subset \partial\mathcal P\) and a nonempty relatively open subset \(\Gamma_1\subset \mathcal F_1\) such that
\[
    \Gamma_1\subset \partial\Omega_0\cap G' .
\]

Since \(\mathcal P\) is \(\mathsf D\)-type, we have \(u=0\) on \(\Gamma_1\).
Since \(u=u'\) in \(\Omega_0\), taking the trace from the common exterior side gives \(u'=0\) on \(\Gamma_1\).
Applying Lemma~\ref{lem:analytic-continuation-hypersurface} to \(u'\) in \(G'\), we obtain \(u'=0\) on the connected component of \(V_1\cap G'\) containing \(\Gamma_1\).
We denote this component by \(\Pi_1\).
Then \(\Pi_1\) is a Dirichlet totally geodesic hypersurface of \(u'\) in \(G'\).

\medskip
\noindent\emph{Step 3. Choice of a proper curve and a fixed hyperbolic scale.}

By Lemma~\ref{lem:mainip1}, the Dirichlet set \(\mathcal D_{u'}\) and all Dirichlet totally geodesic hypersurfaces of \(u'\) are bounded.
In particular, \(\Pi_1\) is bounded.
Hence \(G'\setminus\Pi_1\) has an unbounded connected component.
Choose \(x_1\in\Pi_1\) on the boundary of such an unbounded component.

We choose a proper regular curve \(\gamma:[0,\infty)\to G'\) such that
\[
    \gamma(0)=x_1,
    \qquad
    \gamma(t)\in G'\setminus\Pi_1 \quad \text{for }t>0,
\]
and
\[
    \rho(\gamma(t))\to+\infty
    \qquad
    \text{as }t\to+\infty.
\]
Set \(t_1=0\).
Moreover, \(\gamma\) may be chosen so that
\[
    \rho(\gamma,\partial\mathcal P')
    :=
    \inf_{\substack{t\geq0\\y\in\partial\mathcal P'}}
    \rho(\gamma(t),y)
    >0 .
\]
Indeed, one first connects \(x_1\) to a point outside a large hyperbolic ball containing \(\overline{\mathcal P'}\), and then appends a geodesic ray staying outside a slightly larger ball.
The compact initial part stays in \(G'\), and the final ray stays a positive distance away from \(\partial\mathcal P'\).

Set
\[
    r_0:=\frac12 \rho(\gamma,\partial\mathcal P')>0.
\]
Then
\begin{equation}\label{eq:ball-in-G2}
    \overline{B_{\mathbb H}(\gamma(t),r_0)}
    \subset G',
    \qquad t\geq0.
\end{equation}

\medskip
\noindent\emph{Step 4. One reflected-continuation step.}

Suppose that for some \(j\geq1\) we have already obtained a Dirichlet totally geodesic hypersurface \(\Pi_j\) and a parameter \(t_j\) such that
\(x_j:=\gamma(t_j)\in\Pi_j\), where \(t_j\) is chosen as the last parameter for which \(\gamma\) meets \(\Pi_j\).
Let \(V_j\) be the totally geodesic hypersurface containing \(\Pi_j\).
By Definition~\ref{def:tg-hypersurface-ball}, we may write
\[
   V_j=T_{a_j}(\phi(U_{\nu_j}))
\]
for some \(a_j\in\mathbb B^n\).
Let \(I_j:=I_{\nu_j}^{a_j}\) be the hyperbolic reflection with respect to \(V_j\), defined in \eqref{eq:reflection}.

Since \(x_j\in V_j\), the hyperbolic ball \(B_{\mathbb H}(x_j,r_0)\) is invariant under \(I_j\).
By \eqref{eq:ball-in-G2},
\[
    B_{\mathbb H}(x_j,r_0)\subset G',
    \qquad
    I_j(B_{\mathbb H}(x_j,r_0))
    =
    B_{\mathbb H}(x_j,r_0)\subset G'.
\]
Hence \(B_{\mathbb H}(x_j,r_0)\subset G'\cap I_j(G')\).
Let \(E_j\) be the connected component of \(G'\cap I_j(G')\) which contains \(B_{\mathbb H}(x_j,r_0)\).
Then \(E_j\) is symmetric with respect to \(I_j\), namely \(I_j(E_j)=E_j\).

On the nonempty analytic hypersurface piece
\[
    \Pi_j\cap B_{\mathbb H}(x_j,r_0)\subset V_j\cap E_j,
\]
we have \(u'=0\).
Applying Lemma~\ref{lem:reflection-dirichlet} in the symmetric connected domain \(E_j\), with
\[
    \Gamma=\Pi_j\cap B_{\mathbb H}(x_j,r_0),
\]
we obtain
\begin{equation}\label{eq:odd-symmetry-j}
	u'(x)=-u'(I_jx),
    \qquad x\in E_j .
\end{equation}

We claim that \(E_j\) is bounded.
Since \(E_j\) is a connected component of \(G'\cap I_j(G')\),
\[
    \partial E_j
    \subset
    \partial G'\cup I_j(\partial G')
    =
    \partial\mathcal P'\cup I_j(\partial\mathcal P').
\]
The right-hand side is compact in \(\mathbb B^n\), because \(\mathcal P'\) is compact and \(I_j\) is a hyperbolic isometry.
Hence there exists \(R_j>0\) such that \(\partial E_j\subset B_{\mathbb H}(0,R_j)\).
If \(E_j\) were unbounded, then the connectedness of \(\mathbb B^n\setminus B_{\mathbb H}(0,R_j)\) would imply that the whole exterior region
\[
    \mathbb B^n\setminus B_{\mathbb H}(0,R_j)
\]
is contained in \(E_j\).
In particular, \(E_j\cap V_j\) would contain an unbounded connected portion of \(V_j\).
Since \(I_jx=x\) for \(x\in V_j\), the odd symmetry \eqref{eq:odd-symmetry-j} gives \(u'=0\) on \(E_j\cap V_j\).
Thus \(u'\) vanishes on an unbounded connected portion of a totally geodesic hypersurface.
By Lemma~\ref{lem:analytic-continuation-hypersurface}, this produces an unbounded Dirichlet totally geodesic hypersurface of \(u'\), contradicting Lemma~\ref{lem:mainip1}.
Therefore, \(E_j\) is bounded.

Since \(\rho(\gamma(t))\to+\infty\), the curve \(\gamma\) must leave \(E_j\).
Define
\[
    \widetilde t_{j+1}
    :=
    \inf\left\{
    t>t_j:\ \rho(\gamma(t),x_j)=r_0
    \right\}.
\]
Then \(\widetilde t_{j+1}\) is finite, and by \eqref{eq:ball-in-G2},
\[
    \gamma(t)\in E_j
    \qquad
    \text{for }t_j\leq t\leq \widetilde t_{j+1}.
\]
Since \(E_j\) is bounded and \(\gamma\) is proper, there exists a first exit time \(s_{j+1}>\widetilde t_{j+1}\) such that \(\gamma(s_{j+1})\in\partial E_j\).

The boundary \(\partial E_j\) is contained in finitely many totally geodesic faces of \(\partial\mathcal P'\) and their reflected images under \(I_j\).
Since \(\gamma\) stays a positive distance away from \(\partial\mathcal P'\), the point \(\gamma(s_{j+1})\) does not belong to \(\partial\mathcal P'\).
After a standard small perturbation of \(\gamma\), preserving the properties above, we may assume that \(\gamma(s_{j+1})\) lies in the relative interior of an \((n-1)\)-dimensional reflected totally geodesic face.
Denote this reflected face by \(\mathcal F_{j+1}\), and choose a nonempty relatively open subset
\[
    \Gamma_{j+1}\subset \mathcal F_{j+1}\cap G'
\]
containing \(\gamma(s_{j+1})\).

On \(\Gamma_{j+1}\), one has \(u'=0\).
Indeed, \(\mathcal F_{j+1}\) is the reflected image \(I_j(\mathcal F)\) of a totally geodesic face \(\mathcal F\subset\partial\mathcal P'\).
Since \(u'=0\) on \(\mathcal F\) by the \(\mathsf D\)-type boundary condition, the odd relation \eqref{eq:odd-symmetry-j} gives \(u'=0\) on \(I_j(\mathcal F)\cap E_j\), hence on \(\Gamma_{j+1}\).

Let \(V_{j+1}\) be the totally geodesic hypersurface containing \(\mathcal F_{j+1}\).
Applying Lemma~\ref{lem:analytic-continuation-hypersurface} to \(u'\) in \(G'\), we obtain \(u'=0\) on the connected component of \(V_{j+1}\cap G'\) which contains \(\Gamma_{j+1}\).
We denote this component by \(\Pi_{j+1}\).
Then \(\Pi_{j+1}\) is a Dirichlet totally geodesic hypersurface of \(u'\).

Since \(\gamma(s_{j+1})\in\Pi_{j+1}\), the set
\[
    \{t>0:\gamma(t)\in\Pi_{j+1}\}
\]
is nonempty.
It is bounded because \(\Pi_{j+1}\) is bounded by Lemma~\ref{lem:mainip1} and \(\gamma\) is proper.
We choose
\[
    t_{j+1}
    :=
    \max\{t>0:\gamma(t)\in\Pi_{j+1}\},
\]
and set \(x_{j+1}:=\gamma(t_{j+1})\).
Then \(t_{j+1}>s_{j+1}>\widetilde t_{j+1}>t_j\).
Moreover,
\[
\begin{split}
\operatorname{length}_{\mathbb H}
\left(
\gamma|_{[t_j,t_{j+1}]}
\right)
&\geq
\operatorname{length}_{\mathbb H}
\left(
\gamma|_{[t_j,\widetilde t_{j+1}]}
\right)
\\
&\geq r_0.
\end{split}
\]
Thus, for every \(j\geq1\),
\begin{equation}\label{eq:length-lower-bound}
    \operatorname{length}_{\mathbb H}
    \left(
    \gamma|_{[t_j,t_{j+1}]}
    \right)
    \geq r_0.
\end{equation}

The hypersurfaces \(\Pi_j\) are pairwise distinct.
Indeed, \(t_j\) is chosen to be the last parameter for which \(\gamma\) meets \(\Pi_j\).
If \(\Pi_{j+1}=\Pi_\ell\) for some \(\ell\leq j\), then \(\gamma(t_{j+1})\in\Pi_\ell\) with \(t_{j+1}>t_\ell\), contradicting the definition of \(t_\ell\).

\medskip
\noindent\emph{Step 5. Iteration and contradiction.}

Repeating Step~4, we obtain pairwise distinct Dirichlet totally geodesic hypersurfaces
\[
    \Pi_1,\Pi_2,\Pi_3,\ldots
\]
and a strictly increasing sequence \(0=t_1<t_2<t_3<\cdots\) such that
\[
    x_j:=\gamma(t_j)\in\Pi_j,
    \qquad j=1,2,3,\ldots,
\]
and the lower bound \eqref{eq:length-lower-bound} holds for every \(j\).

We now use the boundedness of the Dirichlet set from Lemma~\ref{lem:mainip1}.
Since \(x_j\in\Pi_j\subset\mathcal D_{u'}\), and \(\mathcal D_{u'}\) is bounded, the sequence \(\{x_j\}\) is contained in a fixed bounded subset of \(G'\).
On the other hand, \(\rho(\gamma(t))\to+\infty\) as \(t\to+\infty\).
Hence the increasing sequence \(\{t_j\}\) cannot tend to \(+\infty\).
Therefore \(t_j\to t_0<+\infty\).
Since \(\gamma\) is \(C^1\), its hyperbolic speed is bounded in a neighborhood of \(t_0\).
Consequently,
\[
\begin{split}
\operatorname{length}_{\mathbb H}
\left(
\gamma|_{[t_j,t_{j+1}]}
\right)
&=
\int_{t_j}^{t_{j+1}}|\dot\gamma(t)|_{\mathbb H}\,dt
\\
&\longrightarrow0
\qquad
\text{as }j\to\infty.
\end{split}
\]
This contradicts \eqref{eq:length-lower-bound}.
Thus the assumption \(\mathcal P\neq\mathcal P'\) is false, and hence \(\mathcal P=\mathcal P'\).

The proof is complete.
\end{proof}
\subsection{Proof of Theorem~\ref{thm:mainip2}}

We first prove the boundedness of the Neumann set, which is the key ingredient in the proof of Theorem~\ref{thm:mainip2}.

\begin{lemma}\label{lem:neumann-set-bounded}
Let
\[
\mathbf U=(u_1,\ldots,u_m)
\]
be the exterior solutions to \eqref{eq:hhelm1} corresponding to the incoming Helgason modes
\[
u^i_\ell=e_{2\lambda_0,\xi_\ell},
\qquad
\xi_\ell\in\mathbb S^{n-1},
\qquad
\ell=1,\ldots,m.
\]
Assume that the boundary labels are not simultaneously contained in the ideal boundary defined in Definition~\ref{def:asymptotic-boundary} of any totally geodesic hypersurface, that is,
\[
\{\xi_1,\ldots,\xi_m\}
\not\subset
\partial_\infty V
\]
for every totally geodesic hypersurface \(V\subset\mathbb B^n\). Then, the Neumann set \(\mathcal N_{\mathbf U}\) and all Neumann totally geodesic hypersurfaces of \(\mathbf U\) are bounded.
\end{lemma}

\begin{proof}
We first prove that no Neumann totally geodesic hypersurface can be unbounded.

Suppose, to the contrary, that there exists an unbounded Neumann totally geodesic hypersurface
\[
    \Pi\subset V\cap G .
\]
Then \(\partial_{\nu_{\mathbb H}}u_\ell=0\) on \(\Pi\) for every \(\ell=1,\ldots,m\).
Equivalently,
\[
    \partial_{\nu_{\mathbb H}}u^i_\ell
    =
    -\partial_{\nu_{\mathbb H}}u^s_\ell
    \qquad
    \text{on } \Pi .
\]

We claim that this forces
\[
    \xi_\ell\in \partial_\infty V,
    \qquad
    \ell=1,\ldots,m,
\]
where \(\partial_\infty V\) is the ideal boundary of \(V\) defined in Definition~\ref{def:asymptotic-boundary}.
{
It is enough to prove the claim for the standard totally geodesic hypersurface
\[
    V_0:=\{x\in\mathbb B^n:x_n=0\}.
\]
Indeed, choose a hyperbolic isometry \(M\) with \(M(V)=V_0\) and set
\(\widetilde u_\ell:=u_\ell\circ M^{-1}\).  If the unit hyperbolic normals are chosen
compatibly, \eqref{eq:normal-covariance-isometry} gives
\[
    \left(\partial_{\nu_{\mathbb H,V_0}}\widetilde u_\ell\right)\circ M
    =\partial_{\nu_{\mathbb H,V}}u_\ell,
\]
so the homogeneous hyperbolic Neumann condition is preserved.  The shifted
Helmholtz operator and the outgoing class are also invariant under \(M\), while
a Helgason kernel is transformed into a nonzero constant multiple of the
kernel with transformed boundary label.  Finally,
\(M_\infty(\partial_\infty V)=\partial_\infty V_0\), where \(M_\infty\)
denotes the boundary action induced by \(M\).  Thus the claim for
\(V_0\) implies the claim for every \(V\).
}

Let \(V=V_0=\{x_n=0\}\). Then
\[
    \partial_\infty V_0
    =
    \{\eta\in\mathbb S^{n-1}:\eta_n=0\}.
\]
Assume that, for some \(\ell\), \(\xi_\ell\notin\partial_\infty V_0\). Then \((\xi_\ell)_n\neq0\).
Choose \(R>0\) sufficiently large so that \(\mathcal P\subset B_{\mathbb H}(R)\).
Since \(\Pi\subset V_0\cap G\) is unbounded, it intersects \(V_0\setminus B_{\mathbb H}(R)\).
Moreover,
\[
    V_0\setminus B_{\mathbb H}(R)\subset V_0\cap G.
\]
Let \(C\) be the connected component of \(V_0\setminus B_{\mathbb H}(R)\) which intersects \(\Pi\).
Since \(\Pi\) is the connected component of \(V_0\cap G\) on which the Neumann condition holds, we have \(C\subset \Pi\).

Recall that, in the Poincar\'e ball model,
\[
    \rho(x)=\log\frac{1+|x|}{1-|x|},
    \qquad
    |x|=\tanh\left(\frac{\rho(x)}{2}\right).
\]
Since \(V_0=\{x\in\mathbb B^n:x_n=0\}\), the unbounded component \(C\) contains a radial geodesic ray of the form
\[
    x_\rho=\tanh\left(\frac{\rho}{2}\right)\eta,
    \qquad
    \rho>R_1,
\]
for some \(R_1>R\) and some \(\eta\in\mathbb S^{n-1}\) with \(\eta_n=0\).
In particular, \(\eta\in \overline{V_0}^{\,\overline{\mathbb B^n}}\cap\mathbb S^{n-1}\).
Since \(\eta_n=0\) and \((\xi_\ell)_n\neq0\), we have \(\eta\neq\xi_\ell\).

For simplicity, write \(\xi=\xi_\ell\) and \(u^i=e_{2\lambda_0,\xi}\).
{Under \eqref{eq:outnormal}, on \(V_0\), the hyperbolic normal derivative is}
\[
    \partial_{\nu_{\mathbb H}}
    =
    \frac{1-|x|^2}{2}\partial_{x_n}.
\]
By \eqref{eq:e},
\[
    u^i(x)
    =
    \left(
    \frac{\sqrt{1-|x|^2}}{|x-\xi|}
    \right)^{n-1+2\lambda_0\mathrm i}.
\]
A direct calculation at points \(x\in V_0\) gives
\[
    \partial_{x_n}u^i(x)
    =
    (n-1+2\lambda_0\mathrm i)
    \frac{\xi_n}{|x-\xi|^2}
    u^i(x).
\]
Hence, along \(x_\rho=\tanh(\rho/2)\eta\),
\[
\begin{split}
\partial_{\nu_{\mathbb H}}u^i(x_\rho)
&=
\frac{1-|x_\rho|^2}{2}
(n-1+2\lambda_0\mathrm i)
\frac{\xi_n}{|x_\rho-\xi|^2}
u^i(x_\rho)
\\
&=
A_{\xi,\eta}(\rho)
\left(\cosh\frac{\rho}{2}\right)^{-(n+1)-2\lambda_0\mathrm i},
\end{split}
\]
where
\[
    A_{\xi,\eta}(\rho)
    =
    \frac{n-1+2\lambda_0\mathrm i}{2}\,
    \xi_n
    \left|
    \tanh\left(\frac{\rho}{2}\right)\eta-\xi
    \right|^{-(n+1)-2\lambda_0\mathrm i}.
\]
Since \(\eta\neq\xi\) and \(\xi_n\neq0\), we have
\[
    A_{\xi,\eta}(\rho)\to A_{\xi,\eta}^0\neq0
    \qquad
    \text{as }\rho\to+\infty .
\]

On the other hand, the outgoing correction \(u^s_\ell\) is outgoing.
Its far-field expansion \eqref{eq:us} gives, after differentiating tangentially to the sphere and using elliptic regularity outside a compact set,
\[
    \partial_{\nu_{\mathbb H}}u^s_\ell(x_\rho)
    =
    B_{\ell,\eta}(\rho)
    \left(\cosh\frac{\rho}{2}\right)^{-(n+1)+2\lambda_0\mathrm i},
\]
where \(B_{\ell,\eta}(\rho)\to B_{\ell,\eta}^0\) for some finite complex number \(B_{\ell,\eta}^0\).

Using \(\partial_{\nu_{\mathbb H}}u^i_\ell=-\partial_{\nu_{\mathbb H}}u^s_\ell\) on \(\Pi\), we get
\[
    A_{\xi,\eta}(\rho)
    \left(\cosh\frac{\rho}{2}\right)^{-(n+1)-2\lambda_0\mathrm i}
    +
    B_{\ell,\eta}(\rho)
    \left(\cosh\frac{\rho}{2}\right)^{-(n+1)+2\lambda_0\mathrm i}
    =
    0 .
\]
Multiplying by \(\left(\cosh\frac{\rho}{2}\right)^{n+1+2\lambda_0\mathrm i}\), we obtain
\[
    A_{\xi,\eta}(\rho)
    +
    B_{\ell,\eta}(\rho)
    \left(\cosh\frac{\rho}{2}\right)^{4\lambda_0\mathrm i}
    =
    0.
\]
Letting \(\rho\to+\infty\), this is impossible.
Indeed, if \(B_{\ell,\eta}^0=0\), then the above identity gives \(A_{\xi,\eta}^0=0\), which contradicts \(A_{\xi,\eta}^0\neq0\).
If \(B_{\ell,\eta}^0\neq0\), then \(\left(\cosh\frac{\rho}{2}\right)^{4\lambda_0\mathrm i}\) would have to converge to a finite limit, which is impossible because \(\lambda_0>0\) and
\[
    \log\left(\cosh\frac{\rho}{2}\right)\to+\infty .
\]
Therefore, our assumption \(\xi_\ell\notin\partial_\infty V_0\) is false.

Thus, for the standard hypersurface \(V_0\), one has \(\xi_\ell\in\partial_\infty V_0\) for every \(\ell=1,\ldots,m\).
By the isometric reduction above, the same conclusion holds for a general totally geodesic hypersurface \(V\):
\[
    \xi_\ell\in\partial_\infty V,
    \qquad
    \ell=1,\ldots,m.
\]
This contradicts the assumption that
\[
    \{\xi_1,\ldots,\xi_m\}
    \not\subset
    \partial_\infty V
\]
for every totally geodesic hypersurface \(V\subset\mathbb B^n\).
Hence no Neumann totally geodesic hypersurface can be unbounded.

It remains to prove that the Neumann set \(\mathcal N_{\mathbf U}\) is bounded.
Choose \(R>0\) such that \(\mathcal P\subset B_{\mathbb H}(R)\).
If there existed \(x\in\mathcal N_{\mathbf U}\) with \(\rho(x)>R\), then, by the definition of \(\mathcal N_{\mathbf U}\), there would exist a totally geodesic hypersurface \(V\ni x\) such that \(\partial_{\nu_{\mathbb H}}\mathbf U=0\) locally near \(x\) on \(V\cap G\).
By Lemma~\ref{lem:analytic-continuation-hypersurface}, this Neumann condition extends to the connected component of \(V\cap G\) containing \(x\).
Since
\[
    V\setminus B_{\mathbb H}(R)\subset V\cap G,
\]
the component of \(V\setminus B_{\mathbb H}(R)\) containing \(x\) is unbounded.
Hence we obtain an unbounded Neumann totally geodesic hypersurface, contradicting the first part of the proof.
Therefore \(\mathcal N_{\mathbf U}\subset B_{\mathbb H}(R)\), and the Neumann set is bounded.
The boundedness of all Neumann totally geodesic hypersurfaces follows immediately, since each of them is contained in \(\mathcal N_{\mathbf U}\).

The proof is complete.
\end{proof}

\begin{proof}[Proof of Theorem~\ref{thm:mainip2}]
Set \(m:=n+1\). By Remark~\ref{rem:affine-independence}, the prescribed
boundary labels are not simultaneously contained in the ideal boundary of
any totally geodesic hypersurface.

We argue by contradiction.
Assume that \(\mathcal P\neq\mathcal P'\), while the far-field patterns corresponding to the prescribed incoming Helgason modes
\[
u^i_\ell=e_{2\lambda_0,\xi_\ell},
\qquad \ell=1,\ldots,m,
\]
coincide.

Let
\[
\mathbf U=(u_1,\ldots,u_m),
\qquad
\mathbf U'=(u'_1,\ldots,u'_m)
\]
be the corresponding exterior solutions.
By the hyperbolic Rellich theorem and unique continuation, we have
\[
\mathbf U=\mathbf U'
\]
in the unbounded connected component of
\[
\mathbb B^n\setminus
\left(
\overline{\mathcal P}\cup\overline{\mathcal P'}
\right).
\]

As in the proof of Theorem~\ref{thm:mainip1}, the assumption \(\mathcal P\neq\mathcal P'\) yields a totally geodesic face
\[
    \mathcal F_1\subset\partial\mathcal P
\]
and a nonempty relatively open subset \(\Gamma_1\subset\mathcal F_1\cap G'\), where
\[
    G':=\mathbb B^n\setminus\overline{\mathcal P'},
\]
such that \(\Gamma_1\) lies on the common exterior boundary.
Since \(\mathcal P\) is \(\mathsf N\)-type,
\[
\partial_{\nu_{\mathbb H}}\mathbf U=0
\qquad
\text{on }\Gamma_1.
\]
Using \(\mathbf U=\mathbf U'\) in the common exterior, we obtain
\[
\partial_{\nu_{\mathbb H}}\mathbf U'=0
\qquad
\text{on }\Gamma_1.
\]
By Lemma~\ref{lem:analytic-continuation-hypersurface}, this condition extends to the connected component of \(V_1\cap G'\) containing \(\Gamma_1\), where \(V_1\) is the totally geodesic hypersurface containing \(\mathcal F_1\).
Denote this component by \(\Pi_1\).
Then \(\Pi_1\) is a Neumann totally geodesic hypersurface of \(\mathbf U'\).

	Starting from \(\Pi_1\), we repeat the reflected-continuation construction used in the proof of Theorem~\ref{thm:mainip1}, with \(\mathcal D_{u'}\) replaced by \(\mathcal N_{\mathbf U'}\) and with the Dirichlet reflection principle replaced by the Neumann reflection principle.
The geometric construction is unchanged: at each step the curve exits the
current reflected region through a reflected totally geodesic face.
{
The Neumann condition is transported to every reflected face as follows.
In each reflected region \(E_j\), Lemma~\ref{lem:reflection-neumann} gives
\[
    \mathbf U'(x)=\mathbf U'(I_jx).
\]
In particular, \eqref{eq:normal-reflection} gives
\(\partial_{\nu_{\mathbb H,V_j}}\mathbf U'=0\) on \(E_j\cap V_j\).  If
\(\mathcal F\) is an original Neumann face and
\(\mathcal F_{j+1}:=I_j(\mathcal F)\), choose the unit hyperbolic normals by
\[
    \nu_{\mathbb H,\mathcal F_{j+1}}(I_jx)
    =d(I_j)_x\nu_{\mathbb H,\mathcal F}(x).
\]
Then the covariance formula \eqref{eq:normal-covariance-isometry} and the even
symmetry imply
\[
    \left(
      \partial_{\nu_{\mathbb H,\mathcal F_{j+1}}}\mathbf U'
    \right)(I_jx)
    =
    \partial_{\nu_{\mathbb H,\mathcal F}}(\mathbf U'\circ I_j)(x)
    =
    \partial_{\nu_{\mathbb H,\mathcal F}}\mathbf U'(x)
    =0.
\]
Thus the homogeneous hyperbolic Neumann condition holds on the reflected face
with no additional zeroth-order term.
}
The same boundedness argument, now using Lemma~\ref{lem:neumann-set-bounded}, shows that each \(E_j\) is bounded and produces a new Neumann totally geodesic hypersurface \(\Pi_{j+1}\).

Therefore we obtain pairwise distinct Neumann totally geodesic hypersurfaces
\[
\Pi_1,\Pi_2,\Pi_3,\ldots
\]
and a strictly increasing sequence
\[
0=t_1<t_2<t_3<\cdots
\]
such that
\[
x_j:=\gamma(t_j)\in\Pi_j\subset\mathcal N_{\mathbf U'},
\]
and
\[
\operatorname{length}_{\mathbb H}
\big(\gamma|_{[t_j,t_{j+1}]}\big)
\ge r_0>0.
\]

Since \(\mathcal N_{\mathbf U'}\) is bounded by Lemma~\ref{lem:neumann-set-bounded}, the sequence \(\{x_j\}\) is bounded.
However,
\[
\rho(\gamma(t))\to+\infty
\qquad
\text{as }t\to+\infty.
\]
Hence \(t_j\to t_0<+\infty\).
Since \(\gamma\) is \(C^1\),
\[
\operatorname{length}_{\mathbb H}
\big(\gamma|_{[t_j,t_{j+1}]}\big)
\to0,
\]
which contradicts the lower bound above.
Therefore
\[
\mathcal P=\mathcal P'.
\]

The proof is complete.
\end{proof}

\section{Quantitative preliminaries in hyperbolic space for stability}\label{sec:quant-prelim}

The preceding sections established the analytic framework for hyperbolic scattering, the reflection principles for totally geodesic hypersurfaces, and the corresponding uniqueness results for totally geodesic defects.
We now lay the groundwork for the quantitative stability analysis by specifying the admissible classes and developing uniform estimates for the associated direct problems. 
More precisely, this section introduces admissible classes of totally geodesic defects and proves the uniform compactness properties, \emph{a priori} bounds, and decay estimates for the corresponding direct problems. 
These estimates will serve as the main analytic input for the subsequent stability proof of the inverse problem.

\subsection{Classes of admissible defects}\label{subsec:admissible-class}

For the qualitative uniqueness results, it was sufficient to work with the broad class of totally geodesic defects introduced in Section~\ref{sec:reflection-principles}.
The stability analysis, however, requires quantitative a priori control of the geometry, in order to obtain compactness and uniform estimates for the corresponding direct problems.
We therefore introduce the admissible classes of defects that will be used throughout the stability analysis.
Before giving the definitions, we fix some notation.

\medskip
We recall the notation for balls in the Poincar\'e ball model \(\mathbb B^n\).
We denote by \(B_{\mathbb H}(x,r)\) the hyperbolic ball centered at \(x\) with radius \(r\), where the radius is measured using the hyperbolic distance \(\rho(\cdot,\cdot)\) defined in \eqref{eq:distance}.
As above, \(B_{\mathbb H}(r)=B_{\mathbb H}(0,r)\) for \(r>0\).
For a subset \(A\subset\mathbb B^n\), we also set
\[
    B_{\mathbb H}(A,r):=\bigcup_{x\in A}B_{\mathbb H}(x,r).
\]

We next define the hyperbolic cones used below.
For \(x\in\mathbb B^n\), let \(S_x\mathbb B^n\) denote the unit sphere in \(T_x\mathbb B^n\) with respect to the hyperbolic metric \eqref{eq:hyper metric}.
Given \(x\in\mathbb B^n\), \(\omega\in S_x\mathbb B^n\), \(r>0\), and \(0<\theta<\pi/2\), the \emph{hyperbolic open cone} with vertex \(x\), axis \(\omega\), hyperbolic radius \(r\), and aperture angle \(\theta\) is defined by
\[
    \mathcal C_{\mathbb H}(x,\omega,r,\theta)
    :=
    \left\{
    \exp_x(t\zeta):
    0<t<r,\ 
    \zeta\in S_x\mathbb B^n,\ 
    \langle \zeta,\omega\rangle_{g_{\mathbb H,x}}>\cos\theta
    \right\}.
\]
Here \(\exp_x\) is the hyperbolic exponential map at \(x\), and \(\langle\cdot,\cdot\rangle_{g_{\mathbb H,x}}\) is the inner product on \(T_x\mathbb B^n\) induced by the hyperbolic metric.
If \(V\subset\mathbb B^n\) is a totally geodesic hypersurface, \(x\in V\),
\(\tau\in S_xV\), \(r>0\), and \(0<\theta<\pi/2\), we define the relative
hyperbolic cone in \(V\) by
\[
    \mathcal C_{\mathbb H}^{V}(x,\tau,r,\theta)
    :=
   \mathcal C_{\mathbb H}(x,\tau,r,\theta)\cap V.
\]

Equivalently, let \(\gamma_{x,y}\) be the unit-speed geodesic from \(x\) to \(y\), parametrized by \(\gamma_{x,y}(0)=x\).
Then, \(y\in \mathcal C_{\mathbb H}(x,\omega,r,\theta)\) if and only if
\[
    0<\rho(x,y)<r
    \quad\text{and}\quad
    \left\langle \dot\gamma_{x,y}(0),\omega\right\rangle_{g_{\mathbb H,x}}>\cos\theta.
\]
We call \(x\) the vertex of the cone.
The geodesic \(t\mapsto\exp_x(t\omega)\), \(0<t<r\), is called the axis geodesic.
The parameter \(r\) is the radius, and \(\theta\) is the aperture angle.
The opposite cone is denoted by \(\mathcal C_{\mathbb H}(x,-\omega,r,\theta)\).

We shall also use the following convention for angles between totally geodesic hypersurfaces.
If \(V_1\) and \(V_2\) are two totally geodesic hypersurfaces meeting at a point \(x\), their angle at \(x\) is the angle between the tangent hyperplanes \(T_xV_1\) and \(T_xV_2\) in \(T_x\mathbb B^n\), computed with respect to the hyperbolic metric \eqref{eq:hyper metric}.
Equivalently, if \(\nu_1\) and \(\nu_2\) are unit normal vectors to \(V_1\) and \(V_2\) at \(x\), then
\[
    \angle_{\mathbb H}(V_1,V_2)(x)
    :=
    \arccos |\langle \nu_1,\nu_2\rangle_{g_{\mathbb H,x}}|.
\]
This definition is independent of the choices of the normal directions.
Since the Poincar\'e metric is conformal to the Euclidean metric, this angle
coincides with the Euclidean angle in the ball model.
When we write
\[
    \angle_{\mathbb H}(V_1,V_2)\geq \alpha,
\]
we mean that this angle is at least \(\alpha\) at every point of
\(V_1\cap V_2\).

\medskip
We define the admissible classes of totally geodesic defects for stability analysis.
These classes include both bulk components and hypersurface-supported components as in Definition~\ref{def:totally-geodesic-defect}.
The definitions are quantitative and rely only on geometric parameters.
The minimal face size is controlled by \(h>0\), while all other constants are fixed a priori.

\begin{definition}\label{def:admissible-tg-defects}
Fix positive constants
\(R,r_0,r_{\rm cone},L,c_0,h_0\), angles
\(0<\alpha_0,\theta_0,\theta_{\rm cone}<\pi/2\), an integer \(N_0\), and a nondecreasing
left-continuous function
\(\delta:(0,+\infty)\to(0,+\infty)\).
Let \(0<h\leq h_0\).
A nonempty compact set \(\mathcal P\subset\mathbb B^n\) belongs to \(\mathcal A_{\mathbb H}^h\) if:

\begin{enumerate}
    \item \textbf{Boundedness and connected exterior.}  
    \(\mathcal P\subset \overline{B_{\mathbb H}(R)}\), and \(G := \mathbb B^n \setminus \mathcal P\) is connected.

    \item \textbf{Totally geodesic face decomposition.}  
    \(\partial\mathcal P = \bigcup_{j=1}^{N_{\mathcal P}} \mathcal F_j\),
    \(1\leq N_{\mathcal P}\leq N_0\),
    where each \(\mathcal F_j\) is a totally geodesic face contained in a totally geodesic hypersurface \(V_j\subset \mathbb B^n\) and is the relative closure in \(V_j\) of a connected relatively open subset of \(V_j\).

    \item \textbf{Uniform \(h\)-regularity of faces.}
Each \(\mathcal F_j\) is the relative closure of a Lipschitz domain in \(V_j\), with Lipschitz constant \(L\) and Lipschitz scale at least \(h\).
In addition, there exist \(y_j\in\mathcal F_j\) and a unit tangent vector \(\tau_j\in S_{y_j}V_j\) such that
\[
    \mathcal C_{\mathbb H}^{V_j}(y_j,\tau_j,c_0h,\theta_0)
    \subset \mathcal F_j .
\]

    \item \textbf{Uniform exterior cone condition.}
For every \(x\in\partial\mathcal P\), there exists
\(\omega_x\in S_x\mathbb B^n\) such that
\[
    \mathcal C_{\mathbb H}
    (x,\omega_x,r_{\rm cone},\theta_{\rm cone})
    \subset
    \mathbb B^n\setminus\mathcal P .
\]

       \item \textbf{Controlled incidence of faces.}
The relative interiors of the faces are pairwise disjoint.
If \(\mathcal F_i\cap\mathcal F_j\neq\emptyset\), then
\(\mathcal F_i\cap\mathcal F_j\) is contained in the relative boundaries of both faces.
Moreover, whenever \(V_i\neq V_j\) and
\(\mathcal F_i\cap\mathcal F_j\neq\emptyset\), one has
\[
    \angle_{\mathbb H}(V_i,V_j)\geq \alpha_0.
\]
We use the convention that \(\rho(x,\emptyset)=+\infty\).
Finally, for every \(x\in\mathcal F_i\),
\[
    \rho
    \left(
        x,\bigcup_{j\neq i}\mathcal F_j
    \right)
    \geq
    c_0\,
    \rho
    \left(
        x,\partial_{V_i}\mathcal F_i
    \right),
\]
whenever
\[
    c_0\,
    \rho
    \left(
        x,\partial_{V_i}\mathcal F_i
    \right)
    < r_0 .
\]
Here \(\partial_{V_i}\mathcal F_i\) denotes the relative boundary of \(\mathcal F_i\) in \(V_i\).

    \item \textbf{Uniform exterior connectedness.}
For every \(t>0\), for any \(x_1,x_2\in G\) such that
\[
    B_{\mathbb H}(x_1,t)\subset G,
    \qquad
    B_{\mathbb H}(x_2,t)\subset G,
\]
and for every \(0<s<\delta(t)\), there exists a \(C^1\) curve
\(\gamma:[0,1]\to G\) with
\[
    \gamma(0)=x_1,
    \qquad
    \gamma(1)=x_2,
    \qquad
    B_{\mathbb H}(\gamma,s)\subset G .
\]
\end{enumerate}
\end{definition}

\medskip
\noindent
The \(\mathsf N\)-type admissible class is obtained from \(\mathcal A_{\mathbb H}^h\) by imposing an additional \emph{local hyperbolic collar condition}, used only for Neumann boundary conditions.

\begin{definition}\label{def:hard-admissible-tg-defects}
Fix, in addition, a nondecreasing left-continuous function
\(\omega_{\rm col}:(0,+\infty)\to(0,+\infty)\).
Let \(0<h\leq h_0\).
A defect \(\mathcal P\in\mathcal A_{\mathbb H}^h\) belongs to the class \(\mathcal B_{\mathbb H}^h\) if the following local hyperbolic collar condition holds.

For every \(x\in\partial\mathcal P\) and every connected component \(U\) of
\((\mathbb B^n\setminus\mathcal P)\cap B_{\mathbb H}(x,r_0)\) such that
\(x\in\partial U\), there exist an open set
\(U'\subset\mathbb B^n\setminus\mathcal P\) and a
\(W^{1,\infty}\)-homeomorphism whose inverse is also \(W^{1,\infty}\)
\[
    T:Q^+\to U',
    \qquad
    Q^+:=(-1,1)^{n-1}\times(0,1),
\]
such that the following conditions are satisfied.
We set $\Gamma:=[-1,1]^{n-1}\times\{0\}$.

\begin{enumerate}
	\item \textbf{Local half-cylinder parametrization.}
One has
\[
    U\cap B_{\mathbb H}(x,r_0/2)
    \subset
    U'
    \subset
    \mathbb B^n\setminus\mathcal P.
\]

	\item \textbf{Boundary correspondence.}
The map \(T\) admits a Lipschitz extension to \(\overline{Q^+}\), still denoted by \(T\).
Moreover,
\[
    T(0)=x,
\]
and
\[
    \partial U\cap B_{\mathbb H}(x,r_0/2)
    \subset
    T(\Gamma)
    \subset
    \partial(\mathbb B^n\setminus\mathcal P).
\]

\item\textbf{Quantitative collar separation.}
For every \(0<s<r_0/2\) and every
\[
    y\in U\cap B_{\mathbb H}(x,r_0/2-s),
\]
one has
\[
    \operatorname{dist}_{\mathbb R^n}
    \left(
        T^{-1}(y),\partial Q^+\setminus\Gamma
    \right)
    \geq
    \omega_{\rm col}(s).
\]

\item \textbf{Uniform chart bounds.}
The \(W^{1,\infty}\)-norms and the Lipschitz constants of \(T\) and \(T^{-1}\), computed in local hyperbolic charts, are bounded by \(L\).
The function \(\omega_{\rm col}\) is regarded as part of the a priori data of the class \(\mathcal B_{\mathbb H}^h\).
\end{enumerate}
\end{definition}

\begin{remark}
The local hyperbolic collar condition in Definition~\ref{def:hard-admissible-tg-defects} is imposed only in the \(\mathsf N\)-type case.
Here a hyperbolic collar means a one-sided neighborhood of the interaction surface inside the exterior domain, measured with respect to hyperbolic balls and parametrized in local hyperbolic charts by a fixed half-cylinder.
For \(\mathsf D\)-type defects, the homogeneous Dirichlet condition is stable under zero extension.
For \(\mathsf N\)-type defects, this collar condition gives uniform local control of the exterior hyperbolic geometry near the boundary, and is used to obtain Mosco convergence and uniform Sobolev estimates for the Neumann direct problem.
\end{remark}

\medskip
We next introduce three distances for comparing admissible defects in the hyperbolic space.
The quantities \(\widetilde d\) and \(\widehat d\) measure the hyperbolic Hausdorff distances between the defects and between their boundaries, respectively.
The third quantity \(d\) is a modified boundary distance adapted to the
first-contact propagation argument in the common exterior.

\begin{definition}\label{def:hyperbolic-distances}
Let \(\mathcal P,\mathcal P'\) be two nonempty compact defects contained in
\(\overline{B_{\mathbb H}(R)}\).
First, we define the Hausdorff distance in hyperbolic space between the defects by
\begin{equation}\label{eq:d1}
    \widetilde d
    :=
    d_{\mathcal H}^{\mathbb H}(\mathcal P,\mathcal P')
    :=
    \max\left\{
    \sup_{x\in\mathcal P}
    \rho(x,\mathcal P'),
    \sup_{y\in\mathcal P'}
    \rho(y,\mathcal P)
    \right\}.
\end{equation}
Second, we define the Hausdorff distance in hyperbolic space between their boundaries by
\begin{equation}\label{eq:d2}
    \widehat d
    :=
    d_{\mathcal H}^{\mathbb H}(\partial\mathcal P,\partial\mathcal P')
    :=
    \max\left\{
    \sup_{x\in\partial\mathcal P}
    \rho(x,\partial\mathcal P'),
    \sup_{y\in\partial\mathcal P'}
    \rho(y,\partial\mathcal P)
    \right\}.
\end{equation}
Third, we define the modified boundary distance by
\begin{equation}\label{eq:d3}
    d
    :=
    \max\left\{
    \sup_{x\in\partial\mathcal P\setminus\mathcal P'}
    \rho(x,\partial\mathcal P'),
    \sup_{y\in\partial\mathcal P'\setminus\mathcal P}
    \rho(y,\partial\mathcal P)
    \right\}.
\end{equation}
Here, the supremum over \(\emptyset\) is understood to be zero, and \(\rho\) is defined
in the sense of \eqref{eq:distH}.
\end{definition}

We introduce the exterior connectedness modulus $\delta$, which controls the thickness of connecting curves in the exterior domain. 
This modulus is part of the a priori data of both the \(\mathsf D\)-type class $\mathcal A_{\mathbb H}^h$ and the \(\mathsf N\)-type class $\mathcal B_{\mathbb H}^h$ in Definitions~\ref{def:admissible-tg-defects} and~\ref{def:hard-admissible-tg-defects}.

Without loss of generality, we assume
$\delta(t) \le t$ for all $t>0$, which can always be achieved by
replacing $\delta$ with $\min\{t,\delta(t)\}$.
We also set $\delta(0) := 0$.
For later reference, we define its generalized inverse by
\begin{equation}\label{eq:ext}
	\delta^{-1}(s) := \sup\{ t \in [0,2R] : \delta(t) \le s\}, \qquad s \ge 0,
\end{equation}
which is nondecreasing and depends only on the a priori exterior connectedness data.
This function will be used to relate the modified boundary distance \eqref{eq:d3} to the Hausdorff distances \eqref{eq:d1} and \eqref{eq:d2} in the stability estimates, see Proposition~\ref{prop:distance-equivalence-hyperbolic}.

\begin{proposition}\label{prop:distance-equivalence-hyperbolic}
Let \(\mathcal P,\mathcal P'\) belong to the same admissible class
\(\mathcal A_{\mathbb H}^h\) or \(\mathcal B_{\mathbb H}^h\), and let
\(d\), \(\widehat d\), and \(\widetilde d\) be defined as in
Definition~\ref{def:hyperbolic-distances}.
Assume the standing normalization \(\delta(t)\le t\) for \(t>0\), and let
\(\delta^{-1}\) be the generalized inverse defined in \eqref{eq:ext}.
Then there exists a constant \(c_*>0\), depending only on the a priori geometric
data of the class, such that
\begin{equation}\label{eq:Hd1}
	   d \le \widehat d,
    \qquad
    c_*\widehat d
    \le \widetilde d
    \le \delta^{-1}(d)
    \le \delta^{-1}(\widehat d).
\end{equation}
If the exterior connectedness modulus \(\delta\) fixed in the admissible class is
linear, namely \(\delta(t)\ge c_{\rm ec}t\) for \(0<t\le 2R\), then
\begin{equation}\label{eq:Hd2}
	  c_*\widehat d
    \le \widetilde d
    \le c_{\rm ec}^{-1}d.
\end{equation}
In particular, in the linear case, \(d\), \(\widehat d\), and \(\widetilde d\) are
linearly comparable up to constants depending only on the a priori data.
\end{proposition}

\begin{proof}
The inequality \(d\le \widehat d\) follows directly from \eqref{eq:d2} and \eqref{eq:d3}.
Indeed, each supremum in the definition of \(d\) is taken over a subset of the corresponding set used in the definition of \(\widehat d\).
Taking the maximum gives \(d\le \widehat d\).

\smallskip
We prove the left side \(c_*\widehat d\le\widetilde d\) of \eqref{eq:Hd1} and \eqref{eq:Hd2}.
If \(\widehat d=0\), this case is trivial.
After interchanging \(\mathcal P\) and \(\mathcal P'\) if necessary, choose \(x\in\partial\mathcal P\) such that \(\widehat d=\rho(x,\partial\mathcal P')\).
If \(\widetilde d\ge \widehat d/2\), the claim follows.
Therefore, we assume \(\widetilde d<\widehat d/2\).

Since \(x\in\mathcal P\), there exists \(x'\in\mathcal P'\) such that
\[
    \rho(x,x')\le \widetilde d<\widehat d/2.
\]
Hence \(B_{\mathbb H}(x,\widehat d)\) meets \(\mathcal P'\).
By the definition of \(\widehat d\), this ball does not meet \(\partial\mathcal P'\).
Since \(B_{\mathbb H}(x,\widehat d)\) is connected and disjoint from
\(\partial\mathcal P'\), it must lie entirely in one connected component of
\(\mathbb B^n\setminus\partial\mathcal P'\).
As it intersects \(\mathcal P'\), it cannot lie in the exterior of \(\mathcal P'\).
Therefore
\[
    B_{\mathbb H}(x,\widehat d)
    \subset \operatorname{int}(\mathcal P').
\]

By the uniform exterior cone condition for \(\mathcal P\) at \(x\), there is a hyperbolic cone
\[
    \mathcal C_{\mathbb H}(x,\omega_x,r_{\rm cone},\theta_{\rm cone})
    \subset \mathbb B^n\setminus\mathcal P .
\]
The elementary interior-ball property of hyperbolic cones gives \(a_{\rm cone}\in(0,1)\), depending only on \(r_{\rm cone}\) and \(\theta_{\rm cone}\), such that for \(0<\ell\le r_{\rm cone}/2\) and \(y=\exp_x(\ell\omega_x)\),
\[
    B_{\mathbb H}(y,a_{\rm cone}\ell)
    \subset \mathbb B^n\setminus\mathcal P .
\]
Choose
\[
    \ell=\frac12\min\{\widehat d,r_{\rm cone}\},
    \qquad
    y=\exp_x(\ell\omega_x).
\]
Since \(\ell<\widehat d\), we have \(y\in B_{\mathbb H}(x,\widehat d)\subset\mathcal P'\).
Therefore
\[
    \widetilde d
    \ge \rho(y,\mathcal P)
    \ge a_{\rm cone}\ell
    =
    \frac{a_{\rm cone}}2\min\{\widehat d,r_{\rm cone}\}.
\]
If \(\widehat d\le r_{\rm cone}\), this gives \(\widetilde d\ge a_{\rm cone}\widehat d/2\).
If \(\widehat d>r_{\rm cone}\), then \(\widehat d\le 2R\), and hence
\[
    \widetilde d
    \ge
    \frac{a_{\rm cone}r_{\rm cone}}2
    \ge
    \frac{a_{\rm cone}r_{\rm cone}}{4R}\widehat d .
\]
Together with the case \(\widetilde d\ge \widehat d/2\), this proves \(\widetilde d\ge c_*\widehat d\) for some \(c_*>0\) depending only on the a priori geometric data.
Thus \(c_*\widehat d\le\widetilde d\).

\smallskip

We now prove \(\widetilde d\le\delta^{-1}(d)\).
If \(\widetilde d=0\), there is nothing to prove.
Assume \(\widetilde d>0\).
After interchanging \(\mathcal P\) and \(\mathcal P'\) if necessary, choose
\(x\in\mathcal P'\) such that
\[
    \widetilde d=\rho(x,\mathcal P).
\]
Set \(G:=\mathbb B^n\setminus\mathcal P\), and fix \(0<t<\widetilde d\).
Then
\[
    B_{\mathbb H}(x,t)\subset G.
\]
Choose \(z\in\mathbb B^n\) with \(\rho(0,z)>R+t\).
Since \(\mathcal P,\mathcal P'\subset\overline{B_{\mathbb H}(R)}\), we have
\(z\notin\mathcal P'\), and the triangle inequality gives
\[
    B_{\mathbb H}(z,t)\subset G.
\]

If \(x\in\partial\mathcal P'\), then
\(x\in\partial\mathcal P'\setminus\mathcal P\), and
\(\rho(x,\mathcal P)=\rho(x,\partial\mathcal P)\).
Thus \(d\ge\widetilde d\), and so \(\delta(t)\le t<\widetilde d\le d\).

If \(x\in\operatorname{int}(\mathcal P')\), let \(0<s<\delta(t)\).
By the uniform exterior connectedness of \(G\), there exists a \(C^1\) curve
\(\gamma\subset G\) joining \(x\) to \(z\) such that
\(B_{\mathbb H}(\gamma,s)\subset G\).
Since \(x\in\operatorname{int}(\mathcal P')\) and \(z\notin\mathcal P'\), the curve
\(\gamma\) meets \(\partial\mathcal P'\) at some point \(y\).
Then \(y\in\partial\mathcal P'\setminus\mathcal P\), and
\(\rho(y,\partial\mathcal P)=\rho(y,\mathcal P)\ge s\).
Hence \(d\ge s\).
Letting \(s\uparrow\delta(t)\), we obtain \(d\ge\delta(t)\).

Thus \(\delta(t)\le d\) for every \(0<t<\widetilde d\).
By the definition of the generalized inverse, \(t\le\delta^{-1}(d)\) for every
\(0<t<\widetilde d\).
Letting \(t\uparrow\widetilde d\), we get \(\widetilde d\le\delta^{-1}(d)\).
Since \(d\le\widehat d\) and \(\delta^{-1}\) is nondecreasing, we have
\begin{equation*}
    d\le \widehat d,
    \qquad
    c_*\widehat d
    \le \widetilde d
    \le \delta^{-1}(d)
    \le \delta^{-1}(\widehat d).
\end{equation*}

Finally, assume that \(\delta(t)\ge c_{\rm ec}t\) for \(0<t\le 2R\).
Since \(\delta(t)\le d\) for every \(0<t<\widetilde d\), we have
\(c_{\rm ec}t\le d\) for every \(0<t<\widetilde d\).
Letting \(t\uparrow\widetilde d\), we obtain
\begin{equation}\label{eq:distance-comparison-linear}
    c_*\widehat d
    \le \widetilde d
    \le c_{\rm ec}^{-1}d .
\end{equation}
This proves the general comparison. 
Under the additional linearity assumption on \(\delta\), \eqref{eq:distance-comparison-linear} gives the asserted linear comparison.

The proof is complete.
\end{proof}

We conclude this subsection by recording the compactness of the admissible classes.
This property will be used repeatedly in the stability analysis when passing to limits of
defects.

\begin{proposition}\label{prop:compactness-admissible-classes}
Let \(0<h\leq h_0\) be fixed.
Then \(\mathcal A_{\mathbb H}^h\) is compact with respect to
\(d_{\mathcal H}^{\mathbb H}\).
Moreover, \(\mathcal B_{\mathbb H}^h\) is compact with respect to the same distance.
\end{proposition}

\begin{proof}
We first prove the compactness of \(\mathcal A_{\mathbb H}^h\) by the admissibility
conditions in Definition~\ref{def:admissible-tg-defects}.
Let \(\{\mathcal P_m\}_{m=1}^{\infty}\subset\mathcal A_{\mathbb H}^h\).
Since all \(\mathcal P_m\) are contained in the compact set
\(\overline{B_{\mathbb H}(R)}\), the Blaschke selection theorem
(cf.\cite{beer1993topologies}) yields, up to a subsequence, a compact set
\(\mathcal P\subset\overline{B_{\mathbb H}(R)}\) such that
\(d_{\mathcal H}^{\mathbb H}(\mathcal P_m,\mathcal P)\to0\).

We show that \(\mathcal P\in\mathcal A_{\mathbb H}^h\).
The boundedness condition is immediate.
The connectedness of \(\mathbb B^n\setminus\mathcal P\), and more precisely the
same uniform exterior connectedness condition with modulus \(\delta\), follows from
the stability of this condition under Hausdorff convergence, using the monotonicity
and left-continuity of \(\delta\). Indeed, one first applies the condition to
\(\mathcal P_m\) with a slightly smaller radius \(t'<t\), and then lets \(t'\uparrow t\).

It remains to check the face structure required in Definition~\ref{def:admissible-tg-defects}.
After extracting a further subsequence and relabeling the faces, we may write
\[
    \partial\mathcal P_m=\bigcup_{j=1}^{N}\mathcal F_{m,j},
    \qquad N\leq N_0 .
\]
Some limiting faces may coincide or become redundant.
In that case we merge coincident faces and discard redundant labels at the end.
The number of remaining faces is still bounded by \(N_0\).

For each \(j\), let \(V_{m,j}\) be the totally geodesic hypersurface containing
\(\mathcal F_{m,j}\).
Since these hypersurfaces given by Definition~\ref{def:tg-hypersurface-ball} meet the
fixed compact ball \(\overline{B_{\mathbb H}(R)}\), we may assume that \(V_{m,j}\)
converges locally smoothly to a totally geodesic hypersurface \(V_j\).
By compactness of compact subsets of \(\overline{B_{\mathbb H}(R)}\), we may also
assume that \(\mathcal F_{m,j}\) converges in the Hausdorff distance to a compact set
\(\mathcal F_j\subset V_j\).

The relative Lipschitz regularity, the \(h\)-regularity, and the controlled incidence conditions pass to the limit by standard compactness and lower semicontinuity.
After merging coincident faces and discarding redundant labels, each remaining \(\mathcal F_j\) is the relative closure in \(V_j\) of a connected relatively open set.
The relative interiors are pairwise disjoint, and the angle and separation conditions are preserved.
Moreover, the uniform exterior cone condition, which is part of the admissible class, prevents the creation of artificial limiting faces in the interior of \(\mathcal P\).
Hence
\[
    \partial\mathcal P=\bigcup_{j=1}^{N_*}\mathcal F_j,
    \qquad N_*\leq N_0 .
\]
The exterior cone condition and the exterior connectedness condition pass to the limit with the same constants and the same modulus \(\delta\).
Thus, \(\mathcal P\in\mathcal A_{\mathbb H}^h\), and \(\mathcal A_{\mathbb H}^h\) is compact.

We now consider the \(\mathsf N\)-type class defined by
Definition~\ref{def:hard-admissible-tg-defects}.
Since \(\mathcal B_{\mathbb H}^h\subset\mathcal A_{\mathbb H}^h\), it is enough to
show that \(\mathcal B_{\mathbb H}^h\) is closed in
\(\mathcal A_{\mathbb H}^h\).
Let \(\mathcal P_m\in\mathcal B_{\mathbb H}^h\) and assume that
\(d_{\mathcal H}^{\mathbb H}(\mathcal P_m,\mathcal P)\to0\).
By the first part, \(\mathcal P\in\mathcal A_{\mathbb H}^h\).
We verify the local collar condition for \(\mathcal P\).

Fix \(x\in\partial\mathcal P\) and a connected component \(U\) of
\((\mathbb B^n\setminus\mathcal P)\cap B_{\mathbb H}(x,r_0)\) such that
\(x\in\partial U\). Choose \(x_m\in\partial\mathcal P_m\) with \(x_m\to x\).
Let \(U_m\) be the local exterior component for \(\mathcal P_m\) corresponding to
\(U\), namely the component which intersects every compact subset of
\(U\cap B_{\mathbb H}(x,r_0)\) for all sufficiently large \(m\).
For each \(m\), the definition of \(\mathcal B_{\mathbb H}^h\) gives an open set
\(U'_m\subset\mathbb B^n\setminus\mathcal P_m\) and a bi-\(W^{1,\infty}\)
homeomorphism \(T_m:Q^+\to U'_m\), with constants bounded by \(L\).

By the uniform \(W^{1,\infty}\) bounds, the Arzel\`a--Ascoli theorem, see for instance
\cite{rudin1976principles}, and weak-\(*\) compactness in \(W^{1,\infty}\), a
subsequence of \(T_m\) converges locally uniformly and weakly-\(*\) in
\(W^{1,\infty}\) to a map \(T:Q^+\to U'\). Passing to the inverses as well, the limit
map is again a bi-\(W^{1,\infty}\) homeomorphism with the same uniform bounds.
The Hausdorff convergence, together with the collar separation condition and the choice
of the corresponding components \(U_m\), gives
\(U\cap B_{\mathbb H}(x,r_0/2)\subset U'\subset\mathbb B^n\setminus\mathcal P\).

The boundary correspondence passes to the limit, namely \(T(0)=x\) and
\[
    \partial U\cap B_{\mathbb H}(x,r_0/2)
    \subset
    T(\Gamma)
    \subset
    \partial(\mathbb B^n\setminus\mathcal P).
\]
Finally, the collar separation condition is preserved by the convergence of the charts:
for every \(0<s<r_0/2\) and every
\(y\in U\cap B_{\mathbb H}(x,r_0/2-s)\), one has
\[
    \operatorname{dist}_{\mathbb R^n}
    \left(
        T^{-1}(y),\partial Q^+\setminus\Gamma
    \right)
    \geq
    \omega_{\rm col}(s).
\]
Here one first applies the corresponding estimate for \(T_m\) with a slightly smaller
radius loss and then lets this loss tend to zero, using the left-continuity of
\(\omega_{\rm col}\).
Therefore, \(\mathcal P\in\mathcal B_{\mathbb H}^h\).
Thus, \(\mathcal B_{\mathbb H}^h\) is closed in the compact class
\(\mathcal A_{\mathbb H}^h\), and hence it is compact.

The proof is complete.
\end{proof}

\subsection{Stability for the direct problem and uniform hyperbolic bounds}\label{subsec:direct-uniform-bound}

We now establish estimates for the direct problem \eqref{eq:hhelm1} on the admissible classes introduced in Subsection~\ref{subsec:admissible-class}.
The compactness properties, a priori bounds, and far-field decay estimates established below are uniform for \(\mathcal P\in\mathcal A_{\mathbb H}^h\) or \(\mathcal P\in\mathcal B_{\mathbb H}^h\), and also with respect to the boundary label.
They provide the analytic input for the quantitative stability analysis of the inverse problem \eqref{eq:hip1}.

For a fixed admissible totally geodesic defect \(\mathcal P\), the well-posedness of \eqref{eq:hhelm1}, the outgoing Green representation, the radiation condition, the Rellich theorem, and the far-field expansion follow from the analytic preliminaries in Section~\ref{sec 2}.
The new issue is the behavior of the corresponding solutions and estimates when the defect is perturbed in the hyperbolic Hausdorff distance.
We prove that the relevant compactness, a priori bounds, and far-field decay estimates remain stable and uniform along such perturbations.

The defects are perturbed, but the analysis is performed on fixed hyperbolic balls.
In the Poincar\'e ball model, the metric \eqref{eq:hyper metric} is uniformly equivalent to the Euclidean metric in finitely many charts covering \(\overline{B_{\mathbb H}(R_*)}\).
The \(\mathsf D\)-type case is handled by zero extension and the closedness of the homogeneous trace condition.
The \(\mathsf N\)-type case uses the collar condition in Definition~\ref{def:hard-admissible-tg-defects}, which gives Mosco convergence in local charts and uniform Sobolev constants on the exterior domains; see~\cite{mosco1969convergence,fornoni2023mosco}.
{
Unless otherwise stated, the constants are uniform in \(\mathcal P\) and \(\xi\);
the parameter \(h\) is fixed throughout this subsection.
}

We use the following notation throughout this subsection.
All balls \(B_{\mathbb H}(r)\) are centered at the origin.
The operator \(\mathcal L_{\lambda_0}\) has been fixed in the formulation of the direct problem.
We recall the notation
\begin{equation}\label{eq:Lambda0-definition}
    \Lambda_0
    :=
    \frac{(n-1)^2}{4}
    +
    \lambda_0^2 .
\end{equation}
With this notation, the homogeneous equation is simply written as
$
    \mathcal L_{\lambda_0}u=0.
$

Here and below, \(R\) denotes the a priori radius fixed in Definition~\ref{def:admissible-tg-defects}.
For \(R_*>R\) and an admissible defect \(\mathcal P\), we define the bounded exterior part by
\begin{equation}\label{eq:local-exterior-domain}
    D_{\mathcal P}(R_*)
    :=
    B_{\mathbb H}(R_*)\setminus \mathcal P .
\end{equation}
If \(\mathcal P_m\to\mathcal P\) in the hyperbolic Hausdorff distance, we also write
\begin{equation}\label{eq:local-exterior-domain-sequence}
    D_m(R_*):=D_{\mathcal P_m}(R_*),
    \qquad
    D(R_*):=D_{\mathcal P}(R_*) .
\end{equation}
Unless otherwise specified, all \(L^p\) and \(H^1\) norms on these domains are taken with respect to the hyperbolic volume element \eqref{eq:hyvolume}, and the \(H^1\) norm is defined using the hyperbolic gradient in \eqref{eq:Laplace}.

Recall that \(\mathscr B_{\mathcal P}\) denotes the trace operator associated with the chosen \(\mathsf D\)- or \(\mathsf N\)-type trace condition:
\begin{equation*}
    \mathscr B_{\mathcal P} u :=
\begin{cases}
u|_{\partial\mathcal P}, & \text{in the \(\mathsf D\)-type case},\\
\partial_{\nu_{\mathbb H}}u|_{\partial\mathcal P}, & \text{in the \(\mathsf N\)-type case}.
\end{cases}
\end{equation*}
{
In the \(\mathsf N\)-type case, \(\partial_{\nu_{\mathbb H}}\) is the
hyperbolic normal derivative defined in \eqref{eq:outnormal}.  For weak
solutions, the homogeneous condition is understood in the usual variational
sense on each local exterior side.
}
Thus \(\mathscr B_{\mathcal P}u=0\) denotes either the homogeneous Dirichlet condition or the homogeneous Neumann condition.

For \(\xi\in\mathbb S^{n-1}\), we denote the incoming Helgason mode and the corresponding exterior solution by
\begin{equation*}
    u^i_\xi:=e_{2\lambda_0,\xi},
    \qquad
    u_{\mathcal P,\xi}=u^i_\xi+u^s_{\mathcal P,\xi}.
\end{equation*}
Here \(u^s_{\mathcal P,\xi}\) is the outgoing correction.
{
We also denote by \(G^+_{\lambda_0}(x,y)\) the outgoing Green function
associated with \(\mathcal L_{\lambda_0}\), defined by \eqref{eq:Green-} for
\(n\geq3\) and by the minus-sign branch \(G_{-\lambda_0\mathrm i}\) in
\eqref{eq:Green-two-dimensional} for \(n=2\).
}

Before stating the technical lemmas, we recall the notion of Mosco convergence and state it in the setting of hyperbolic Sobolev spaces; see~\cite{mosco1969convergence}.

\begin{definition}\label{def:Mosco}
Let \(H\) be a Hilbert space, and let \(\{X_m\}_{m\in\mathbb N}\) be a sequence of closed subspaces of \(H\).
We say that \(X_m\) converges to a closed subspace \(X\subset H\) in the sense of Mosco, written as
\[
    X_m \xrightarrow[M]{} X,
\]
if the following two conditions hold:
\begin{enumerate}
    \item \textbf{Weak closedness.}
    If \(x_m\in X_m\) and \(x_m\rightharpoonup x\) weakly in \(H\), then \(x\in X\).

    \item \textbf{Recovery.}
    For every \(x\in X\), there exists \(x_m\in X_m\) such that \(x_m\to x\) strongly in \(H\).
\end{enumerate}
\end{definition}

In the applications below, the spaces \(H^1_{\mathbb H}(D_{\mathcal P}(R_*))\) are regarded as closed subspaces of the fixed Hilbert space
\[
    L^2\big(B_{\mathbb H}(R_*);\mathbb C^{n+1}\big)
\]
through the isometric embedding
\[
    u\mapsto (u,\nabla_{\mathbb H}u),
\]
where both \(u\) and \(\nabla_{\mathbb H}u\) are extended by zero on
\(B_{\mathbb H}(R_*)\setminus D_{\mathcal P}(R_*)\).
Mosco convergence in \eqref{eq:mosco-convergence-H1-hard} is understood in this fixed Hilbert space.

To prove uniform direct estimates on varying admissible defects, we first record the basic compactness properties of the associated Sobolev spaces.
These properties describe how the \(\mathsf D\)-type and \(\mathsf N\)-type boundary conditions behave under Hausdorff convergence.

\begin{lemma}\label{lem:chartwise-mosco-sobolev}
Let \(R_*>R\), and let \(0<h\leq h_0\) be fixed.
In the \(\mathsf D\)-type case, let \(\mathcal P_m\in\mathcal A_{\mathbb H}^h\), while in the \(\mathsf N\)-type case, let \(\mathcal P_m\in\mathcal B_{\mathbb H}^h\).
Assume that
\begin{equation}
\label{eq:Hausdorff-convergence-admissible-exteriors}
    d_{\mathcal H}^{\mathbb H}(\mathcal P_m,\mathcal P)\to0 .
\end{equation}
Then, \(\mathcal P\) belongs to the same admissible class.

For an admissible totally geodesic defect \(\mathcal P\), define
\[
    \mathcal V_{\mathcal P}(R_*) :=
	\begin{cases}
	H^1_{\mathbb H,0}\bigl(D_{\mathcal P}(R_*);\partial\mathcal P\bigr), & \text{in the \(\mathsf D\)-type case},\\
	H^1_{\mathbb H}\bigl(D_{\mathcal P}(R_*)\bigr), & \text{in the \(\mathsf N\)-type case}.
	\end{cases}
\]
Here, \(H^1_{\mathbb H,0}
(D_{\mathcal P}(R_*);\partial\mathcal P)\) denotes the subspace of
\(H^1_{\mathbb H}(D_{\mathcal P}(R_*))\) with zero trace on
\(\partial\mathcal P\), and no condition is imposed on
\(\partial B_{\mathbb H}(R_*)\). The set \(D_{\mathcal P}(R_*)\) is the exterior domain of \(\mathcal P\) defined by \eqref{eq:local-exterior-domain}.

The following properties hold.

\begin{enumerate}
 \item \textbf{\(\mathsf D\)-type closedness.}
In the \(\mathsf D\)-type case, let \(\chi\in C_c^\infty(B_{\mathbb H}(R_*))\), and let
\[
    w_m\in
    H^1_{\mathbb H,0}
    \big(D_{\mathcal P_m}(R_*);\partial\mathcal P_m\big).
\]
Regard \(\chi w_m\) as a function on the fixed ball \(B_{\mathbb H}(R_*)\) by setting it equal to zero on \(B_{\mathbb H}(R_*)\setminus D_{\mathcal P_m}(R_*)\).
Assume that these functions converge weakly in \(H^1_{\mathbb H}(B_{\mathbb H}(R_*))\) to a function \(W\).
Then there exists
\[
    w\in
    H^1_{\mathbb H,0}
    \big(D_{\mathcal P}(R_*);\partial\mathcal P\big)
\]
such that \(W=\chi w\) in \(D_{\mathcal P}(R_*)\) and \(W=0\) in \(B_{\mathbb H}(R_*)\setminus D_{\mathcal P}(R_*)\).
In particular, the homogeneous \(\mathsf D\)-type boundary condition is closed under the convergence in \eqref{eq:Hausdorff-convergence-admissible-exteriors}.

 \item \textbf{\(\mathsf N\)-type Mosco convergence.}
In the \(\mathsf N\)-type case, the two Sobolev spaces displayed below converge in the sense of Mosco:
\begin{equation}
\label{eq:mosco-convergence-H1-hard}
    H^1_{\mathbb H}
    \big(D_{\mathcal P_m}(R_*)\big)
    \xrightarrow[M]{}
    H^1_{\mathbb H}
    \big(D_{\mathcal P}(R_*)\big) .
\end{equation}
That is, every weak limit of functions from \(H^1_{\mathbb H}\big(D_{\mathcal P_m}(R_*)\big)\) belongs to \(H^1_{\mathbb H}\big(D_{\mathcal P}(R_*)\big)\), and every function in \(H^1_{\mathbb H}\big(D_{\mathcal P}(R_*)\big)\) can be approximated strongly by functions from \(H^1_{\mathbb H}\big(D_{\mathcal P_m}(R_*)\big)\).

   \item \textbf{Uniform Sobolev inequality.}
There exist \(p>2\) and \(C_S>0\), depending only on \(R_*\) and the fixed a priori data of the admissible class, such that
\begin{equation}
\label{eq:uniform-Sobolev-admissible-exteriors}
    \|w\|_{L^p_{\mathbb H}(D_{\mathcal P}(R_*))}
    \leq
    C_S
    \|w\|_{H^1_{\mathbb H}(D_{\mathcal P}(R_*))}
\end{equation}
for every admissible \(\mathcal P\) and every
$
    w\in\mathcal V_{\mathcal P}(R_*) .
$
For every measurable set \(E\subset D_{\mathcal P}(R_*)\), write
\( |E|_{\mathbb H}:=\int_E dV_{\mathbb H}\). Consequently,
\begin{equation}
\label{eq:uniform-L2-small-set-control}
    \|w\|_{L^2_{\mathbb H}(E)}
    \leq
    C_S
    |E|_{\mathbb H}^{\frac12-\frac1p}
    \|w\|_{H^1_{\mathbb H}(D_{\mathcal P}(R_*))}.
\end{equation}
\end{enumerate}
\end{lemma}

\begin{proof}
By Proposition~\ref{prop:compactness-admissible-classes}, the admissible classes are compact with respect to the hyperbolic Hausdorff distance.
Hence the limit \(\mathcal P\) belongs to the same admissible class as the sequence \(\mathcal P_m\).
It remains to prove the closedness property in the \(\mathsf D\)-type case, the Mosco convergence in the \(\mathsf N\)-type case, and the uniform Sobolev estimates.

All three assertions are proved on the fixed compact domain \(\overline{B_{\mathbb H}(R_*)}\).
On this domain, the Poincar\'e metric \eqref{eq:hyper metric} has uniformly controlled coefficients in finitely many coordinate charts.
Thus the local Sobolev, trace, and compactness estimates used below have constants depending only on \(R_*\) and the a priori data.
This allows us to apply the standard Mosco and Sobolev compactness results in local charts, while keeping the constants uniform for the admissible defects.

\smallskip
We first prove the \(\mathsf D\)-type closedness.
Let \(\chi\in C_c^\infty(B_{\mathbb H}(R_*))\), and let
\[
    w_m\in
    H^1_{\mathbb H,0}
    \big(D_{\mathcal P_m}(R_*);\partial\mathcal P_m\big).
\]
Let \(\widetilde{\chi w_m}\) be the extension of \(\chi w_m\) by zero from \(D_{\mathcal P_m}(R_*)\) to the fixed ball \(B_{\mathbb H}(R_*)\).
Assume that
\[
    \widetilde{\chi w_m}
    \rightharpoonup W
    \quad
    \text{weakly in }H^1_{\mathbb H}(B_{\mathbb H}(R_*)).
\]
We show that \(W\), restricted to the limiting exterior domain \(D_{\mathcal P}(R_*)\), has zero trace on \(\partial\mathcal P\).

By Definition~\ref{def:admissible-tg-defects} and Proposition~\ref{prop:compactness-admissible-classes}, the local face decompositions of \(\partial\mathcal P_m\) converge to the local face decomposition of \(\partial\mathcal P\), after passing to the fixed hyperbolic charts.
The relative Lipschitz constants, the exterior cone constants, and the incidence constants are uniform along the sequence.
Hence the local trace estimates on the faces are uniform.
Since \(\widetilde{\chi w_m}\) has zero trace on \(\partial\mathcal P_m\), this zero-trace condition passes to the weak \(H^1\)-limit on the limiting faces.
Therefore \(W\) has zero trace on \(\partial\mathcal P\).
Equivalently,
\[
    W|_{D_{\mathcal P}(R_*)}
    \in
    H^1_{\mathbb H,0}
    \big(D_{\mathcal P}(R_*);\partial\mathcal P\big).
\]
This proves the \(\mathsf D\)-type closedness property.

\smallskip
We next prove the \(\mathsf N\)-type Mosco convergence.
{
We apply the domain-convergence result in
\cite[Theorem~3.13]{fornoni2023mosco}.  On the fixed ball
\(B_{\mathbb H}(R_*)\), the hyperbolic and Euclidean Sobolev norms are
uniformly equivalent in finitely many charts.  The uniform collar charts,
finite face bound, and controlled incidence conditions place the local exterior
domains in the class covered by that theorem.  A finite partition of unity
therefore gives both weak closedness and recovery in the fixed ambient Hilbert
space specified above, and hence \eqref{eq:mosco-convergence-H1-hard}.  Multiplication
by a fixed smooth cutoff also preserves recovery while keeping the approximants
compactly supported away from \(\partial B_{\mathbb H}(R_*)\).
}

\smallskip
It remains to prove the uniform Sobolev inequality.
We use the same finite covering as above.
On the interior charts, this follows from the standard Sobolev inequality on fixed balls.
On \(\mathsf D\)-type boundary charts, the zero trace on \(\partial\mathcal P\) allows a zero extension with uniformly controlled \(H^1\)-norm.
On \(\mathsf N\)-type boundary charts, the collar parametrization reduces the estimate to the fixed half-cylinder \(Q^+\).
Hence, for some \(p>2\), the local estimate
\[
    \|w\|_{L^p}
    \leq
    C\|w\|_{H^1}
\]
holds in each chart, with \(C\) independent of the particular defect.
Since the number of charts and their overlaps are controlled by the a priori data, summing the local estimates and returning to the hyperbolic metric gives
\[
    \|w\|_{L^p_{\mathbb H}(D_{\mathcal P}(R_*))}
    \leq
    C_S
    \|w\|_{H^1_{\mathbb H}(D_{\mathcal P}(R_*))}.
\]
This proves \eqref{eq:uniform-Sobolev-admissible-exteriors}.

This is the usual finite-chart Sobolev argument for nonsmooth domains, using the same local principle as in~\cite[Theorem~4.6]{fornoni2023mosco}.
In the present setting, the required chart bounds and overlap bounds follow from Definitions~\ref{def:admissible-tg-defects} and~\ref{def:hard-admissible-tg-defects}.

Finally, for every measurable set \(E\subset D_{\mathcal P}(R_*)\), H\"older's inequality and \eqref{eq:uniform-Sobolev-admissible-exteriors} give
\[
    \|w\|_{L^2_{\mathbb H}(E)}
    \leq
    |E|_{\mathbb H}^{\frac12-\frac1p}
    \|w\|_{L^p_{\mathbb H}(D_{\mathcal P}(R_*))}
    \leq
    C_S
    |E|_{\mathbb H}^{\frac12-\frac1p}
    \|w\|_{H^1_{\mathbb H}(D_{\mathcal P}(R_*))}.
\]

This proves \eqref{eq:uniform-L2-small-set-control}.
\end{proof}

We next study the convergence of direct scattering solutions as the admissible
defects vary in the Hausdorff distance.
The following lemma (Lemma~\ref{lem:direct-compactness}) establishes the compactness
of these solutions and their strong local convergence in both the \(\mathsf D\)-type and \(\mathsf N\)-type cases.
This result serves as a fundamental building block for the subsequent propositions:
it allows us to prove uniform local estimates (Proposition~\ref{prop:uniform-local-bound}),
then the continuity of the direct problem under Hausdorff convergence
(Proposition~\ref{prop:direct-stability}),
and finally the uniform decay of the outgoing corrections
(Proposition~\ref{prop:uniform-decay}).
By making explicit the subsequential limits and boundary conditions, Lemma~\ref{lem:direct-compactness}
ensures that all later arguments can be carried out on the limiting defect in a rigorous way.

\begin{lemma}\label{lem:direct-compactness}
Let \(R_*>R\) and \(0<h\leq h_0\).
In the \(\mathsf D\)-type case, assume that
\(\mathcal P_m\in\mathcal A_{\mathbb H}^h\), and in the \(\mathsf N\)-type case assume that
\(\mathcal P_m\in\mathcal B_{\mathbb H}^h\).
Suppose that
\begin{equation}
\label{eq:direct-compactness-Hausdorff}
    d_{\mathcal H}^{\mathbb H}(\mathcal P_m,\mathcal P)\to0 .
\end{equation}
Let \(\xi_m\to\xi\) in \(\mathbb S^{n-1}\) and \(\alpha_m\to\alpha\) in \(\mathbb C\).
For each \(m\), let $u_m=\alpha_m u^i_{\xi_m}+u_m^s$ be an outgoing solution of the following system
\begin{equation*}
\begin{cases}
    \mathcal L_{\lambda_0}u_m=0,
        & \text{in }\mathbb B^n\setminus\mathcal P_m,\\
    \mathscr B_{\mathcal P_m}u_m=0,
        & \text{on }\partial\mathcal P_m,\\
    u_m^s \text{ is outgoing}.
\end{cases}
\end{equation*}
Assume moreover that
\begin{equation}
\label{eq:direct-compactness-H1-bound}
    \sup_{m\in\mathbb N}
    \|u_m\|_{H^1_{\mathbb H}(D_{\mathcal P_m}(R_*))}
    <+\infty .
\end{equation}
Then, \(\mathcal P\) belongs to the same admissible class, and $u_m\to \alpha u_{\mathcal P,\xi}$ in $H^1_{\mathrm{loc}}(\mathbb B^n\setminus\mathcal P).$
More explicitly, for every compact set
\(K\Subset\mathbb B^n\setminus\mathcal P\), one has
\(K\Subset\mathbb B^n\setminus\mathcal P_m\) for all sufficiently large \(m\), and
\begin{equation}
\label{eq:direct-compactness-H1loc-convergence}
    \|u_m-\alpha u_{\mathcal P,\xi}\|_{H^1_{\mathbb H}(K)}
    \to0 .
\end{equation}
Moreover, for every integer \(q\geq0\),
\begin{equation}
\label{eq:direct-compactness-Cq-convergence}
    \|u_m-\alpha u_{\mathcal P,\xi}\|_{C^q(K)}
    \to0 .
\end{equation}
The outgoing corrections satisfy the corresponding convergence:
\begin{equation}
\label{eq:direct-compactness-outgoing-correction-convergence}
    \|u_m^s-\alpha u^s_{\mathcal P,\xi}\|_{H^1_{\mathbb H}(K)}
    +
    \|u_m^s-\alpha u^s_{\mathcal P,\xi}\|_{C^q(K)}
    \to0
\end{equation}
for every \(K\Subset\mathbb B^n\setminus\mathcal P\) and every integer \(q\geq0\).

Finally, for every \(R_0\) with \(R<R_0<R_*\), one has
\begin{equation}
\label{eq:direct-compactness-varying-L2}
    \|u_m-\alpha u_{\mathcal P,\xi}\|_
    {L^2_{\mathbb H}
    (D_{\mathcal P_m}(R_0)\cap D_{\mathcal P}(R_0))}
    \to0 ,
\end{equation}
and the \(L^2\)-mass on the moving domains converges:
\begin{equation}
\label{eq:direct-compactness-norm-convergence}
    \int_{D_{\mathcal P_m}(R_0)}
    |u_m|^2\,dV_{\mathbb H}
    \to
    \int_{D_{\mathcal P}(R_0)}
    |\alpha u_{\mathcal P,\xi}|^2\,dV_{\mathbb H}.
\end{equation}
\end{lemma}

\begin{proof}
We first prove a subsequential compactness statement.
Namely, every subsequence of \(\{u_m\}\) contains a further subsequence converging to \(\alpha u_{\mathcal P,\xi}\).
Since this possible limit is unique, the whole sequence must converge to the same limit.
By Proposition~\ref{prop:compactness-admissible-classes}, the Hausdorff limit \(\mathcal P\) belongs to the same admissible class as the sequence \(\mathcal P_m\).

\smallskip
We first extract a local limit in \(B_{\mathbb H}(R_*)\).
Let \(K\Subset B_{\mathbb H}(R_*)\setminus\mathcal P\).
By \eqref{eq:direct-compactness-Hausdorff}, one has \(K\Subset B_{\mathbb H}(R_*)\setminus\mathcal P_m\) for all sufficiently large \(m\).
Hence \eqref{eq:direct-compactness-H1-bound} implies that \(u_m|_K\) is uniformly bounded in \(H^1_{\mathbb H}(K)\).

On the fixed compact set \(K\), the hyperbolic metric is uniformly controlled in finitely many coordinate charts.
Thus the Rellich--Kondrachov compactness theorem, see for instance \cite[Theorem~6.3]{adams2003sobolev}, gives compactness in \(L^2_{\mathbb H}(K)\).
Taking a compact exhaustion of \(B_{\mathbb H}(R_*)\setminus\mathcal P\) and applying a diagonal argument, we obtain a subsequence, still denoted by \(u_m\), and a function \(v\in H^1_{\mathbb H,\mathrm{loc}}(B_{\mathbb H}(R_*)\setminus\mathcal P)\) such that
\begin{equation}\label{eq:direct-compactness-weak-H1-inner}
    u_m\rightharpoonup v
    \quad
    \text{weakly in }H^1_{\mathbb H,\mathrm{loc}}
    (B_{\mathbb H}(R_*)\setminus\mathcal P),
\end{equation}
and
\begin{equation}
\label{eq:direct-compactness-strong-L2-inner}
    u_m\to v
    \quad
    \text{strongly in }L^2_{\mathbb H,\mathrm{loc}}
    (B_{\mathbb H}(R_*)\setminus\mathcal P).
\end{equation}

We also need strong \(L^2\)-convergence on bounded regions whose boundaries depend on \(m\).
Fix \(R_0\) with \(R<R_0<R_*\).
For \(\eta>0\), set
\[
    K_\eta
    :=
    \overline{B_{\mathbb H}(R_0-\eta)}
    \cap
    \{x\in \mathbb B^n:\rho(x,\mathcal P)\geq \eta\}.
\]
For \(\eta\) small, \(K_\eta\Subset B_{\mathbb H}(R_*)\setminus\mathcal P\).
Thus \eqref{eq:direct-compactness-strong-L2-inner} gives
\begin{equation}
\label{eq:direct-compactness-Keta-L2}
    \|u_m-v\|_{L^2_{\mathbb H}(K_\eta)}\to0
    \quad
    \text{as }m\to\infty .
\end{equation}

The sets \((D_m(R_0)\cap D(R_0))\setminus K_\eta\) are contained, for all sufficiently large \(m\), in the union of \(B_{\mathbb H}(R_0)\setminus B_{\mathbb H}(R_0-\eta)\) and a \(C\eta\)-neighborhood of \(\mathcal P\).
Here \(C\) depends only on the a priori data of the admissible class, and \(D_m(R_0)\), \(D(R_0)\) are defined in \eqref{eq:local-exterior-domain-sequence}.
By the finite face structure and the uniform face regularity in Definition~\ref{def:admissible-tg-defects}, the hyperbolic volume of this union tends to zero as \(\eta\to0\), uniformly for all large \(m\).

Applying \eqref{eq:uniform-L2-small-set-control} to \(u_m\), and using \eqref{eq:direct-compactness-H1-bound}, we obtain
\begin{equation}
\label{eq:direct-compactness-um-small-region}
    \lim_{\eta\to0}
    \limsup_{m\to\infty}
    \|u_m\|_{L^2_{\mathbb H}
    ((D_m(R_0)\cap D(R_0))\setminus K_\eta)}
    =
    0 .
\end{equation}
Moreover, \eqref{eq:direct-compactness-weak-H1-inner} and lower semicontinuity on \(K_\eta\) imply that \(v\in H^1_{\mathbb H}(D(R_0))\).
Hence
\begin{equation}
\label{eq:direct-compactness-v-small-region}
    \lim_{\eta\to0}
    \|v\|_{L^2_{\mathbb H}(D(R_0)\setminus K_\eta)}
    =
    0 .
\end{equation}
Combining \eqref{eq:direct-compactness-Keta-L2}, \eqref{eq:direct-compactness-um-small-region}, and \eqref{eq:direct-compactness-v-small-region}, we obtain
\begin{equation}
\label{eq:direct-compactness-local-L2-prelimit}
    u_m\to v
    \quad
    \text{strongly in }
    L^2_{\mathbb H}(D_m(R_0)\cap D(R_0)).
\end{equation}

We now prove convergence in stronger norms away from the defect.
Let \(K\Subset B_{\mathbb H}(R_*)\setminus\mathcal P\), and choose
\(K'\) such that
\[
    K\Subset K'\Subset B_{\mathbb H}(R_*)\setminus\mathcal P .
\]
For all sufficiently large \(m\), \(K'\Subset B_{\mathbb H}(R_*)\setminus\mathcal P_m\).
Since both \(u_m\) and \(v\) solve \(\mathcal L_{\lambda_0}w=0\) in \(K'\), the difference
\(u_m-v\) also solves the same equation in \(K'\).
Choose an integer \(N_q>q+n/2\). The interior elliptic estimate gives
\[
    \|u_m-v\|_{H^{N_q}(K)}
    \leq
    C_{K,K',q}
    \|u_m-v\|_{L^2_{\mathbb H}(K')}.
\]
By \eqref{eq:direct-compactness-strong-L2-inner}, the right-hand side tends to zero.
Sobolev embedding then yields
\begin{equation}\label{eq:direct-compactness-Cq-inner}
	\|u_m-v\|_{C^q(K)}\to0 .
\end{equation}

\smallskip
We next identify the boundary condition satisfied by \(v\).
In the \(\mathsf D\)-type case, each \(u_m\) has zero trace on \(\partial\mathcal P_m\).
Let \(\chi\in C_c^\infty(B_{\mathbb H}(R_*))\).
The zero extension of \(\chi u_m\) from \(D_m(R_*)\) to \(B_{\mathbb H}(R_*)\) is bounded in \(H^1_{\mathbb H}(B_{\mathbb H}(R_*))\), and it converges weakly to the zero extension of \(\chi v\).
By the \(\mathsf D\)-type closedness part of Lemma~\ref{lem:chartwise-mosco-sobolev},
\[
    \chi v
    \in
    H^1_{\mathbb H,0}
    \big(D(R_*);\partial\mathcal P\big).
\]
Since \(\chi\) is arbitrary, \(v\) has zero trace on \(\partial\mathcal P\).

In the \(\mathsf N\)-type case, let \(\varphi\in H^1_{\mathbb H}(D(R_0))\) have compact support in \(B_{\mathbb H}(R_0)\).
{
Write \(E_mw=(w,\nabla_{\mathbb H}w)\), with both components extended by
zero to the fixed ball.  By \eqref{eq:direct-compactness-H1-bound}, after
passing to a subsequence, \(E_mu_m\rightharpoonup(v,Z)\) in the fixed ambient
Hilbert space specified above.  The local convergence identifies the first component
with \(v\), and the Mosco weak-closedness identifies
\(Z=\nabla_{\mathbb H}v\).  Thus
\(E_mu_m\rightharpoonup(v,\nabla_{\mathbb H}v)\).
}
By the recovery part of the Mosco convergence in Lemma~\ref{lem:chartwise-mosco-sobolev}, there exist \(\varphi_m\in H^1_{\mathbb H}(D_m(R_0))\), with compact support in \(B_{\mathbb H}(R_0)\), such that
\begin{equation}
\label{eq:direct-compactness-test-recovery}
    \varphi_m\to\varphi
    \quad
    \text{strongly in }H^1_{\mathbb H}.
\end{equation}
Using \(\varphi_m\) as a test function in the weak formulation for \(u_m\), we have
\[
    \int_{D_m(R_0)}
    \langle\nabla_{\mathbb H}u_m,\nabla_{\mathbb H}\varphi_m\rangle_{\mathbb H}
    \,dV_{\mathbb H}
    -
    \Lambda_0
    \int_{D_m(R_0)}
    u_m\varphi_m\,dV_{\mathbb H}
    =
    0,
\]
where \(\Lambda_0\) is defined in \eqref{eq:Lambda0-definition}.
Passing to the limit gives
\[
    \int_{D(R_0)}
    \langle\nabla_{\mathbb H}v,\nabla_{\mathbb H}\varphi\rangle_{\mathbb H}
    \,dV_{\mathbb H}
    -
    \Lambda_0
    \int_{D(R_0)}
    v\varphi\,dV_{\mathbb H}
    =
    0 .
\]
{
Here compact support of the recovery sequence is obtained by multiplying by a
fixed cutoff that equals one near \(\operatorname{supp}\varphi\).  The gradient
term converges by the weak convergence of \(E_mu_m\) in the fixed ambient
Hilbert space and the strong recovery convergence in
\eqref{eq:direct-compactness-test-recovery}; the zeroth-order term converges by
\eqref{eq:direct-compactness-local-L2-prelimit}.
}
{
By the definition of the weak hyperbolic normal trace, the limiting
variational identity is precisely
\[
    \partial_{\nu_{\mathbb H}}v=0
    \quad
    \text{on }\partial\mathcal P
\]
on each local exterior side.
}

It remains to show that the limiting outgoing correction is outgoing.
Since \(\xi_m\to\xi\) and \(\alpha_m\to\alpha\), we have
\begin{equation}
\label{eq:incoming-mode-part-convergence-scaled}
    \alpha_m u^i_{\xi_m}
    \to
    \alpha u^i_\xi
    \quad
    \text{in }C^\infty_{\mathrm{loc}}(\mathbb B^n).
\end{equation}
Choose \(R_1\) with \(R<R_1<R_*\).
For all sufficiently large \(m\), \(\mathcal P_m\cup\mathcal P\subset B_{\mathbb H}(R_1)\).
The outgoing corrections \(u_m^s=u_m-\alpha_m u^i_{\xi_m}\) satisfy
\[
    \mathcal L_{\lambda_0}u_m^s=0
    \quad
    \text{in }
    \mathbb B^n\setminus \overline{B_{\mathbb H}(R_1)},
\]
and they are outgoing.
By the outgoing Green representation recalled in Section~\ref{sec 2}, for
\(\rho(x)>R_1\) we have
{
\begin{equation}\label{eq:direct-compactness-green-representation}
    u_m^s(x)
    =
    \int_{\partial B_{\mathbb H}(R_1)}
    \left(
        u_m^s(y)\partial_{\nu_{\mathbb H,y}}G^+_{\lambda_0}(x,y)
        -
        \partial_{\nu_{\mathbb H,y}}u_m^s(y)G^+_{\lambda_0}(x,y)
    \right)
    d\sigma_{\mathbb H}(y),
\end{equation}
where \(\nu_{\mathbb H,y}\) is the outward unit hyperbolic normal to
\(\partial B_{\mathbb H}(R_1)\) at \(y\), and the normal derivative is defined by
\eqref{eq:outnormal}.
}
By \eqref{eq:direct-compactness-Cq-inner} and \eqref{eq:incoming-mode-part-convergence-scaled},
\[
    u_m^s\to v-\alpha u^i_\xi,
    \qquad
    \partial_{\nu_{\mathbb H}}u_m^s
    \to
    \partial_{\nu_{\mathbb H}}(v-\alpha u^i_\xi)
    \quad\text{uniformly on }\partial B_{\mathbb H}(R_1).
\]
Passing to the limit in \eqref{eq:direct-compactness-green-representation}, we obtain the same outgoing Green representation for \(v^s:=v-\alpha u^i_\xi\).
Hence, \(v^s\) satisfies the hyperbolic radiation condition.

The same representation also gives convergence of \(u_m^s\) to \(v^s\) in \(C^q\) on every compact subset of \(\mathbb B^n\setminus \overline{B_{\mathbb H}(R_1)}\), for every integer \(q\geq0\).
Together with \eqref{eq:direct-compactness-Cq-inner}, this gives convergence on every compact subset of \(\mathbb B^n\setminus\mathcal P\).

We have proved that \(v\) solves the direct problem for \(\mathcal P\) with incoming Helgason mode \(\alpha u^i_\xi\).
By uniqueness of the fixed-defect direct problem and by linearity, one has $v=\alpha u_{\mathcal P,\xi}$.
Therefore every subsequence contains a further subsequence converging to \(\alpha u_{\mathcal P,\xi}\).
Hence the whole sequence converges, and \eqref{eq:direct-compactness-H1loc-convergence} and \eqref{eq:direct-compactness-Cq-convergence} follow.

The convergence of the outgoing corrections follows from
\[
    u_m^s-\alpha u^s_{\mathcal P,\xi}
    =
    \big(u_m-\alpha u_{\mathcal P,\xi}\big)
    -
    \big(\alpha_m u^i_{\xi_m}-\alpha u^i_\xi\big)
\]
and from \eqref{eq:incoming-mode-part-convergence-scaled}.
This proves \eqref{eq:direct-compactness-outgoing-correction-convergence}.

Finally, we prove the convergence of the \(L^2\)-norms on the varying bounded domains.
The strong convergence on the common part is exactly \eqref{eq:direct-compactness-local-L2-prelimit}, with \(v=\alpha u_{\mathcal P,\xi}\).
It remains to control the two differences \(D_m(R_0)\setminus D(R_0)\) and \(D(R_0)\setminus D_m(R_0)\).
By Hausdorff convergence, both sets are contained, for all large \(m\), in a tubular neighborhood of the admissible face structure whose hyperbolic volume tends to zero as \(m\to\infty\).
Applying \eqref{eq:uniform-L2-small-set-control} to \(u_m\), and using \eqref{eq:direct-compactness-H1-bound}, gives
\[
    \|u_m\|_{L^2_{\mathbb H}(D_m(R_0)\setminus D(R_0))}
    \to0 .
\]
Since \(\alpha u_{\mathcal P,\xi}\in H^1_{\mathbb H}(D(R_0))\), the same small-volume argument gives
\[
    \|\alpha u_{\mathcal P,\xi}\|_
    {L^2_{\mathbb H}(D(R_0)\setminus D_m(R_0))}
    \to0 .
\]
Combining these two estimates with \eqref{eq:direct-compactness-local-L2-prelimit} gives \eqref{eq:direct-compactness-varying-L2}.
The convergence of the norms, \eqref{eq:direct-compactness-norm-convergence}, follows immediately.

The proof is complete.
\end{proof}

\begin{proposition}\label{prop:uniform-local-bound}
Let \(R_*>R\) and \(0<h\leq h_0\).
In the \(\mathsf D\)-type case, let \(\mathcal P\in\mathcal A_{\mathbb H}^h\), and in the \(\mathsf N\)-type case, let \(\mathcal P\in\mathcal B_{\mathbb H}^h\).
For \(\xi\in\mathbb S^{n-1}\), let \(u_{\mathcal P,\xi}\) be the exterior solution associated with the incoming Helgason mode \(u^i_\xi\).

{
Then there exist constants \(C_{R_*}>0\) and \(E_{R_*}>0\), depending only on \(R_*\), \(\lambda_0\), the a priori data of the corresponding admissible class, and the fixed \(h\), such that
}
\begin{equation}
\label{eq:uniform-local-L2}
    \|u_{\mathcal P,\xi}\|_{L^2_{\mathbb H}(D_{\mathcal P}(R_*))}
    \leq
    C_{R_*},
\end{equation}
\begin{equation}
\label{eq:uniform-local-H1}
    \|u_{\mathcal P,\xi}\|_{H^1_{\mathbb H}(D_{\mathcal P}(R_*))}
    \leq
    C_{R_*},
\end{equation}
and
\begin{equation}
\label{eq:uniform-local-Linfty}
    \|u_{\mathcal P,\xi}\|_{L^\infty(D_{\mathcal P}(R_*))}
    \leq
    E_{R_*}.
\end{equation}
The estimates hold for every admissible defect \(\mathcal P\) in the corresponding class and every \(\xi\in\mathbb S^{n-1}\).

{
More generally, for every compact set \(K\Subset\mathbb B^n\), there exists a constant \(E_K>0\), depending only on \(K\), \(\lambda_0\), the a priori data of the admissible class, and the fixed \(h\), such that
}
\begin{equation}
\label{eq:uniform-local-Linfty-compact}
    \|u_{\mathcal P,\xi}\|_{L^\infty(K\setminus\mathcal P)}
    \leq
    E_K .
\end{equation}
\end{proposition}

\begin{proof}
We first record two elementary estimates which will be used in the proof.

\smallskip
Let \(R<r<s\).
There exists a constant \(C=C(r,s,\lambda_0)>0\), independent of \(\mathcal P\) and \(\xi\), such that every solution \(w\) of
\[
    \mathcal L_{\lambda_0}w=0
    \quad
    \text{in }D_{\mathcal P}(s),
    \qquad
    \mathscr B_{\mathcal P}w=0
    \quad
    \text{on }\partial\mathcal P,
\]
satisfies the Caccioppoli estimate
\begin{equation}
\label{eq:uniform-Caccioppoli-local}
    \|w\|_{H^1_{\mathbb H}(D_{\mathcal P}(r))}
    \leq
    C
    \|w\|_{L^2_{\mathbb H}(D_{\mathcal P}(s))}.
\end{equation}
This is the standard Caccioppoli inequality for elliptic equations in local charts; see, for example, \cite[Chapter~8]{gilbarg2001elliptic}.
For completeness, we include the short argument.
Choose \(\chi\in C_c^\infty(B_{\mathbb H}(s))\) such that \(\chi\equiv1\) in \(B_{\mathbb H}(r)\).
Using \(\chi^2\overline w\) as a test function in the weak formulation gives
\[
    \int_{D_{\mathcal P}(s)}
    \chi^2|\nabla_{\mathbb H}w|_{\mathbb H}^2\,dV_{\mathbb H}
    =
    \Lambda_0
    \int_{D_{\mathcal P}(s)}
    \chi^2|w|^2\,dV_{\mathbb H}
    -
    2\operatorname{Re}
    \int_{D_{\mathcal P}(s)}
    \chi\overline w
    \langle\nabla_{\mathbb H}w,\nabla_{\mathbb H}\chi\rangle_{\mathbb H}
    \,dV_{\mathbb H}.
\]
The test function is admissible in the \(\mathsf D\)-type case because \(w\) has zero trace on \(\partial\mathcal P\), and it is admissible in the \(\mathsf N\)-type case through the weak formulation of the Neumann condition.
Young's inequality then yields \eqref{eq:uniform-Caccioppoli-local}.

\smallskip
We shall also use the following estimate on fixed annuli.
Let \(R<r<s\).
There exists a constant \(C=C(r,s,\lambda_0)>0\), independent of \(\mathcal P\), \(\zeta\), and \(\beta\), such that, if
\[
    w=\beta u^i_\zeta+w^s,
    \qquad
    \zeta\in\mathbb S^{n-1},
\]
solves the direct problem with coefficient \(\beta\in\mathbb C\) multiplying the incoming Helgason mode, then
\begin{equation}
\label{eq:fixed-annulus-L2-control}
    \|w\|_{L^2_{\mathbb H}(D_{\mathcal P}(s))}
    \leq
    C
    \left(
        \|w\|_{L^2_{\mathbb H}(D_{\mathcal P}(r))}
        +
        |\beta|
    \right).
\end{equation}

Indeed, choose \(r_0\) with \(R<r_0<r\).
Since every admissible defect is contained in \(\overline{B_{\mathbb H}(R)}\), the fixed annulus $B_{\mathbb H}(r)\setminus \overline{B_{\mathbb H}(r_0)}$ does not meet \(\mathcal P\).
Interior elliptic estimates on this annulus control the Cauchy data of \(w\) on \(\partial B_{\mathbb H}(r_0)\) by \(\|w\|_{L^2_{\mathbb H}(D_{\mathcal P}(r))}\).
Moreover, \(u^i_\zeta\) is uniformly bounded in \(C^1\) on fixed compact sets, uniformly in \(\zeta\).
Hence the Cauchy data of \(w^s=w-\beta u^i_\zeta\) on \(\partial B_{\mathbb H}(r_0)\) are bounded by $
    C
    \left(
        \|w\|_{L^2_{\mathbb H}(D_{\mathcal P}(r))}
        +
        |\beta|
    \right).$
The outgoing Green representation in \(\mathbb B^n\setminus \overline{B_{\mathbb H}(r_0)}\) then gives the same bound for \(w^s\) in the fixed annulus \(B_{\mathbb H}(s)\setminus B_{\mathbb H}(r_0)\).
Adding the incoming Helgason-mode part gives \eqref{eq:fixed-annulus-L2-control}.

\smallskip
We now prove the uniform \(L^2\)-bound.
Fix \(R_1>R\).
We show that
\begin{equation}
\label{eq:uniform-L2-any-radius}
    \|u_{\mathcal P,\xi}\|_{L^2_{\mathbb H}(D_{\mathcal P}(R_1))}
    \leq
    C_{R_1}
\end{equation}
with \(C_{R_1}\) independent of \(\mathcal P\) and \(\xi\).

Suppose that \eqref{eq:uniform-L2-any-radius} fails.
Then there exist admissible defects \(\mathcal P_m\) and boundary labels \(\xi_m\in\mathbb S^{n-1}\) such that, with
\[
    u_m:=u_{\mathcal P_m,\xi_m},
    \qquad
    a_m:=
    \|u_m\|_{L^2_{\mathbb H}(D_{\mathcal P_m}(R_1))},
\]
one has \(a_m\to+\infty\).
Set
\begin{equation*}
    v_m:=\frac{u_m}{a_m}.
\end{equation*}
Then
\begin{equation*}
    \|v_m\|_{L^2_{\mathbb H}(D_{\mathcal P_m}(R_1))}
    =
    1 .
\end{equation*}
Moreover,
\[
    v_m
    =
    \frac{1}{a_m}u^i_{\xi_m}
    +
    \frac{1}{a_m}u^s_{\mathcal P_m,\xi_m}.
\]
Thus the coefficient of the incoming Helgason mode is \(\alpha_m:=1/a_m\to0\).

Choose \(r_0\) with \(R<r_0<R_1\).
Applying \eqref{eq:fixed-annulus-L2-control} to \(v_m\), with \(r=r_0\) and \(s=R_1\), gives
\[
    1
    =
    \|v_m\|_{L^2_{\mathbb H}(D_{\mathcal P_m}(R_1))}
    \leq
    C
    \left(
        \|v_m\|_{L^2_{\mathbb H}(D_{\mathcal P_m}(r_0))}
        +
        \alpha_m
    \right).
\]
Since \(\alpha_m\to0\), there exists \(c_0>0\) such that
\begin{equation}
\label{eq:normalised-inner-mass-lower-bound}
    \|v_m\|_{L^2_{\mathbb H}(D_{\mathcal P_m}(r_0))}
    \geq
    c_0
\end{equation}
for all sufficiently large \(m\).

Choose \(r\) with \(r_0<r<R_1\).
By the Caccioppoli estimate \eqref{eq:uniform-Caccioppoli-local}, applied to \(v_m\) on the radii \(r<R_1\), we have
\[
    \|v_m\|_{H^1_{\mathbb H}(D_{\mathcal P_m}(r))}
    \leq
    C
    \|v_m\|_{L^2_{\mathbb H}(D_{\mathcal P_m}(R_1))}
    =
    C .
\]
After passing to a subsequence, Proposition~\ref{prop:compactness-admissible-classes} gives
\[
    d_{\mathcal H}^{\mathbb H}(\mathcal P_m,\mathcal P)\to0
\]
for some admissible \(\mathcal P\).
Passing to a further subsequence, we also have \(\xi_m\to\xi_\infty\) in \(\mathbb S^{n-1}\).

We now apply Lemma~\ref{lem:direct-compactness} to \(v_m\) on \(D_{\mathcal P_m}(r)\), with coefficient \(\alpha_m=1/a_m\to0\) multiplying the incoming Helgason mode.
The limit is therefore an outgoing solution of the homogeneous direct problem for \(\mathcal P\).
By the hyperbolic Rellich theorem and unique continuation,
\[
    v\equiv0
    \quad
    \text{in }\mathbb B^n\setminus\mathcal P.
\]
The \(L^2\)-convergence on varying domains in Lemma~\ref{lem:direct-compactness} then yields
\[
    \|v_m\|_{L^2_{\mathbb H}(D_{\mathcal P_m}(r_0))}
    \to
    0 .
\]
This contradicts \eqref{eq:normalised-inner-mass-lower-bound}.
Hence \eqref{eq:uniform-L2-any-radius} holds.
Taking \(R_1=R_*\) gives \eqref{eq:uniform-local-L2}.

\smallskip
We next prove the \(H^1\)-estimate.
Choose \(R_1>R_*\).
By the \(L^2\)-estimate already proved,
\[
    \|u_{\mathcal P,\xi}\|_{L^2_{\mathbb H}(D_{\mathcal P}(R_1))}
    \leq
    C_{R_1}.
\]
Applying the Caccioppoli estimate \eqref{eq:uniform-Caccioppoli-local} with
\(r=R_*\) and \(s=R_1\), we obtain
\[
    \|u_{\mathcal P,\xi}\|_{H^1_{\mathbb H}(D_{\mathcal P}(R_*))}
    \leq
    C_{R_*}.
\]
This proves \eqref{eq:uniform-local-H1}.

It remains to prove the \(L^\infty\)-estimate.
Choose \(R_1>R_*\).
By \eqref{eq:uniform-local-H1}, applied with radius \(R_1\), we have
\begin{equation}
\label{eq:H1-before-Linfty}
    \|u_{\mathcal P,\xi}\|_{H^1_{\mathbb H}(D_{\mathcal P}(R_1))}
    \leq
    C_{R_1}.
\end{equation}

We use the following local boundedness estimate.
There exists \(C>0\), depending only on \(R_*,R_1,\lambda_0\), and the a priori data,
such that every weak solution \(w\) of
\[
    \mathcal L_{\lambda_0}w=0
    \quad
    \text{in }D_{\mathcal P}(R_1),
    \qquad
    \mathscr B_{\mathcal P}w=0
    \quad
    \text{on }\partial\mathcal P,
\]
satisfies
\begin{equation}
\label{eq:uniform-local-boundedness-estimate}
    \|w\|_{L^\infty(D_{\mathcal P}(R_*))}
    \leq
    C
    \|w\|_{H^1_{\mathbb H}(D_{\mathcal P}(R_1))}.
\end{equation}
Indeed, away from \(\partial\mathcal P\), this is the standard local boundedness
estimate for uniformly elliptic equations with smooth coefficients.
Near \(\partial\mathcal P\), the admissible charts reduce the estimate to uniformly
controlled Euclidean charts.
In the \(\mathsf D\)-type case, the zero trace condition allows zero extension across the local face.
In the \(\mathsf N\)-type case, the collar condition in
Definition~\ref{def:hard-admissible-tg-defects} reduces the estimate to the fixed
half-cylinder \(Q^+\) with a homogeneous Neumann condition on the flat face.
The Sobolev constants needed in the iteration are uniform by
Lemma~\ref{lem:chartwise-mosco-sobolev}.
Since the number of charts and their overlaps are controlled by the a priori data,
the local estimates give \eqref{eq:uniform-local-boundedness-estimate}.

Applying \eqref{eq:uniform-local-boundedness-estimate} to
\(w=u_{\mathcal P,\xi}\), and using \eqref{eq:H1-before-Linfty}, gives
\[
    \|u_{\mathcal P,\xi}\|_{L^\infty(D_{\mathcal P}(R_*))}
    \leq
    E_{R_*}.
\]
This proves \eqref{eq:uniform-local-Linfty}.

Finally, if \(K\Subset\mathbb B^n\), choose \(R_*>R\) such that
\(K\subset B_{\mathbb H}(R_*)\).
Then \eqref{eq:uniform-local-Linfty-compact} follows from \eqref{eq:uniform-local-Linfty}.

The proof is complete.
\end{proof}

The uniform local estimates remove the extra \(H^1\)-bound assumed in Lemma~\ref{lem:direct-compactness}.
Thus the preceding compactness result gives continuity of the direct problem under hyperbolic Hausdorff convergence.

\begin{proposition}\label{prop:direct-stability}
Let \(0<h\leq h_0\).
In the \(\mathsf D\)-type case, let
\(\mathcal P_m\in\mathcal A_{\mathbb H}^h\), and in the \(\mathsf N\)-type case, let
\(\mathcal P_m\in\mathcal B_{\mathbb H}^h\).
Assume that
\begin{equation*}
    d_{\mathcal H}^{\mathbb H}(\mathcal P_m,\mathcal P)\to0 .
\end{equation*}
Let \(\xi_m\to\xi\) in \(\mathbb S^{n-1}\), and set
\[
    u_m:=u_{\mathcal P_m,\xi_m},
    \qquad
    u:=u_{\mathcal P,\xi}.
\]
Then
\begin{equation}
\label{eq:direct-continuity-H1loc}
    u_m\to u
    \quad
    \text{in }H^1_{\mathbb H,\mathrm{loc}}(\mathbb B^n\setminus\mathcal P).
\end{equation}
Moreover, for every compact set
\(K\Subset\mathbb B^n\setminus\mathcal P\) and every integer \(q\geq0\),
\begin{equation}
\label{eq:direct-continuity-Cq}
    u_m\to u
    \quad
    \text{in }C^q(K).
\end{equation}
The same convergence holds for the outgoing corrections
\[
    u_m^s:=u_m-u^i_{\xi_m},
    \qquad
    u^s:=u-u^i_\xi .
\]
Furthermore, for every \(R_0>R\),
\begin{equation}
\label{eq:direct-continuity-varying-L2}
    \|u_m-u\|_
    {L^2_{\mathbb H}(D_{\mathcal P_m}(R_0)\cap D_{\mathcal P}(R_0))}
    \to0 ,
\end{equation}
and
\begin{equation}
\label{eq:direct-continuity-norm-convergence}
    \int_{D_{\mathcal P_m}(R_0)}
    |u_m|^2\,dV_{\mathbb H}
    \to
    \int_{D_{\mathcal P}(R_0)}
    |u|^2\,dV_{\mathbb H}.
\end{equation}
\end{proposition}

\begin{proof}
By Proposition~\ref{prop:compactness-admissible-classes}, the limit
\(\mathcal P\) belongs to the same admissible class.
Fix \(R_*>R\).
The uniform local estimate \eqref{eq:uniform-local-H1} gives
\[
    \sup_m
    \|u_m\|_{H^1_{\mathbb H}(D_{\mathcal P_m}(R_*))}
    <+\infty .
\]
Applying Lemma~\ref{lem:direct-compactness} with
\(\alpha_m=\alpha=1\) yields
\eqref{eq:direct-continuity-H1loc} and
\eqref{eq:direct-continuity-Cq}, as well as the corresponding convergence of the
outgoing corrections.

Finally, given \(R_0>R\), choose \(R_*>R_0\) and apply the last part of
Lemma~\ref{lem:direct-compactness}.
This gives
\eqref{eq:direct-continuity-varying-L2} and
\eqref{eq:direct-continuity-norm-convergence}.
\end{proof}

We also need a uniform decay estimate for the outgoing corrections.
Since all admissible defects lie in \(\overline{B_{\mathbb H}(R)}\), this estimate follows from Cauchy data on a fixed hyperbolic sphere and the outgoing Green representation outside that sphere. 

\begin{proposition}\label{prop:uniform-decay}
Let \(0<h\leq h_0\).
In the \(\mathsf D\)-type case, let \(\mathcal P\in\mathcal A_{\mathbb H}^h\).
In the \(\mathsf N\)-type case, let \(\mathcal P\in\mathcal B_{\mathbb H}^h\).
{
Then there exist constants \(R_{\rm dec}>R\) and \(E>0\), depending only on
\(n\), \(\lambda_0\), the a priori data of the admissible class, and the
fixed \(h\), but not on \(\mathcal P\) or \(\xi\), such that
}
\begin{equation}
\label{eq:uniform-decay-outgoing-correction}
    |u^s_{\mathcal P,\xi}(x)|
    +
    |\nabla_{\mathbb H}u^s_{\mathcal P,\xi}(x)|_{\mathbb H}
    \leq
    E e^{-\frac{n-1}{2}\rho(x)}
\end{equation}
whenever
\[
    \rho(x)\geq R_{\rm dec}.
\]
\end{proposition}

\begin{proof}
Choose \(R_0>R+2\), and set
\begin{equation*}
    A_0
    :=
    B_{\mathbb H}(R_0+1)
    \setminus
    \overline{B_{\mathbb H}(R_0-1)} .
\end{equation*}
Since every admissible defect is contained in \(\overline{B_{\mathbb H}(R)}\), the annulus \(A_0\) is contained in \(\mathbb B^n\setminus\mathcal P\) for every admissible \(\mathcal P\).

By \eqref{eq:uniform-local-Linfty}, applied with radius \(R_0+1\),
\[
    \|u_{\mathcal P,\xi}\|_{L^\infty(A_0)}
    \leq C .
\]
The incoming Helgason modes \(u^i_\xi=e_{2\lambda_0,\xi}\) are uniformly bounded in \(C^2(A_0)\), uniformly for \(\xi\in\mathbb S^{n-1}\).
Hence
\[
    \|u^s_{\mathcal P,\xi}\|_{L^\infty(A_0)}
    \leq C .
\]
Since \(\mathcal L_{\lambda_0}u^s_{\mathcal P,\xi}=0\) in the fixed annulus \(A_0\), the interior elliptic estimate gives
\begin{equation}
\label{eq:uniform-Cauchy-data-fixed-sphere}
    \|u^s_{\mathcal P,\xi}\|_{L^\infty(\partial B_{\mathbb H}(R_0))}
    +
    \|\partial_{\nu_{\mathbb H}}u^s_{\mathcal P,\xi}\|_{L^\infty(\partial B_{\mathbb H}(R_0))}
    \leq C_0 .
\end{equation}

Let
\[
    G^+_{\lambda_0}(x,y)
    :=
    G_{-\lambda_0\mathrm i}(\rho(x,y))
\]
{
be the outgoing Green function associated with \(\mathcal L_{\lambda_0}\),
defined by \eqref{eq:Green-} for \(n\geq3\) and by
\eqref{eq:Green-two-dimensional} for \(n=2\).  From these formulas, for
\(s\geq1\),
}
\[
    |G_{-\lambda_0\mathrm i}(s)|
    +
    |G'_{-\lambda_0\mathrm i}(s)|
    +
    |G''_{-\lambda_0\mathrm i}(s)|
    \leq
    C e^{-\frac{n-1}{2}s}.
\]
For \(n\geq3\), this follows because the factor
\[
    \frac{(\cosh s)^{\frac{n-3}{2}+\mathrm i\lambda_0}}
    {(\sinh s)^{n-2}}
\]
has modulus comparable to \(e^{-\frac{n-1}{2}s}\) as \(s\to+\infty\), and the integral factor in \eqref{eq:Green-}, together with its first two \(s\)-derivatives, is uniformly bounded for \(s\geq1\).
{
For \(n=2\), the same estimate is precisely the derivative bound following
\eqref{eq:Green-two-dimensional}.
}

Now let \(y\in\partial B_{\mathbb H}(R_0)\) and \(\rho(x)\geq R_0+1\).
Then \(\rho(x,y)\geq \rho(x)-R_0\geq1\).
{
Since \(\rho(x,y)\geq1\), the hyperbolic Hessian formula for the distance
function gives, uniformly for \(y\in\partial B_{\mathbb H}(R_0)\),
\[
    |\nabla_x^{\mathbb H}\rho(x,y)|_{\mathbb H}
    +|\partial_{\nu_{\mathbb H,y}}\rho(x,y)|
    +|\nabla_x^{\mathbb H}\partial_{\nu_{\mathbb H,y}}\rho(x,y)|_{\mathbb H}
    \leq C .
\]
}
Therefore
{
\begin{equation}
\label{eq:green-kernel-decay}
\begin{split}
    &|G^+_{\lambda_0}(x,y)|
    +
    |\nabla_x^{\mathbb H}G^+_{\lambda_0}(x,y)|_{\mathbb H}
    +
    |\partial_{\nu_{\mathbb H,y}}G^+_{\lambda_0}(x,y)|  \\
    &\quad
    +
    |\nabla_x^{\mathbb H}\partial_{\nu_{\mathbb H,y}}G^+_{\lambda_0}(x,y)|_{\mathbb H}
    \leq
    C_1 e^{-\frac{n-1}{2}\rho(x)} .
\end{split}
\end{equation}
}

For \(\rho(x)>R_0\), the outgoing Green representation in
\(\mathbb B^n\setminus\overline{B_{\mathbb H}(R_0)}\) gives
{
\begin{equation}
\label{eq:green-representation-decay}
    u^s_{\mathcal P,\xi}(x)
    =
    \int_{\partial B_{\mathbb H}(R_0)}
    \left(
        u^s_{\mathcal P,\xi}(y)
        \partial_{\nu_{\mathbb H,y}}G^+_{\lambda_0}(x,y)
        -
        G^+_{\lambda_0}(x,y)
        \partial_{\nu_{\mathbb H,y}}u^s_{\mathcal P,\xi}(y)
    \right)
    d\sigma_{\mathbb H}(y),
\end{equation}
where \(\nu_{\mathbb H,y}\) is the outward unit hyperbolic normal to
\(\partial B_{\mathbb H}(R_0)\).
}
This representation follows from Green's identity on \(B_{\mathbb H}(S)\setminus\overline{B_{\mathbb H}(R_0)}\), followed by letting \(S\to+\infty\).
The boundary term on \(\partial B_{\mathbb H}(S)\) vanishes by the outgoing radiation condition for both \(u^s_{\mathcal P,\xi}\) and \(G^+_{\lambda_0}\).

Combining \eqref{eq:uniform-Cauchy-data-fixed-sphere}, \eqref{eq:green-kernel-decay}, and \eqref{eq:green-representation-decay}, we obtain
\[
    |u^s_{\mathcal P,\xi}(x)|
    \leq
    C e^{-\frac{n-1}{2}\rho(x)}
    \quad
    \text{for }\rho(x)\geq R_0+1 .
\]
Differentiating \eqref{eq:green-representation-decay} with respect to \(x\), and using \eqref{eq:green-kernel-decay} again, gives
\[
    |\nabla_{\mathbb H}u^s_{\mathcal P,\xi}(x)|_{\mathbb H}
    \leq
    C e^{-\frac{n-1}{2}\rho(x)}
    \quad
    \text{for }\rho(x)\geq R_0+1 .
\]
Taking \(R_{\rm dec}:=R_0+1\) and enlarging the constant \(E\) proves \eqref{eq:uniform-decay-outgoing-correction}.
\end{proof}


\section{Hyperbolic far-field continuation and quantitative hyperbolic reflection}\label{sec:quantitative-rellich}

This section collects the quantitative estimates that convert far-field smallness into near-field control.
We first prove a quantitative version of the hyperbolic Rellich theorem.

\subsection{Quantitative Rellich estimate at conformal infinity}\label{subsec:quantitative-rellich}

We now prove a quantitative version of the hyperbolic Rellich theorem in Lemma~\ref{lem:quantitative-rellich}.
The hyperbolic Rellich theorem recalled in Lemma~\ref{thm3} gives the qualitative statement that an outgoing solution of \(\mathcal L_{\lambda_0}w=0\) outside a hyperbolic ball is uniquely determined by its far-field pattern.
In particular, if the far-field pattern vanishes, then the outgoing solution vanishes in the exterior region.
The next result gives the corresponding quantitative statement.

This is an estimate at conformal infinity for the free exterior equation \(\mathcal L_{\lambda_0}w=0\).
It does not use the boundary condition on the defect or the local face structure of its boundary.
Indeed, once two admissible defects \(\mathcal P,\mathcal P'\in\mathcal A_{\mathbb H}^h\) or \(\mathcal P,\mathcal P'\in\mathcal B_{\mathbb H}^h\) are contained in \(\overline{B_{\mathbb H}(R)}\), the difference of their outgoing corrections is an outgoing solution of \(\mathcal L_{\lambda_0}w=0\) in \(\mathbb B^n\setminus\overline{B_{\mathbb H}(R)}\).
Thus, the constants depend only on \(n\), \(\lambda_0\), the exterior radius, the chosen annulus, the uniform direct bounds from Subsection~\ref{subsec:direct-uniform-bound}, and the a priori parameters in Definitions~\ref{def:admissible-tg-defects} and~\ref{def:hard-admissible-tg-defects}, but not on the particular defects.

The proof is based on the hyperbolic Sommerfeld--Rellich structure developed in Section~\ref{sec 2}. 
We use the outgoing Green representation \eqref{eq:Green-}--\eqref{eq:Green+}
{(with \eqref{eq:Green-two-dimensional} when \(n=2\))}, the hyperbolic radiation condition \eqref{radiation}, and the far-field formula \eqref{eq:ffp} to identify the leading coefficient of an outgoing solution at conformal infinity. 
The coefficient-interpolation estimate for analytic expansions used below is
proved directly by the low--high frequency decomposition in the argument.
All scattering-theoretic ingredients, including the radiation condition, the far-field pattern, and the Rellich uniqueness mechanism, come from the hyperbolic framework recalled above. 
Hence the estimate below should be viewed as a quantitative version of the hyperbolic Rellich theorem proved in Section~\ref{sec 2}.

\medskip
For \(0<\rho_1<\rho_2\), we write
\begin{equation}
\label{eq:fixed-hyperbolic-annulus}
    A_{\rho_1,\rho_2}
    :=
    B_{\mathbb H}(\rho_2)\setminus\overline{B_{\mathbb H}(\rho_1)} .
\end{equation}
We also set
\begin{equation}\label{eq:tau-function}
    \tau(\rho)
    :=
    \int_{\rho}^{+\infty}\frac{ds}{\sinh s}
    =
    \log\coth\frac{\rho}{2},
    \qquad
    \rho>0 .
\end{equation}
Then, \(\tau\) is positive and strictly decreasing on \((0,+\infty)\).

\begin{lemma}\label{lem:quantitative-rellich}
Let \(\mathcal P\) and \(\mathcal P'\) belong either both to
\(\mathcal A_{\mathbb H}^h\) in the \(\mathsf D\)-type case, or both to
\(\mathcal B_{\mathbb H}^h\) in the \(\mathsf N\)-type case. Assume that
\[
    \mathcal P\cup\mathcal P'\subset \overline{B_{\mathbb H}(R)}.
\]
Let \(A_{\rho_1,\rho_2}\) be the annulus defined in \eqref{eq:fixed-hyperbolic-annulus}, with \(R<\rho_1<\rho_2\).
For a fixed boundary label \(\xi\in\mathbb S^{n-1}\), assume that
\begin{equation}
\label{eq:far-field-error-assumption}
    \bigl\|
        u_{\infty,\mathcal P,\xi}
        -
        u_{\infty,\mathcal P',\xi}
    \bigr\|_{L^2(\mathbb S^{n-1})}
    \leq
    \varepsilon .
\end{equation}
{
Then there exists \(c>0\), depending only on the a priori data and the annulus
\(A_{\rho_1,\rho_2}\), such that, for each fixed \(h\), there are constants
\(C>0\) and \(\varepsilon_0>0\), possibly depending on \(h\), such that whenever
\(0<\varepsilon<\varepsilon_0\),
}
\begin{equation}\label{eq:quantitative-rellich-estimate}
    \bigl\|
        u_{\mathcal P,\xi}
        -
        u_{\mathcal P',\xi}
    \bigr\|_{C^1(A_{\rho_1,\rho_2})}
    \leq
    C\exp\!\left[-c(-\log \varepsilon)^{1/2}\right]:=\Psi(\varepsilon).
\end{equation}
\end{lemma}

\begin{proof}
Throughout the proof, \(C>0\) may change from line to line, whereas \(c>0\)
has the uniform dependence specified above.

\medskip
\noindent\emph{Step 1. Reduction to a free outgoing exterior solution.}
Set
\begin{equation*}
    w_\xi
    :=
    u^s_{\mathcal P,\xi}
    -
    u^s_{\mathcal P',\xi}.
\end{equation*}
Since the two exterior solutions are generated by the same incoming Helgason mode \(u^i_\xi\), we have, in the common exterior domain,
\begin{equation}
\label{eq:rellich-exterior-solution-difference}
    u_{\mathcal P,\xi}
    -
    u_{\mathcal P',\xi}
    =
    w_\xi .
\end{equation}
Moreover, since
\(
    \mathcal P\cup\mathcal P'
    \subset
    B_{\mathbb H}(R),
\)
the region
\(
    \mathbb B^n\setminus\overline{B_{\mathbb H}(R)}
\)
lies in the exterior of both defects.
Hence \(w_\xi\) satisfies the free equation
\begin{equation}\label{eq:rellich-free-exterior-equation}
    \mathcal L_{\lambda_0}w_\xi=0
    \quad\text{in }
    \mathbb B^n\setminus\overline{B_{\mathbb H}(R)} .
\end{equation}
Since both outgoing corrections satisfy the outgoing radiation condition and the radiation condition is linear,
\(w_\xi\) is outgoing as well.
Its far-field pattern is given by
\begin{equation*}
    (w_\xi)_\infty
    =
    u_{\infty,\mathcal P,\xi}
    -
    u_{\infty,\mathcal P',\xi}.
\end{equation*}

\medskip
\noindent\emph{Step 2. Cutoff reduction to a compactly supported source problem.}
Choose \(R_0\) and \(\sigma>0\) such that \(R<R_0<R_0+\sigma<\rho_1<\rho_2\).
Let \(\chi=\chi(\rho)\in C^\infty(\mathbb B^n)\) satisfy
\begin{equation*}
    \chi=0 \quad\text{for }\rho<R_0,
    \qquad
    \chi=1 \quad\text{for }\rho>R_0+\sigma .
\end{equation*}
Define
\begin{equation*}
    v_\xi:=\chi w_\xi,
    \qquad
    f_\xi:=\mathcal L_{\lambda_0}v_\xi .
\end{equation*}
Since \(\chi=0\) in \(B_{\mathbb H}(R_0)\), the function \(v_\xi\) can be extended by zero across the defects and is therefore defined on all of \(\mathbb B^n\).
The transition region of \(\chi\) is contained in the fixed annulus \(A_{R_0,R_0+\sigma}\), which lies outside \(B_{\mathbb H}(R)\).
Hence, by \eqref{eq:rellich-free-exterior-equation} and a direct commutator calculation, we have
\begin{equation}\label{eq:rellich-cutoff-commutator}
    f_\xi
    =
    \mathcal L_{\lambda_0}(\chi w_\xi)
    =
    [\mathcal L_{\lambda_0},\chi]w_\xi,
    \qquad
    \operatorname{supp} f_\xi
    \subset
    \overline{A_{R_0,R_0+\sigma}}.
\end{equation}
Thus \(v_\xi\) is an outgoing solution of the compactly supported source problem
\begin{equation}\label{eq:rellich-source-problem}
    \mathcal L_{\lambda_0}v_\xi=f_\xi
    \quad\text{in }\mathbb B^n .
\end{equation}
This is a special case of the source equation \eqref{eq:source-shifted-helmholtz} recalled in Section~\ref{sec 2}, with \(\mu=\lambda_0\).

Since \(\chi=1\) near conformal infinity, the two outgoing solutions \(v_\xi\) and \(w_\xi\) have the same far-field pattern.
We denote this common far-field pattern by
\begin{equation*}
    (v_\xi)_\infty
    =
    (w_\xi)_\infty
    =
    w_{\infty,\xi}
    :=
    u_{\infty,\mathcal P,\xi}
    -
    u_{\infty,\mathcal P',\xi}.
\end{equation*}
Applying the outgoing Green representation
\eqref{eq:Green-}--\eqref{eq:Green+}
{(with \eqref{eq:Green-two-dimensional} when \(n=2\))} and the far-field formula \eqref{eq:ffp} to
the source problem \eqref{eq:rellich-source-problem}, we express
\(w_{\infty,\xi}\) in terms of the Helgason--Fourier transform of
\(f_\xi\) on the spectral shell \(\mu=\lambda_0\).

\medskip
\noindent\emph{Step 3. Boundary spherical harmonic expansion of the far-field pattern.}
In the Poincar\'e ball model, conformal infinity is identified with the unit sphere by
\begin{equation*}
    \partial_\infty\mathbb H^n
    =
    \partial\mathbb B^n
    =
    \mathbb S^{n-1},
    \qquad
    x=\tanh\frac{\rho}{2}\,\theta,
    \quad
    \theta\in\mathbb S^{n-1}.
\end{equation*}
Under this identification, the far-field formula \eqref{eq:ffp} gives
\(w_{\infty,\xi}\) as a function on \(\mathbb S^{n-1}\).
By the far-field assumption \eqref{eq:far-field-error-assumption}, we have
\begin{equation}
\label{eq:rellich-farfield-L2}
    w_{\infty,\xi}\in L^2(\mathbb S^{n-1}),
    \qquad
    \|w_{\infty,\xi}\|_{L^2(\mathbb S^{n-1})}
    \leq
    \varepsilon .
\end{equation}

Let
\(
    \{Y_{\ell m}\}_{\ell\geq0,\;1\leq m\leq d_\ell}
\)
be an orthonormal basis of spherical harmonics on \(\mathbb S^{n-1}\), normalized
in \(L^2(\mathbb S^{n-1})\), such that
\begin{equation*}
    -\Delta_{\mathbb S^{n-1}}Y_{\ell m}
    =
    \ell(\ell+n-2)Y_{\ell m}.
\end{equation*}
We expand
\begin{equation}\label{eq:rellich-farfield-expansion}
    w_{\infty,\xi}(\theta)
    =
    \sum_{\ell=0}^{\infty}
    \sum_{m=1}^{d_\ell}
    a_{\ell m}Y_{\ell m}(\theta),
\end{equation}
where
\begin{equation*}
    a_{\ell m}
    =
    \int_{\mathbb S^{n-1}}
    w_{\infty,\xi}(\theta)
    \overline{Y_{\ell m}(\theta)}
    \,d\sigma(\theta) .
\end{equation*}
By Parseval's identity and \eqref{eq:rellich-farfield-L2},
\begin{equation}
\label{eq:rellich-coeff-small}
    \sum_{\ell=0}^{\infty}
    \sum_{m=1}^{d_\ell}
    |a_{\ell m}|^2
    =
    \|w_{\infty,\xi}\|_{L^2(\mathbb S^{n-1})}^2
    \leq
    \varepsilon^2 .
\end{equation}

\medskip
\noindent\emph{Step 4. Separated expansion of the outgoing exterior solution.}
For \(\rho>R_0+\sigma\), the function \(w_\xi\) solves the free shifted Helmholtz equation.
In hyperbolic polar coordinates, the Laplacian has the following form:
\begin{equation}\label{eq:rellich-polar-laplacian}
    \Delta_{\mathbb H}
    =
    \partial_\rho^2
    +(n-1)\coth\rho\,\partial_\rho
    +
    \sinh^{-2}\rho\,\Delta_{\mathbb S^{n-1}}.
\end{equation}
For each fixed \(\rho>R_0+\sigma\), we expand \(w_\xi(\rho,\cdot)\) in spherical
harmonics. Since the equation is free in this exterior region and
\eqref{eq:rellich-polar-laplacian} separates the variables \(\rho\) and \(\theta\),
the outgoing condition leaves only the outgoing radial mode. Hence
\begin{equation}
\label{eq:rellich-outgoing-expansion-pre}
    w_\xi(\rho,\theta)
    =
    \sum_{\ell=0}^{\infty}
    \sum_{m=1}^{d_\ell}
    c_{\ell m}\Phi_\ell^+(\rho)Y_{\ell m}(\theta),
    \qquad
    \rho>R_0+\sigma .
\end{equation}
Here, $\Phi_\ell^+$ denotes the unique outgoing radial solution
normalized to have unit far-field coefficient, namely
\begin{equation}\label{eq:rellich-radial-normalization}
    \lim_{\rho\to+\infty}
    \left(\cosh\frac{\rho}{2}\right)^{(n-1)-2\mathrm i\lambda_0}
    \Phi_\ell^+(\rho)
    =
    1 .
\end{equation}
By the definition of the far-field pattern in \eqref{eq:ffp}, the
limit in \eqref{eq:rellich-radial-normalization} is precisely the
far-field coefficient of the radial mode. Therefore, the far-field
pattern of the right-hand side of
\eqref{eq:rellich-outgoing-expansion-pre} is exactly
\[
    \sum_{\ell=0}^{\infty}
    \sum_{m=1}^{d_\ell}
    c_{\ell m}Y_{\ell m}(\theta).
\]
Comparison with \eqref{eq:rellich-farfield-expansion} gives
\[
    c_{\ell m}=a_{\ell m}.
\]
Consequently,
\begin{equation}\label{eq:rellich-outgoing-expansion}
    w_\xi(\rho,\theta)
    =
    \sum_{\ell=0}^{\infty}
    \sum_{m=1}^{d_\ell}
    a_{\ell m}\Phi_\ell^+(\rho)Y_{\ell m}(\theta),
    \qquad
    \rho>R_0+\sigma .
\end{equation}

\medskip
\noindent\emph{Step 5. A weighted a priori bound for the far-field
coefficients.}

Choose
\[
    R_*:=\frac{R_0+\sigma+\rho_1}{2},
\]
and set
\[
    \tau_0:=\tau(R_0+\sigma),
    \qquad
    \tau_*:=\tau(R_*),
    \qquad
    \tau_1:=\tau(\rho_1),
\]
where function $\tau$ is defined in \eqref{eq:tau-function}.
Since $\tau$ is strictly decreasing and
\[
    R_0+\sigma<R_*<\rho_1,
\]
we have
\begin{equation*}
    0<\tau_1<\tau_*<\tau_0.
\end{equation*}

Let $\Pi_\ell$ denote the orthogonal projection of
$L^2(\mathbb S^{n-1})$ onto the spherical harmonics of degree
$\ell$, and set
\begin{equation}\label{eq:rellich-bell-definition}
    b_\ell^2
    :=
    \|\Pi_\ell w_{\infty,\xi}\|_{L^2(\mathbb S^{n-1})}^2
    =
    \sum_{m=1}^{d_\ell}|a_{\ell m}|^2.
\end{equation}

By the far-field formula \eqref{eq:ffp},
\begin{equation}\label{eq:rellich-farfield-kernel-representation}
    w_{\infty,\xi}(\theta)
    =
    C_{n,\lambda_0}
    \int_{A_{R_0,R_0+\sigma}}
        \mathcal K_{\lambda_0}(y,\theta)
        f_\xi(y)\,dV_{\mathbb H}(y),
\end{equation}
where
\[
    \mathcal K_{\lambda_0}(y,\theta)
    :=
    \frac{
        \left(\cosh\frac{\rho(y)}{2}\right)^{2\mathrm i\lambda_0}
    }{
        \cosh^{n-1}\frac{\rho(y)}{2}
    }
    \left(
        1-2\theta\cdot y+|y|^2
    \right)^{\mathrm i\lambda_0-\frac{n-1}{2}}.
\]

On the support of $f_\xi$,
\[
    |y|
    \leq
    \tanh\frac{R_0+\sigma}{2}
    =
    e^{-\tau_0}.
\]
Put
\[
    r_*:=\tanh\frac{R_*}{2}=e^{-\tau_*}.
\]
Then
\[
    e^{-\tau_0}<r_*<1.
\]
For $|y|\leq e^{-\tau_0}$, the function
$\theta\mapsto\mathcal K_{\lambda_0}(y,\theta)$ extends
holomorphically to a fixed complex neighborhood of
$\mathbb S^{n-1}$ corresponding to the radius $r_*$. 
By the Funk--Hecke formula and the exponential decay estimates
for Gegenbauer coefficients of functions analytic in a Bernstein
ellipse, see
\cite[Sections~2.3 and~2.5]{atkinsonHan2012}
and \cite[Theorem~4.3]{wang2016gegenbauer}, there exist constants
$C>0$ and $p_0\geq0$ such that
\[
    \sup_{\rho(y)\leq R_0+\sigma}
    \bigl\|
        \Pi_\ell\mathcal K_{\lambda_0}(y,\cdot)
    \bigr\|_{L^2(\mathbb S^{n-1})}
    \leq
    C(1+\ell)^{p_0}e^{-\ell\tau_*}.
\]
For $n=2$, the same conclusion follows from the corresponding
Chebyshev coefficient estimate; see
\cite[Theorem~8.1]{trefethen2013approximation}.

Applying $\Pi_\ell$ to
\eqref{eq:rellich-farfield-kernel-representation}, and using
Minkowski's and Cauchy--Schwarz inequalities, we obtain
\[
\begin{aligned}
    b_\ell
    &\leq
    C
    \int_{A_{R_0,R_0+\sigma}}
        |f_\xi(y)|
        \bigl\|
            \Pi_\ell
            \mathcal K_{\lambda_0}(y,\cdot)
        \bigr\|_{L^2(\mathbb S^{n-1})}
        \,dV_{\mathbb H}(y)
    \\
    &\leq
    C(1+\ell)^{p_0}e^{-\ell\tau_*}
    \|f_\xi\|_{L^2_{\mathbb H}
        (A_{R_0,R_0+\sigma})}.
\end{aligned}
\]
Therefore,
\begin{equation}\label{eq:rellich-single-degree-bound}
    b_\ell^2
    \leq
    C(1+\ell)^{2p_0}e^{-2\ell\tau_*}
    \|f_\xi\|_{L^2_{\mathbb H}
        (A_{R_0,R_0+\sigma})}^2.
\end{equation}

Choose a new exponent
\[
    Q_0>2p_0+1.
\]
Multiplying \eqref{eq:rellich-single-degree-bound} by
$(1+\ell)^{-Q_0}e^{2\ell\tau_*}$ and summing over $\ell$, we obtain
\begin{align}
    \sum_{\ell=0}^{\infty}
        (1+\ell)^{-Q_0}
        e^{2\ell\tau_*}b_\ell^2
    &\leq
    C
    \|f_\xi\|_{L^2_{\mathbb H}
        (A_{R_0,R_0+\sigma})}^2
    \sum_{\ell=0}^{\infty}
        (1+\ell)^{2p_0-Q_0}
    \notag\\
    &\leq
    C
    \|f_\xi\|_{L^2_{\mathbb H}
        (A_{R_0,R_0+\sigma})}^2.
\label{eq:rellich-source-coeff-bound}
\end{align}

By \eqref{eq:rellich-cutoff-commutator},
\[
    f_\xi=[\mathcal L_{\lambda_0},\chi]w_\xi,
\]
and \([\mathcal L_{\lambda_0},\chi]\) is a first-order differential operator
whose coefficients and support are fixed.
{
Proposition~\ref{prop:uniform-local-bound}, applied on the fixed annulus
\(A_{R_0,R_0+\sigma}\), and the identity
\(w_\xi=u_{\mathcal P,\xi}-u_{\mathcal P',\xi}\) give a constant
\(M>0\), depending only on the a priori data, the fixed annulus, and the fixed
\(h\), such that
}
\[
    \|w_\xi\|_{H^1_{\mathbb H}(A_{R_0,R_0+\sigma})}
    \leq M .
\]
Hence, after enlarging \(M\) if necessary,
\begin{equation}\label{eq:rellich-source-uniform-bound}
    \|f_\xi\|_{L^2_{\mathbb H}(A_{R_0,R_0+\sigma})}
    \leq
    C
    \|w_\xi\|_{H^1_{\mathbb H}
        (A_{R_0,R_0+\sigma})}
    \leq
    M.
\end{equation}
Combining
\eqref{eq:rellich-source-coeff-bound} and
\eqref{eq:rellich-source-uniform-bound}, we obtain
\begin{equation}\label{eq:rellich-coeff-weighted}
    \sum_{\ell=0}^{\infty}
        (1+\ell)^{-Q_0}
        e^{2\ell\tau_*}b_\ell^2
    \leq
    CM^2.
\end{equation}

\medskip
\noindent\emph{Step 6. Radial propagation estimate on the fixed annulus.}
We now estimate the radial factor \(\Phi_\ell^+\) appearing in
\eqref{eq:rellich-outgoing-expansion}.
We first derive its radial equation.

For a single separated mode
\(
    \Phi(\rho)Y_{\ell m}(\theta),
\)
the polar expression \eqref{eq:rellich-polar-laplacian}, together with the spherical harmonic identity
\(
    -\Delta_{\mathbb S^{n-1}}Y_{\ell m}
    =
    \ell(\ell+n-2)Y_{\ell m},
\)
gives
\[
\begin{aligned}
    \Delta_{\mathbb H}
    \bigl(\Phi(\rho)Y_{\ell m}(\theta)\bigr)
    &=
    \left[
        \Phi''(\rho)
        +(n-1)\coth\rho\,\Phi'(\rho)
        -
        \frac{\ell(\ell+n-2)}{\sinh^2\rho}\Phi(\rho)
    \right]
    Y_{\ell m}(\theta).
\end{aligned}
\]
Since \(\mathcal L_{\lambda_0}=-\Delta_{\mathbb H}-(n-1)^2/4-\lambda_0^2\), the equation
\(\mathcal L_{\lambda_0}(\Phi Y_{\ell m})=0\) is equivalent to
\[
    \Delta_{\mathbb H}(\Phi Y_{\ell m})
    +
    \left(
        \lambda_0^2+\frac{(n-1)^2}{4}
    \right)\Phi Y_{\ell m}
    =
    0.
\]
Thus the outgoing radial mode \(\Phi_\ell^+\) satisfies
\begin{equation}\label{eq:rellich-radial-equation}
    (\Phi_\ell^+)^{\prime\prime}
    +
    (n-1)\coth\rho\,(\Phi_\ell^+)'
    +
    \left[
        \lambda_0^2+\frac{(n-1)^2}{4}
        -
        \frac{\ell(\ell+n-2)}{\sinh^2\rho}
    \right]\Phi_\ell^+
    =
    0 .
\end{equation}
The last term is the angular momentum contribution of the \(\ell\)-th spherical
harmonic mode.

To remove the first-order derivative in \eqref{eq:rellich-radial-equation}, set
\begin{equation}\label{eq:rellich-liouville-transform}
    V_\ell(\rho)
    :=
    (\sinh\rho)^{\frac{n-1}{2}}\Phi_\ell^+(\rho).
\end{equation}
Substituting \eqref{eq:rellich-liouville-transform} into
\eqref{eq:rellich-radial-equation} and simplifying, we obtain
\begin{equation}
\label{eq:rellich-liouville-equation}
    V_\ell''
    +
    \left[
        \lambda_0^2
        -
        \frac{\nu_\ell^2-\frac14}{\sinh^2\rho}
    \right]V_\ell
    =
    0,
    \qquad
    \nu_\ell
    :=
    \ell+\frac{n-2}{2}.
\end{equation}

To justify the required dependence on $\ell$, we set
\(
    z:=\coth\rho.
\)
Since
\[
    \frac{dz}{d\rho}=-(z^2-1),
    \qquad
    \frac{1}{\sinh^2\rho}=z^2-1,
\]
equation \eqref{eq:rellich-liouville-equation} becomes
\[
    (z^2-1)\frac{d^2V_\ell}{dz^2}
    +
    2z\frac{dV_\ell}{dz}
    -
    \left[
        \left(\nu_\ell-\frac12\right)
        \left(\nu_\ell+\frac12\right)
        +
        \frac{(\mathrm i\lambda_0)^2}{z^2-1}
    \right]
    V_\ell
    =
    0.
\]
This is the associated Legendre equation. With the standard branch
on $z>1$, the outgoing solution is therefore of the form
\begin{equation}\label{eq:rellich-legendre-representation}
    V_\ell(\rho)
    =
    \gamma_{n,\lambda_0}
    P_{\nu_\ell-\frac12}^{\,\mathrm i\lambda_0}
    (\coth\rho),
\end{equation}
where $\gamma_{n,\lambda_0}\neq0$ is chosen according to
\eqref{eq:rellich-radial-normalization}. 
Indeed, as $\rho\to+\infty$, one has
$\coth\rho\to1^+$ and
\[
    P_{\nu_\ell-\frac12}^{\,\mathrm i\lambda_0}
    (\coth\rho)
    =
    \frac{e^{\mathrm i\lambda_0\rho}}
         {\Gamma(1-\mathrm i\lambda_0)}
    \bigl(1+o(1)\bigr).
\]
Thus, $P_{\nu_\ell-\frac12}^{\,\mathrm i\lambda_0}$ corresponds to the outgoing branch. 
Moreover, the leading coefficient in this asymptotic formula is independent of $\ell$, and hence the normalizing constant $\gamma_{n,\lambda_0}$ is independent of $\ell$.

Moreover, by \eqref{eq:tau-function}, one has:
\[
    \coth\rho=\cosh\tau(\rho),
    \qquad
    \tau(\rho)=\log\coth\frac{\rho}{2}.
\]
The uniform large-degree, fixed-order expansions for associated
Legendre functions, together with explicit remainder bounds, are
given in \cite[Theorems~3.1--3.2]{nemesOldeDaalhuis2020}.
Combining these expansions with the standard differential
recurrence relations for associated Legendre functions
\cite[Chapter~14, Section~14.10]{olverEtAl2010nist}, we obtain,
for every integer $s\geq0$, constants $C_s>0$ and $p_s\geq0$
such that
\[
    \left|
        \partial_\tau^k
        P_{\nu_\ell-\frac12}^{\,\mathrm i\lambda_0}
        (\cosh\tau)
    \right|
    \leq
    C_s(1+\ell)^{p_s}e^{\nu_\ell\tau},
    \qquad 0\leq k\leq s,
\]
uniformly for
\[
    \tau\in[\tau(\rho_2),\tau(\rho_1)]
\]
and all sufficiently large $\ell$.
By enlarging $C_s$, the finitely many remaining values of $\ell$
are absorbed into the same estimate.

Since $\tau$ is decreasing, all derivatives of $\tau(\rho)$ are
bounded on $[\rho_1,\rho_2]$, and
\[
    \Phi_\ell^+(\rho)
    =
    (\sinh\rho)^{-\frac{n-1}{2}}
    V_\ell(\rho),
\]
the chain rule and
\eqref{eq:rellich-legendre-representation} give
\[
\begin{aligned}
    \sup_{\rho\in[\rho_1,\rho_2]}
    \bigl|
        \partial_\rho^j\Phi_\ell^+(\rho)
    \bigr|
    &\leq
    C_s(1+\ell)^{p_s}
    e^{\nu_\ell\tau(\rho_1)}
    \\
    &\leq
    C_s(1+\ell)^{p_s}
    e^{\ell\tau_1},
    \qquad
    0\leq j\leq s,
\end{aligned}
\]
where the fixed factor
$e^{(n-2)\tau_1/2}$ has been absorbed into $C_s$. Thus
\begin{equation}\label{eq:rellich-radial-mode-bound}
    \sup_{\rho\in[\rho_1,\rho_2]}
    \bigl|
        \partial_\rho^j\Phi_\ell^+(\rho)
    \bigr|
    \leq
    C_s(1+\ell)^{p_s}e^{\ell\tau_1},
    \qquad
    0\leq j\leq s.
\end{equation}

\medskip
\noindent\emph{Step 7. Low--high angular mode interpolation.}

By \eqref{eq:rellich-coeff-small} and \eqref{eq:rellich-bell-definition} in Steps 3 and 5, we get the first estimation:
\begin{equation}\label{eq:rellich-bell-small}
    \sum_{\ell=0}^{\infty}b_\ell^2
    \leq
    \varepsilon^2.
\end{equation}
On the other hand, \eqref{eq:rellich-coeff-weighted} gives the second estimate:
\begin{equation}\label{eq:rellich-bell-weighted}
    \sum_{\ell=0}^{\infty}
        (1+\ell)^{-Q_0}
        e^{2\ell\tau_*}b_\ell^2
    \leq
    CM^2.
\end{equation}

Let $q>0$ be fixed, to be chosen sufficiently large in Step~8.
Set
\[
    L_\varepsilon:=-\log\varepsilon,
    \qquad
    N_\varepsilon
    :=
    \left\lfloor L_\varepsilon^{1/2}\right\rfloor.
\]
We choose $\varepsilon_0>0$ sufficiently small so that
$L_\varepsilon\geq4$ for $0<\varepsilon<\varepsilon_0$.

For the low angular modes, by \eqref{eq:rellich-bell-small}, one has
\begin{equation}\label{eq:rellich-low-mode-bound}
    \sum_{\ell\leq N_\varepsilon}
        (1+\ell)^q e^{2\ell\tau_1}b_\ell^2
    \leq
    (1+N_\varepsilon)^q
    e^{2N_\varepsilon\tau_1}
    \varepsilon^2
    \leq
    C\exp\left(-cL_\varepsilon^{1/2}\right)
\end{equation}
For the high angular modes, we write
\[
\begin{aligned}
    (1+\ell)^q e^{2\ell\tau_1}
    &=
    (1+\ell)^{q+Q_0}
    e^{-2\ell(\tau_*-\tau_1)}
    \left[
        (1+\ell)^{-Q_0}
        e^{2\ell\tau_*}
    \right].
\end{aligned}
\]
Since $\tau_*-\tau_1>0$, there exist constants $C,c>0$
such that
\[
    \sup_{\ell>N_\varepsilon}
    (1+\ell)^{q+Q_0}
    e^{-2\ell(\tau_*-\tau_1)}
    \leq
    Ce^{-cN_\varepsilon}.
\]
Therefore, by \eqref{eq:rellich-bell-weighted}, one has:
\begin{align}
    \sum_{\ell>N_\varepsilon}
        (1+\ell)^q e^{2\ell\tau_1}b_\ell^2
    &\leq
    Ce^{-cN_\varepsilon}
    \sum_{\ell>N_\varepsilon}
        (1+\ell)^{-Q_0}
        e^{2\ell\tau_*}b_\ell^2
    \notag\\
    &\leq
    CM^2e^{-cN_\varepsilon}
    \notag\\
    &\leq
    C\exp\left(-cL_\varepsilon^{1/2}\right).
\label{eq:rellich-high-mode-bound}
\end{align}

Combining
\eqref{eq:rellich-low-mode-bound} and
\eqref{eq:rellich-high-mode-bound}, we conclude that
\begin{equation}\label{eq:rellich-interpolated-coeff-bound}
    \sum_{\ell=0}^{\infty}
        (1+\ell)^q e^{2\ell\tau_1}b_\ell^2
    \leq
    C\exp\left[
        -c\sqrt{-\log\varepsilon}
    \right].
\end{equation}

\medskip
\noindent\emph{Step 8. From coefficients to the $C^1$-estimate.}

Choose an integer
\[
    s_0>\frac n2+1.
\]
For radial derivatives of order $j$ and angular derivatives of
order $k$, with $j+k\leq s_0$, the estimate
\eqref{eq:rellich-radial-mode-bound} and the spectral
characterization of Sobolev norms on $\mathbb S^{n-1}$ give
\[
    \|w_\xi\|_{H^{s_0}_{\mathbb H}
        (A_{\rho_1,\rho_2})}^2
    \leq
    C
    \sum_{\ell=0}^{\infty}
        (1+\ell)^q
        e^{2\ell\tau_1}b_\ell^2,
\]
provided $q$ is chosen sufficiently large to dominate all polynomial factors arising from the radial and angular derivatives.
For instance, it is enough to choose
\[
    q\geq
    \max_{j+k\leq s_0}(2k+2p_j).
\]

Consequently, by \eqref{eq:rellich-interpolated-coeff-bound}, taking square roots and renaming $C$ and $c$, we obtain
\[
    \|w_\xi\|_{H^{s_0}_{\mathbb H}
        (A_{\rho_1,\rho_2})}
    \leq
    C\exp\left[
        -c\sqrt{-\log\varepsilon}
    \right].
\]
Since $s_0>n/2+1$, Sobolev embedding on the fixed smooth annulus
gives
\[
    \|w_\xi\|_{C^1(A_{\rho_1,\rho_2})}
    \leq
    C\exp\left[
        -c\sqrt{-\log\varepsilon}
    \right].
\]
Finally, by \eqref{eq:rellich-exterior-solution-difference}, one has $
    w_\xi
    =
    u_{\mathcal P,\xi}
    -
    u_{\mathcal P',\xi}$ on $A_{\rho_1,\rho_2}$.
Therefore, \eqref{eq:quantitative-rellich-estimate} follows.

The proof is complete.
\end{proof}

\subsection{Propagation of smallness through geodesic annuli}


We now prepare the local tools used to propagate far-field errors through the exterior domain.
The basic idea is to turn smallness on one hyperbolic ball into quantitative smallness on nearby balls for solutions of \(\mathcal L_{\lambda_0}u=0\).
This local propagation mechanism is encoded in a three-spheres inequality and will later be iterated along chains of admissible balls.

We use the notation introduced above.
In particular, \(\mathcal L_{\lambda_0}\) is the shifted Helmholtz operator fixed above, and \(\Lambda_0\) is defined by \eqref{eq:Lambda0-definition}.
When the center is displayed, \(B_{\mathbb H}(x,r)\) denotes the hyperbolic ball centered at \(x\); as before, \(B_{\mathbb H}(r)=B_{\mathbb H}(0,r)\).
The following estimate is a hyperbolic local version of the classical \emph{three-spheres theorem} for second-order elliptic equations.
It is obtained by writing \(\mathcal L_{\lambda_0}\) in geodesic normal coordinates and checking that the resulting elliptic operators satisfy uniform local bounds.
Hence the constants are uniform with respect to the centers of the balls.

\begin{lemma}\label{lem:local-three-spheres}
There exist constants \(\rho_*>0\), \(C>0\), and \(c_1\in(0,1)\), depending only
on \(n\) and \(\lambda_0\), such that the following holds.

Let \(x_0\in\mathbb B^n\), and let
\[
    w\in H^1_{\mathrm{loc}}\bigl(B_{\mathbb H}(x_0,\rho_2)\bigr)
    \cap L^\infty\bigl(B_{\mathbb H}(x_0,\rho_2)\bigr)
\]
be a weak solution of \(\mathcal L_{\lambda_0}w=0\) in \(B_{\mathbb H}(x_0,\rho_2)\).
Assume that \(0<\rho_1<\rho<s<\rho_2\leq \rho_*\).
Then
\begin{equation}\label{eq:three}
    \|w\|_{L^\infty(B_{\mathbb H}(x_0,\rho))}
    \leq
    C\left(1-\frac{\rho}{s}\right)^{-n/2}
    \|w\|_{L^\infty(B_{\mathbb H}(x_0,\rho_2))}^{1-\beta}
    \|w\|_{L^\infty(B_{\mathbb H}(x_0,\rho_1))}^{\beta},
\end{equation}
where \(\beta=\beta(\rho_1,\rho_2,s)\in(0,1)\) satisfies
\[
    c_1
    \frac{\log(\rho_2/s)}{\log(\rho_2/\rho_1)}
    \leq
    \beta
    \leq
    1-
    c_1
    \frac{\log(s/\rho_1)}{\log(\rho_2/\rho_1)} .
\]
\end{lemma}

\begin{proof}
By local elliptic regularity, \(w\) is smooth in \(B_{\mathbb H}(x_0,\rho_2)\).
Since
\[
    \mathcal L_{\lambda_0}
    =
    -\Delta_{\mathbb H}
    -
    \Lambda_0,
    \qquad
    \Lambda_0
    =
    \frac{(n-1)^2}{4}+\lambda_0^2,
\]
the equation \(\mathcal L_{\lambda_0}w=0\) is equivalent to
\(\Delta_{\mathbb H}w+\Lambda_0 w=0\).

We work in geodesic normal coordinates centered at \(x_0\).
Let \(y=\exp_{x_0}z\) and \(W(z):=w(\exp_{x_0}z)\).
After choosing an orthonormal basis of \(T_{x_0}\mathbb B^n\), the ball
\(B_{\mathbb H}(x_0,r)\) is identified with the Euclidean ball
\(B_r(0)\subset\mathbb R^n\).
In these coordinates, \(W\) satisfies
\begin{equation}\label{eq:normal-coordinate-equation}
    L_{x_0}W=0
    \quad\text{in }B_{\rho_2}(0),
\end{equation}
where
\[
    L_{x_0}W
    =
    a^{ij}_{x_0}(z)\partial_{ij}W
    +
    b^i_{x_0}(z)\partial_iW
    +
    \Lambda_0 W .
\]
The coefficients \(a^{ij}_{x_0}\) and \(b^i_{x_0}\) are determined by the
hyperbolic metric in normal coordinates.

Choose \(\rho_*>0\) sufficiently small.
Since hyperbolic space is homogeneous, the operators \(L_{x_0}\), with
\(x_0\in\mathbb B^n\), are uniformly elliptic in \(B_{\rho_*}(0)\), and the required
\(C^2\)-bounds for their coefficients are uniform with respect to \(x_0\).
These bounds depend only on \(n\), while the zero-order term is controlled by
\(n\) and \(\lambda_0\).

We may therefore apply the classical local \(L^\infty\) three-spheres theorem for
second-order elliptic equations to \eqref{eq:normal-coordinate-equation}; see
\cite[Theorem~6.3]{brummelhuis1995three}.
It gives, for \(0<\rho_1<\rho<s<\rho_2\leq\rho_*\),
\[
    \|W\|_{L^\infty(B_\rho(0))}
    \leq
    C\left(1-\frac{\rho}{s}\right)^{-n/2}
    \|W\|_{L^\infty(B_{\rho_2}(0))}^{1-\beta}
    \|W\|_{L^\infty(B_{\rho_1}(0))}^{\beta},
\]
where \(C>0\) and \(\beta\in(0,1)\) depend only on the uniform ellipticity and
coefficient bounds above.
In particular, the constants are independent of the center \(x_0\).
Moreover, after possibly decreasing \(\rho_*\), the exponent \(\beta\) satisfies
\[
    c_1
    \frac{\log(\rho_2/s)}{\log(\rho_2/\rho_1)}
    \leq
    \beta
    \leq
    1-
    c_1
    \frac{\log(s/\rho_1)}{\log(\rho_2/\rho_1)}
\]
for some \(c_1\in(0,1)\) depending only on \(n\) and \(\lambda_0\).

Finally, the map \(z\mapsto \exp_{x_0}z\) sends \(B_r(0)\) onto
\(B_{\mathbb H}(x_0,r)\).
Thus, the preceding estimate is precisely \eqref{eq:three}.

The proof is complete.
\end{proof}

We now record the standard propagation of smallness obtained by iterating the
local three-spheres inequality along a chain of overlapping hyperbolic balls.

For a hyperbolic ball
\[
    B_i:=B_{\mathbb H}(x_i,r_i),
\]
and for \(\theta>0\), we write
\[
    B_i^\theta:=B_{\mathbb H}(x_i,\theta r_i).
\]

\begin{definition}\label{def:regular-hyperbolic-chain}
Let \(G\subset\mathbb B^n\) be open, and fix constants
\[
    0<a<b<1<A .
\]
A finite ordered family of hyperbolic balls
\[
    B_i=B_{\mathbb H}(x_i,r_i),
    \qquad i=0,\ldots,N,
\]
is called an \emph{\((a,b,A)\)-regular hyperbolic chain} in \(G\) if the following two
conditions hold.

\begin{itemize}
	\item First, each ball has room inside \(G\):
\[
    B_i^A\subset G,
    \qquad i=0,\ldots,N.
\]
\item Second, consecutive balls overlap with a fixed margin:
\[
    B_{i+1}^a\subset B_i^b,
    \qquad i=0,\ldots,N-1.
\]
\end{itemize}
\end{definition}

The following lemma is obtained by iterating Lemma~\ref{lem:local-three-spheres} along a regular hyperbolic chain.

\begin{lemma}\label{lem:propagation-chain}
Let \(G\subset\mathbb B^n\) be open, and let
\(w\in H^1_{\mathrm{loc}}(G)\) solve
\[
    \mathcal L_{\lambda_0}w=0
    \quad\text{in }G .
\]
Let \(B_i=B_{\mathbb H}(x_i,r_i)\), \(i=0,\ldots,N\), be an
\((a,b,A)\)-regular hyperbolic chain in \(G\), in the sense of
Definition~\ref{def:regular-hyperbolic-chain}.
Assume also that \(A r_i\leq \rho_*\) for \(i=0,\ldots,N\).
Assume that
\[
    \max_{0\leq i\leq N}\|w\|_{L^\infty(B_i^A)}\leq E,
    \qquad
    \|w\|_{L^\infty(B_0^a)}\leq \varepsilon,
    \qquad
    0<\varepsilon\leq E .
\]
Then there exist constants \(C_0\geq 1\) and \(\gamma\in(0,1)\), depending only on
\(n,\lambda_0,a,b,A\), such that
\begin{equation}\label{eq:chain-propagation}
	 \|w\|_{L^\infty(B_N^a)}
    \leq
    C_0 E
    \left(\frac{\varepsilon}{E}\right)^{\gamma^N}.
\end{equation}
\end{lemma}

\begin{proof}
Set
\[
    M_i:=\|w\|_{L^\infty(B_i^a)},
    \qquad i=0,\ldots,N .
\]
Then \(M_0\leq\varepsilon\). Fix \(\sigma:=(b+A)/2\).

For each \(i=0,\ldots,N-1\), the overlap condition gives
\(B_{i+1}^a\subset B_i^b\). Hence
\[
    M_{i+1}
    \leq
    \|w\|_{L^\infty(B_i^b)} .
\]
We apply Lemma~\ref{lem:local-three-spheres} in the ball \(B_i^A\), with
\[
    \rho_1=ar_i,
    \qquad
    \rho=br_i,
    \qquad
    s=\sigma r_i,
    \qquad
    \rho_2=Ar_i .
\]
This is allowed because \(B_i^A\subset G\) and \(Ar_i\leq\rho_*\).
Since the ratios \(a,b,\sigma,A\) are fixed, Lemma~\ref{lem:local-three-spheres}
gives constants \(C_*\geq1\) and \(\gamma\in(0,1)\), depending only on
\(n,\lambda_0,a,b,A\), such that
\[
    \|w\|_{L^\infty(B_i^b)}
    \leq
    C_* E^{1-\beta_i}M_i^{\beta_i},
    \qquad
    \beta_i\geq\gamma .
\]
Since \(M_i\leq E\), this implies the one-step estimate
\begin{equation}\label{eq:chain-one-step}
    M_{i+1}
    \leq
    C_*E\left(\frac{M_i}{E}\right)^\gamma,
    \qquad i=0,\ldots,N-1 .
\end{equation}

Let \(m_i:=M_i/E\). Dividing \eqref{eq:chain-one-step} by \(E\), we get
\[
    m_{i+1}\leq C_*m_i^\gamma,
    \qquad
    m_0\leq \frac{\varepsilon}{E}.
\]
Iterating this inequality gives
\begin{equation}\label{eq:chain-iteration}
    m_N
    \leq
    C_*^{1+\gamma+\cdots+\gamma^{N-1}}
    m_0^{\gamma^N}.
\end{equation}
Since \(1+\gamma+\cdots+\gamma^{N-1}\leq(1-\gamma)^{-1}\), \eqref{eq:chain-iteration}
yields
\[
    M_N
    \leq
    C_*^{1/(1-\gamma)}
    E
    \left(\frac{\varepsilon}{E}\right)^{\gamma^N}.
\]
The conclusion \eqref{eq:chain-propagation} follows by taking \(C_0:=C_*^{1/(1-\gamma)}\).
The proof is complete.
\end{proof}

\subsection{Quantitative reflection across totally geodesic hypersurfaces}\label{subsec:quantitative-reflection}

We shall also need a quantitative form of the reflection principle across totally geodesic hypersurfaces.
The reflection principle in hyperbolic space has been introduced and analyzed in Section~\ref{sec:reflection-principles}.
The relevant normal transformation and exact reflection identities have been established in Lemmas~\ref{lem:out-normal}--\ref{lem:reflection-neumann}.

For a $C^1$ function $U$ defined near a relatively open subset of a totally geodesic hypersurface $V$, we denote by
\[
    \nabla_{\mathbb H}^T U
    :=
    \nabla_{\mathbb H}U
    -
    \bigl\langle
        \nabla_{\mathbb H}U,\nu_{\mathbb H}
    \bigr\rangle_{\mathbb H}\nu_{\mathbb H}
\]
the tangential component of the hyperbolic gradient along $V$, where $\nu_{\mathbb H}$ is either choice of a unit hyperbolic normal to $V$.
This definition is independent of the choice of the sign of $\nu_{\mathbb H}$.

In the sequel, we shall use this quantitative reflection principle to control the solution near the reflecting hypersurface.

\begin{lemma}\label{lem:quant-reflection}
Let \(V\subset\mathbb B^n\) be a totally geodesic hypersurface defined in Definition~\ref{def:tg-hypersurface-ball}.
Let \(I_V\) be the hyperbolic reflection with respect to \(V\), defined in
Definition~\ref{def:reflection-general-hypersurface}.
There exist constants \(r_*>0\), \(C>0\), and \(\vartheta\in(0,1)\), depending only on
\(n\) and \(\lambda_0\), with the following property.

Let \(p\in V\), \(0<r\leq r_*\), and set
\[
    \Sigma_r:=V\cap B_{\mathbb H}(p,r).
\]
Let
\[
    w\in H^1_{\mathbb H,\mathrm{loc}}\bigl(B_{\mathbb H}(p,r)\bigr)
    \cap L^\infty\bigl(B_{\mathbb H}(p,r)\bigr)
\]
be a weak solution of
\[
    \mathcal L_{\lambda_0}w=0
    \quad
    \text{in }B_{\mathbb H}(p,r).
\]
Assume that
\[
    \|w\|_{L^\infty(B_{\mathbb H}(p,r))}\leq E.
\]

If \(w\) has small Dirichlet-type trace data on \(\Sigma_r\), namely
\[
    \|w\|_{L^\infty(\Sigma_r)}
    +
    r\|\nabla_{\mathbb H}^T w\|_{L^\infty(\Sigma_r)}
    \leq
    \varepsilon,
\]
then
\begin{equation}
\label{eq:qr-odd-estimate}
    \|w+w\circ I_V\|_{L^\infty(B_{\mathbb H}(p,r/2))}
    \leq
    C E^{1-\vartheta}\varepsilon^\vartheta .
\end{equation}

If \(w\) has small Neumann-type normal data on \(\Sigma_r\), namely
\[
    r\|\partial_{\nu_{\mathbb H}}w\|_{L^\infty(\Sigma_r)}
    \leq
    \varepsilon,
\]
then
\begin{equation}
\label{eq:qr-even-estimate}
    \|w-w\circ I_V\|_{L^\infty(B_{\mathbb H}(p,r/2))}
    \leq
    C E^{1-\vartheta}\varepsilon^\vartheta .
\end{equation}
\end{lemma}

\begin{proof}
\medskip
\noindent\emph{Step 1. Reflection preserves the local equation.}
By interior elliptic regularity, \(w\) is smooth in \(B_{\mathbb H}(p,r)\), and hence all traces appearing below are well defined.
Since \(p\in V\), we have \(I_V(p)=p\). As \(I_V\) is a hyperbolic isometry,
\[
    I_V\bigl(B_{\mathbb H}(p,r)\bigr)=B_{\mathbb H}(p,r).
\]
Thus \(w\circ I_V\) is well defined in \(B_{\mathbb H}(p,r)\).
Since \(I_V\) is a hyperbolic isometry and \(\Delta_{\mathbb H}\) is invariant under isometries, the conclusion drawn in Lemma~\ref{lem3} implies that
\[
    \mathcal L_{\lambda_0}(w\circ I_V)
    =
    (\mathcal L_{\lambda_0}w)\circ I_V
    =
    0 .
\]

\medskip
\noindent\emph{Step 2. The odd reflection residual.}
Define \(W_+:=w+w\circ I_V\).
Then, it satisfies \(\mathcal L_{\lambda_0}W_+=0\) in \(B_{\mathbb H}(p,r)\), and
\[
    \|W_+\|_{L^\infty(B_{\mathbb H}(p,r))}\leq 2E .
\]

\medskip
\noindent\emph{Step 3. Cauchy data of the odd residual on \(V\).}
Since \(I_V\) fixes every point of \(V\), one has \(W_+=2w\) on \(\Sigma_r\).
Since \(dI_V\) is the identity on tangential vectors to \(V\), one also has
\[
    \nabla_{\mathbb H}^T W_+
    =
    2\nabla_{\mathbb H}^T w
    \quad\text{on }\Sigma_r .
\]
By the normal transformation formula in Lemma~\ref{lem:out-normal},
\[
    \partial_{\nu_{\mathbb H}}(w\circ I_V)
    =
    -\partial_{\nu_{\mathbb H}}w
    \quad\text{on }V.
\]
Hence \(\partial_{\nu_{\mathbb H}}W_+=0\) on \(\Sigma_r\).
Combining these identities with the Dirichlet-type trace smallness assumption gives
\begin{equation}
\label{eq:qr-Wplus-cauchy}
    \|W_+\|_{L^\infty(\Sigma_r)}
    +
    r\|\nabla_{\mathbb H}^T W_+\|_{L^\infty(\Sigma_r)}
    +
    r\|\partial_{\nu_{\mathbb H}}W_+\|_{L^\infty(\Sigma_r)}
    \leq
    C\varepsilon .
\end{equation}

\medskip
\noindent\emph{Step 4. Cauchy stability for the odd residual.}
We first reduce the pair $(p,V)$ to a standard position by using the M\"obius isometries $T_a$ defined in \eqref{eq:Mobius}.
By Definition~\ref{def:tg-hypersurface-ball}, there exist $a\in\mathbb B^n$ and $\nu\in\mathbb S^{n-1}$ such that
\[
    V=T_a(P_\nu),
    \qquad
    P_\nu:=\{x\in\mathbb B^n:x\cdot\nu=0\}.
\]
Set
\[
    b:=T_a^{-1}(p).
\]
Then $b\in P_\nu$.
Since $b\in P_\nu$, it follows directly from \eqref{eq:Mobius} that
\[
    T_b(P_\nu)=P_\nu .
\]
Choose an orthogonal map $Q$ such that $Q\nu=e_n$, and define
\[
    \Phi:=Q\circ T_b\circ T_a^{-1}.
\]
Then $\Phi$ is a hyperbolic isometry and
\[
    \Phi(p)=0,
    \qquad
    \Phi(V)=V_0:=\{x_n=0\}\cap\mathbb B^n .
\]
By the uniqueness of the hyperbolic reflection with respect to a totally geodesic hypersurface, we also have
\[
    \Phi\circ I_V\circ\Phi^{-1}=I_0,
    \qquad
    I_0(x',x_n)=(x',-x_n).
\]

Let $U$ be any solution of $\mathcal L_{\lambda_0}U=0$ in $B_{\mathbb H}(p,r)$, and set
\[
    \widetilde U:=U\circ\Phi^{-1}.
\]
Since $\Phi$ is a hyperbolic isometry and $\mathcal L_{\lambda_0}$ is invariant under hyperbolic isometries,
\[
    \mathcal L_{\lambda_0}\widetilde U=0
    \quad\text{in }B_{\mathbb H}(0,r).
\]
Moreover,
\[
    \Phi(\Sigma_r)=V_0\cap B_{\mathbb H}(0,r).
\]
The $L^\infty$ norm of $\widetilde U$ on $B_{\mathbb H}(0,r)$, as well as the $L^\infty$ norms of its tangential and normal hyperbolic Cauchy data on $V_0\cap B_{\mathbb H}(0,r)$, coincide with the corresponding quantities for $U$ on $\Sigma_r$.

In the standard Poincar\'e coordinates near the origin, $\mathcal L_{\lambda_0}$ is a uniformly elliptic second-order operator with smooth coefficients.
Choosing $r_*>0$ sufficiently small and scaling the coordinate ball to a unit ball, the ellipticity constants and the required coefficient bounds are controlled only by $n$ and $\lambda_0$.
By the local pointwise Cauchy stability estimates for second-order uniformly elliptic equations \cite{trytten1963pointwise,payne1975improperly}, applied in the standard Poincar\'e coordinates near the origin and then pulled back by $\Phi$, we obtain
\begin{equation}
\label{eq:qr-local-cauchy}
    \|U\|_{L^\infty(B_{\mathbb H}(p,r/2))}
    \leq
    C
    \|U\|_{L^\infty(B_{\mathbb H}(p,r))}^{1-\vartheta}
    \mathcal C_{\Sigma_r}(U)^\vartheta,
\end{equation}
where the estimate is applied separately on the two sides of $V$, and
\[
    \mathcal C_{\Sigma_r}(U)
    :=
    \|U\|_{L^\infty(\Sigma_r)}
    +
    r\|\nabla_{\mathbb H}^T U\|_{L^\infty(\Sigma_r)}
    +
    r\|\partial_{\nu_{\mathbb H}}U\|_{L^\infty(\Sigma_r)}.
\]

Applying \eqref{eq:qr-local-cauchy} with $U=W_+$, and using \eqref{eq:qr-Wplus-cauchy} together with
\[
    \|W_+\|_{L^\infty(B_{\mathbb H}(p,r))}\leq 2E,
\]
we obtain
\[
    \|W_+\|_{L^\infty(B_{\mathbb H}(p,r/2))}
    \leq
    C E^{1-\vartheta}\varepsilon^\vartheta .
\]
By the definition of $W_+$, this is precisely \eqref{eq:qr-odd-estimate}.

\medskip\noindent\emph{Step 5. The even reflection residual.}
Define \(W_-:=w-w\circ I_V\).
Then \(\mathcal L_{\lambda_0}W_-=0\) in \(B_{\mathbb H}(p,r)\), and
\[
    \|W_-\|_{L^\infty(B_{\mathbb H}(p,r))}\leq 2E .
\]

\medskip
\noindent\emph{Step 6. Cauchy data of the even residual on \(V\).}
Since \(I_V(y)=y\) for every \(y\in V\), one has \(W_-=0\) on \(\Sigma_r\).
Consequently, \(\nabla_{\mathbb H}^T W_-=0\) on \(\Sigma_r\).
Using again the normal transformation formula in Lemma~\ref{lem:out-normal}, we obtain
\[
    \partial_{\nu_{\mathbb H}}W_-
    =
    2\partial_{\nu_{\mathbb H}}w
    \quad\text{on }\Sigma_r .
\]
Combining these identities with the Neumann-type normal smallness assumption gives
\begin{equation}
\label{eq:qr-Wminus-cauchy}
    \|W_-\|_{L^\infty(\Sigma_r)}
    +
    r\|\nabla_{\mathbb H}^T W_-\|_{L^\infty(\Sigma_r)}
    +
    r\|\partial_{\nu_{\mathbb H}}W_-\|_{L^\infty(\Sigma_r)}
    \leq
    C\varepsilon .
\end{equation}

\medskip
\noindent\emph{Step 7. Cauchy stability for the even residual.}
Applying \eqref{eq:qr-local-cauchy} with \(U=W_-\), and using
\eqref{eq:qr-Wminus-cauchy} together with
\(\|W_-\|_{L^\infty(B_{\mathbb H}(p,r))}\leq 2E\), we obtain
\[
    \|W_-\|_{L^\infty(B_{\mathbb H}(p,r/2))}
    \leq
    C E^{1-\vartheta}\varepsilon^\vartheta .
\]
By the definition of \(W_-\), this is precisely \eqref{eq:qr-even-estimate}.

The proof is complete.
\end{proof}


\section{Stable determination of totally geodesic defects}\label{sec:quantitative-stability}

In this section, we prove the quantitative stability estimates for the far-field inverse problem at conformal infinity \eqref{eq:hip1}, as stated in Theorems~\ref{thm:main-soft-farfield} and~\ref{thm:main-hard-farfield-simple}.
Our main point is that the far-field error can be quantitatively converted into geometric control of the totally geodesic defects.
This is achieved by combining the quantitative estimates developed in Sections~\ref{sec:quant-prelim} and \ref{sec:quantitative-rellich}.

Throughout this section, the admissible class refers to \(\mathcal A_{\mathbb H}^h\) in the \(\mathsf D\)-type case and to \(\mathcal B_{\mathbb H}^h\) in the \(\mathsf N\)-type case, as defined in Definitions~\ref{def:admissible-tg-defects} and~\ref{def:hard-admissible-tg-defects}, respectively.
The a priori data are the parameters appearing in these definitions.

{
Throughout this section, \(h\in(0,h_0]\) is fixed.  Unless explicitly stated
otherwise, the auxiliary constants denoted by \(C,c,C_j,c_j\) are positive,
may change from line to line, and may depend on \(h\), on \(n\), on
\(\lambda_0\), and on the a priori data of the corresponding admissible class.
When relevant, they may also depend on the boundary label \(\xi\) or on the
admissible family \(\Xi\), but they are independent of the particular defects
and the far-field error.  The geometric chain constant \(C_{\rm geo}\) and the
propagation exponent \(\beta\) are independent of \(h\).  Any \(h\)-dependence
of auxiliary prefactors and lower bounds is absorbed into the admissible error
thresholds; the constants \(C\) and \(\alpha\) in the final stability estimates
remain independent of \(h\).
}

\subsection{\texorpdfstring{Stable determination of \(\mathsf N\)-type totally geodesic defects}{Stable determination of N-type totally geodesic defects}}\label{subsec:stable-hard-tg-defects}

We now prove Theorem~\ref{thm:main-hard-farfield-simple}.
Let \(0<h\leq h_0\), and let \(\mathcal P,\mathcal P'\in\mathcal B_{\mathbb H}^h\) be two admissible \(\mathsf N\)-type totally geodesic defects defined in Definition~\ref{def:hard-admissible-tg-defects}.
Let \(\Xi=(\xi_1,\ldots,\xi_m)\in\bigl(\mathbb S^{n-1}\bigr)^m\) be an admissible family of boundary labels in the sense of Definition~\ref{def:hard-directions}.
For each \(\ell=1,\ldots,m\), write
\begin{equation}\label{eq:hard-exterior-solutions-section-six}
    u^i_{\xi_\ell}:=e_{2\lambda_0,\xi_\ell},
    \qquad
    u_\ell:=u_{\mathcal P,\xi_\ell}=u^i_{\xi_\ell}+u^s_{\mathcal P,\xi_\ell},
    \qquad
    u'_\ell:=u_{\mathcal P',\xi_\ell}=u^i_{\xi_\ell}+u^s_{\mathcal P',\xi_\ell}.
\end{equation}
Define the \emph{far-field error} by
\begin{equation}\label{eq:hard-farfield-error-section-six}
    \varepsilon
    :=
    \max_{\ell=1,\ldots,m}
    \bigl\|
        u_{\infty,\mathcal P,\xi_\ell}
        -
        u_{\infty,\mathcal P',\xi_\ell}
    \bigr\|_{L^2(\mathbb S^{n-1})}.
\end{equation}
Assume that \(0<\varepsilon\leq\varepsilon_{\mathsf N}(h)\).

Set
\begin{equation*}
    G:=\mathbb B^n\setminus\mathcal P,
    \qquad
    G':=\mathbb B^n\setminus\mathcal P',
\end{equation*}
and let \(H\) be the unbounded connected component of \(G\cap G'\).
The final estimate is stated in terms of the hyperbolic Hausdorff distance \(d_{\mathcal H}^{\mathbb H}(\mathcal P,\mathcal P')\).
In the proof, we use the modified boundary distance \(d(\mathcal P,\mathcal P')\) from Definition~\ref{def:hyperbolic-distances}, which is adapted to propagation through the exterior domain.
Set 
\begin{equation*}
    d:=d(\mathcal P,\mathcal P'),
    \qquad
    \bar h:=\min\{d,h\}.
\end{equation*}
If \(d=0\), then the modified boundary discrepancy has already vanished.
Hence, in the rest of the proof, we assume \(d>0\).

We fix an observation ball
\begin{equation}\label{eq:hard-observation-ball-section-six}
    B_{\mathbb H}(x_0,\widehat\rho)
    \subset
    \mathbb B^n\setminus \overline{B_{\mathbb H}(R)},
\end{equation}
where \(R\) is the a priori radius in Definition~\ref{def:admissible-tg-defects}.
Since every admissible defect is contained in \(\overline{B_{\mathbb H}(R)}\), this ball lies in the common exterior \(H\) and serves as the initial region for propagating the smallness obtained from the far-field error.

\medskip
The following proposition is the quantitative reflected-continuation step for the
\(\mathsf N\)-type case.  It starts directly from the far-field error
\eqref{eq:hard-farfield-error-section-six}.  The far-field-to-near-field conversion is included in the proof.
We also use the far-field-to-near-field modulus supplied by Lemma~\ref{lem:quantitative-rellich}, namely
\begin{equation}\label{eq:Psi-ff-definition-section-six}
    \Psi(t)
    :=
    C\exp\!\left(-c(-\log t)^{1/2}\right).
\end{equation}
{
Here \(c>0\) depends only on the corresponding a priori data and
the fixed observation ball.
}

\begin{proposition}\label{prop:hard-quant-reflected-continuation}
Let \(d=d(\mathcal P,\mathcal P')\) be the modified boundary distance and set \(\bar h:=\min\{d,h\}\).
Assume that \(\bar h>0\).
Let \(U=(u_1,\ldots,u_m)\) and \(U'=(u'_1,\ldots,u'_m)\) be defined by \eqref{eq:hard-exterior-solutions-section-six}.
Let \(\varepsilon\) be defined by \eqref{eq:hard-farfield-error-section-six}.
There exist constants
\[
    \begin{gathered}
        R_*>R,\qquad T_*>0,\qquad C\geq1,\qquad C_{\rm geo}\geq1,\\
        \beta\in(0,1),\qquad \varepsilon_0\in(0,e^{-1}).
    \end{gathered}
\]
{
The geometric exponent \(\beta\) and the constants in the bounds for \(J\) and
\(M\) depend only on the stated a priori data; \(\varepsilon_0\) may also
depend on \(h\).
}
They have the following property.
{
The far annulus may be prescribed in advance: the same conclusion holds for
every fixed \(T_*>0\) and every \(R_*\) larger than an a priori radius, with
the constants also allowed to depend on this prescribed annulus.
}

If
\[
    0<\varepsilon\leq\varepsilon_0,
\]
then, after possibly interchanging the roles of
\((\mathcal P,U)\) and \((\mathcal P',U')\), there exist integers \(J,M\geq1\) and
a totally geodesic hypersurface \(V_J\), intersecting a fixed compact hyperbolic
ball depending only on the a priori data, such that
\begin{equation}
\label{eq:hard-reflected-continuation-upper}
    \bar h
    \max_{\ell=1,\ldots,m}
    \sup_{x\in V_J\cap
    \left(B_{\mathbb H}(R_*+T_*)\setminus B_{\mathbb H}(R_*)\right)}
    e^{\frac{n+1}{2}\rho(x)}
    \left|
        \partial_{\nu_{\mathbb H}}\Big|_{V_J}u_\ell(x)
    \right|
    \leq
    C^M\bigl(\Psi(\varepsilon)\bigr)^{\beta^M}.
\end{equation}
Here, \(\partial_{\nu_{\mathbb H}}\big|_{V_J}\) denotes the hyperbolic normal derivative along \(V_J\), with either choice of the unit hyperbolic normal.
The left-hand side is independent of this choice.
The function \(\Psi\) is the far-field-to-near-field modulus given by \eqref{eq:Psi-ff-definition-section-six}; up to changing the constants,
\[
    \Psi(\varepsilon)
    =
    C\exp\!\left(-c(-\log\varepsilon)^{1/2}\right).
\]
{
Here \(J\) is the number of reflection stages, whereas \(M\) is the total
number of elementary propagation and reflection steps.  Moreover,
\[
    J\leq C_{\rm geo}\log\frac{2eR}{d}.
\]
}
In addition,
\begin{equation*}
    M
    \leq
    C_{\rm geo}
    \log\frac{2eR}{d}
    \left(
        \log\frac{2eR}{d}
        +
        \log\frac{2eR}{\bar h}
    \right).
\end{equation*}
\end{proposition}

\begin{proof}
We divide the proof into five steps.

\smallskip
\noindent
\emph{Step 1. Far-field error, near-field error, and uniform a priori bounds.}

For each \(\ell=1,\ldots,m\), set
\[
    w_\ell:=u_\ell-u'_\ell
    =
    u^s_{\mathcal P,\xi_\ell}
    -
    u^s_{\mathcal P',\xi_\ell}.
\]
Since the two exterior solutions are generated by the same incoming Helgason mode, \(w_\ell\) is an outgoing solution of
\[
    \mathcal L_{\lambda_0}w_\ell=0
    \quad\text{in }\mathbb B^n\setminus \overline{B_{\mathbb H}(R)}.
\]
Moreover, by \eqref{eq:hard-farfield-error-section-six},
\[
    \|(w_\ell)_\infty\|_{L^2(\mathbb S^{n-1})}
    =
    \bigl\|
        u_{\infty,\mathcal P,\xi_\ell}
        -
        u_{\infty,\mathcal P',\xi_\ell}
    \bigr\|_{L^2(\mathbb S^{n-1})}
    \leq
    \varepsilon .
\]
Applying Lemma~\ref{lem:quantitative-rellich} on a fixed annulus containing \(B_{\mathbb H}(x_0,\widehat\rho)\) in \eqref{eq:hard-observation-ball-section-six}, we obtain
\begin{equation*}
    \max_{\ell=1,\ldots,m}
    \|u_\ell-u'_\ell\|_{L^\infty(B_{\mathbb H}(x_0,\widehat\rho))}
    \leq
    \varepsilon_1,
    \qquad
    \varepsilon_1:=\Psi(\varepsilon).
\end{equation*}
Here, \(\Psi\) is the far-field-to-near-field modulus defined in
\eqref{eq:Psi-ff-definition-section-six}, with the dependence specified there.

We also fix the uniform upper bound used in all propagation and reflection steps.
All balls and patches appearing below are contained in compact hyperbolic regions determined only by the a priori geometry.
Hence Proposition~\ref{prop:uniform-local-bound} gives a constant \(E_*\geq1\), depending only on these fixed regions, on \(n\), on \(\lambda_0\), on the \(\mathsf N\)-type a priori data, and on the fixed \(h\), such that for every ball \(B\) used below,
\begin{equation}
\label{eq:hard-uniform-upper-bound-used}
    \max_{\ell=1,\ldots,m}
    \left(
        \|u_\ell\|_{L^\infty(B)}
        +
        \|u'_\ell\|_{L^\infty(B)}
    \right)
    \leq
    E_* .
\end{equation}
This is the a priori \(L^\infty\)-bound used below when applying Lemma~\ref{lem:propagation-chain} and Lemma~\ref{lem:quant-reflection}.
{
Fix \(\beta\in(0,1)\) no larger than the exponents in
Lemmas~\ref{lem:propagation-chain} and~\ref{lem:quant-reflection}.
}
After decreasing \(\varepsilon_0\), we may assume
\[
    0<\varepsilon_1=\Psi(\varepsilon)\leq E_*,
    \qquad
    \varepsilon_1<e^{-1}.
\]

\smallskip
\noindent
\emph{Step 2. Propagation to the first boundary discrepancy.}

Using the modified boundary distance $d$ in Definition~\ref{def:hyperbolic-distances}, after possibly interchanging \(\mathcal P\) and \(\mathcal P'\), we may choose \(x_1\in\partial\mathcal P'\setminus\mathcal P \) so that
\begin{equation*}
    \rho(x_1,\partial\mathcal P)\geq d/2 .
\end{equation*}
This possible interchange does not change \(d\), \(\bar h\), or \(\varepsilon\).
Using \(d/2\) only changes constants.

Since \(x_1\notin\mathcal P\) and
\(\rho(x_1,\partial\mathcal P)\geq d/2\), we have $B_{\mathbb H}(x_1,d/2)\subset G$.
{
We first construct in \(G\) a regular chain ending at \(x_1\), with
nonincreasing radii and terminal radius comparable to \(d\).  We then truncate
it immediately before the first enlarged ball meeting \(\mathcal P'\); all
retained enlarged balls lie in \(H\), while radius monotonicity keeps the last
retained radius bounded below by \(cd\) and the overlap conditions place the
last ball at controlled distance from \(\partial\mathcal P'\).
}
Thus, after relabeling, we obtain
\[
    B_i=B_{\mathbb H}(z_i,r_i),
    \qquad i=0,\ldots,N_1,
\]
with fixed regular-chain parameters such that
\[
    B_0^a\subset B_{\mathbb H}(x_0,\widehat\rho),
    \qquad
    B_i^A\subset H,\quad i=0,\ldots,N_1.
\]
The number \(N_1\) of iterations along this first regular chain satisfies
\begin{equation}\label{eq:first-chain-length-bound-hard}
    N_1\leq C_{\rm geo}\log\frac{2eR}{d}.
\end{equation}

Since \(w_\ell\) solves \(\mathcal L_{\lambda_0}w_\ell=0\) in \(H\), and since
\(B_i^A\subset H\) for \(i=0,\ldots,N_1\), the regular chain constructed above is
admissible for Lemma~\ref{lem:propagation-chain}.
On the initial ball \(B_0^a\), Step~1 gives
\[
    \|w_\ell\|_{L^\infty(B_0^a)}
    \leq
    \varepsilon_1
    =
    \Psi(\varepsilon).
\]
On the enlarged balls \(B_i^A\), the uniform bound
\eqref{eq:hard-uniform-upper-bound-used} gives, after enlarging \(E_*\) by a fixed
factor,
\[
    \|w_\ell\|_{L^\infty(B_i^A)}
    \leq
    E_*,
    \qquad i=0,\ldots,N_1 .
\]
Applying Lemma~\ref{lem:propagation-chain}, namely the propagation-of-smallness
estimate along regular hyperbolic chains, we obtain
\begin{equation*}
    \max_{\ell=1,\ldots,m}
    \|u_\ell-u'_\ell\|_{L^\infty(B_{N_1}^a)}
    \leq
    C E_*
    \left(
        \frac{\Psi(\varepsilon)}{E_*}
    \right)^{\beta^{N_1}} .
\end{equation*}
Absorbing the fixed \(E_*\) into \(C\), we shall use this estimate in the form
\begin{equation}\label{eq:hard-smallness-first-discrepancy-simple}
    \max_{\ell=1,\ldots,m}
    \|u_\ell-u'_\ell\|_{L^\infty(B_{N_1}^a)}
    \leq
    C^{N_1+1}
    \bigl(\Psi(\varepsilon)\bigr)^{\beta^{N_1}},
    \qquad
    N_1\leq C_{\rm geo}\log\frac{2eR}{d}.
\end{equation}

\smallskip
\noindent
\emph{Step 3. Propagation from the discrepancy point to a regular face patch.}

We next pass from the neighborhood obtained in Step~2 to a regular face patch of \(\partial\mathcal P'\).
Let \(z_*\in\partial\mathcal P'\) be a boundary point detected by the last ball of the first chain.
Choose a totally geodesic face \(\mathcal F'\subset\partial\mathcal P'\) containing \(z_*\).
If \(z_*\) is close to the relative boundary of \(\mathcal F'\), the controlled incidence condition and the relative \(h\)-regularity allow us to move within the same face, on the side accessible from the common exterior, to a point \(y_1\in\mathcal F'\).
The point \(y_1\) can be chosen so that two quantitative properties hold.
First, it stays away from the other defect at the \(d\)-scale:
\[
    \rho(y_1,\mathcal P)\geq c_1 d .
\]
Second, it is an interior point of the face at the \(\bar h\)-scale.
More precisely, if \(V_{\mathcal F'}\) denotes the totally geodesic hypersurface containing \(\mathcal F'\), then
\begin{equation}
\label{eq:face-patch-hard-section-six}
    B_{\mathbb H}(y_1,c_1\bar h)\cap V_{\mathcal F'}
    \subset
    \mathcal F' .
\end{equation}
Here \(c_1\in(0,1)\) depends only on the \(\mathsf N\)-type a priori geometry.

	These two properties are the geometric output of the face-selection construction.
They are the quantitative analogue of the face-selection step in the
reflected-continuation construction used in the proof of
Theorem~\ref{thm:mainip1}.
Starting from the first boundary discrepancy, one follows the accessible side
of the common exterior until the construction meets a face of
\(\partial\mathcal P'\).
If the contact point lies close to the relative boundary of that face, the
controlled incidence condition and the relative \(h\)-regularity in
Definition~\ref{def:admissible-tg-defects} allow the point to be shifted
along the same face, after decreasing \(c_1\) by a fixed factor, to an interior
point \(y_1\) satisfying \eqref{eq:face-patch-hard-section-six} while
preserving the \(d\)-scale separation from \(\mathcal P\).

The point \(y_1\) is chosen on the side of \(\mathcal F'\) accessible from the common exterior.
Thus there is a ball \(B_{N_1+N_2}\), with radius comparable to \(\bar h\), lying on the exterior side of this patch and satisfying
\[
    B_{N_1+N_2}^A\subset H .
\]
The passage from \(B_{N_1}\) to \(B_{N_1+N_2}\) is obtained by concatenating a bounded number of fixed-scale balls with a conical chain whose radii decrease geometrically down to the scale \(\bar h\).
The number \(N_2\) of iterations along this second regular chain satisfies
\begin{equation}\label{eq:second-chain-length-bound-hard}
    N_2
    \leq
    C_{\rm geo}\left(
        \log\frac{2eR}{d}
        +
        \log\frac{2eR}{\bar h}
    \right).
\end{equation}
{
Indeed, on each conical part the balls may be chosen so that
\(r_k\asymp q^{k-1}r_1\), where \(q\in(0,1)\) depends only on the fixed
cone aperture, while the fixed-scale parts contain only a uniformly bounded
number of balls.  If \(N(s)\) is the first index for which
\(r_{N(s)}\lesssim s\), then either \(N(s)=1\), or
\(r_{N(s)-1}\gtrsim s\), and hence in both cases
\(N(s)\leq C_{\rm geo}[1+\log(r_1/s)]\); applying this with
\(r_1\leq C_{\rm geo}R\) and
\(s=d,\bar h\), and noting that the first-contact truncation can only reduce
the number of balls, gives
\eqref{eq:first-chain-length-bound-hard} and
\eqref{eq:second-chain-length-bound-hard} after changing \(C_{\rm geo}\).
}
Applying Lemma~\ref{lem:propagation-chain} again, with \eqref{eq:hard-smallness-first-discrepancy-simple} as the initial input and with
the same a priori bound \(E_*\), we obtain
\begin{equation}
\label{eq:hard-smallness-near-face}
    \max_{\ell=1,\ldots,m}
    \|u_\ell-u'_\ell\|_{L^\infty(B_{N_1+N_2}^a)}
    \leq
    C^{N_1+N_2}
    \bigl(\Psi(\varepsilon)\bigr)^{\beta^{N_1+N_2}} .
\end{equation}

\smallskip
\noindent
\emph{Step 4. Neumann trace on the face patch and one reflection.}

On the patch \eqref{eq:face-patch-hard-section-six}, the exterior solutions \(u'_\ell\) satisfy
\[
    \partial_{\nu_{\mathbb H}}\Big|_{V_{\mathcal F'}}u'_\ell=0,
    \qquad
    \ell=1,\ldots,m .
\]
We first pass from the smallness of \(u_\ell-u'_\ell\) on the exterior side to the smallness of the Neumann trace of \(u_\ell\) on the face patch.

Set
\[
    r_h:=c\bar h ,
\]
where \(c>0\) will be chosen sufficiently small depending only on the \(\mathsf N\)-type a priori geometry.
By the geometric output of Step~3, we have
\[
    \rho(y_1,\mathcal P)\geq c_1d,
    \qquad
    B_{\mathbb H}(y_1,c_1\bar h)\cap V_{\mathcal F'}
    \subset
    \mathcal F' .
\]
Since \(\bar h\leq d\), we choose \(c>0\) so small that
\[
    4c\leq c_1 .
\]
Then
\[
    B_{\mathbb H}(y_1,4r_h)\subset G
\]
and
\[
    V_{\mathcal F'}\cap B_{\mathbb H}(y_1,4r_h)\subset \mathcal F' .
\]
Let \(\Omega_h^+\) be the component of
\[
    B_{\mathbb H}(y_1,4r_h)\setminus V_{\mathcal F'}
\]
which lies on the side accessible from the common exterior.
Then \(\Omega_h^+\subset H\).
We also set
\[
    \Omega_h^-:=I_{V_{\mathcal F'}}(\Omega_h^+).
\]
After decreasing \(c\) once more, the last ball constructed in Step~3 satisfies
\[
    B_{N_1+N_2}^a\subset \Omega_h^+ .
\]

{
Since \(u'_\ell\) satisfies the homogeneous hyperbolic Neumann condition
on the patch
\[
    V_{\mathcal F'}\cap B_{\mathbb H}(y_1,4r_h),
\]
the one-sided weak reflection principle in
Lemma~\ref{lem:one-sided-reflection-extension} gives the even reflected extension
}
\[
    \widetilde u'_\ell(x)
    :=
    \begin{cases}
        u'_\ell(x),
        & x\in \Omega_h^+,\\[2mm]
        u'_\ell\bigl(I_{V_{\mathcal F'}}x\bigr),
        & x\in \Omega_h^- .
    \end{cases}
\]
Then
\[
    \mathcal L_{\lambda_0}\widetilde u'_\ell=0
    \quad
    \text{in }B_{\mathbb H}(y_1,4r_h)
\]
in the weak sense.
Since \(B_{\mathbb H}(y_1,4r_h)\subset G\), the exterior solution \(u_\ell\) also satisfies
\[
    \mathcal L_{\lambda_0}u_\ell=0
    \quad
    \text{in }B_{\mathbb H}(y_1,4r_h).
\]

Set
\[
    \widetilde w_\ell:=u_\ell-\widetilde u'_\ell .
\]
Then \(\widetilde w_\ell\) satisfies
\[
    \mathcal L_{\lambda_0}\widetilde w_\ell=0
    \quad
    \text{in }B_{\mathbb H}(y_1,4r_h).
\]
Moreover, on \(\Omega_h^+\), one has \(\widetilde u'_\ell=u'_\ell\), and hence
\(
    \widetilde w_\ell=w_\ell .
\)
Therefore, by \eqref{eq:hard-smallness-near-face} and
\(B_{N_1+N_2}^a\subset\Omega_h^+\), one has:
\[
    \max_{\ell=1,\ldots,m}
    \|\widetilde w_\ell\|_{L^\infty(B_{N_1+N_2}^a)}
    \leq
    C^{N_1+N_2}
    \bigl(\Psi(\varepsilon)\bigr)^{\beta^{N_1+N_2}} .
\]

Since all radii in this local configuration are comparable to \(r_h\), a finite family of regular hyperbolic chains of uniformly bounded length propagates this smallness from \(B_{N_1+N_2}^a\) to \(B_{\mathbb H}(y_1,2r_h)\) inside \(B_{\mathbb H}(y_1,4r_h)\).
Applying Lemma~\ref{lem:propagation-chain} again, with the a priori bound \(2E_*\), still denoted by \(E_*\), we obtain
\[
    \max_{\ell=1,\ldots,m}
    \|\widetilde w_\ell\|_{L^\infty(B_{\mathbb H}(y_1,2r_h))}
    \leq
    C^{N_1+N_2}
    \bigl(\Psi(\varepsilon)\bigr)^{\beta^{N_1+N_2}} .
\]
The interior \(C^1\)-estimate for \(\mathcal L_{\lambda_0}\) in \(B_{\mathbb H}(y_1,2r_h)\) gives
\[
    r_h
    \max_{\ell=1,\ldots,m}
    \|\nabla_{\mathbb H}\widetilde w_\ell\|_{L^\infty(B_{\mathbb H}(y_1,r_h))}
    \leq
    C^{N_1+N_2}
    \bigl(\Psi(\varepsilon)\bigr)^{\beta^{N_1+N_2}} .
\]
Since \(r_h=c\bar h\), the fixed factor \(c\) is absorbed into the constant.

The reflected extension \(\widetilde u'_\ell\) is even with respect to \(V_{\mathcal F'}\).
Therefore
\[
    \partial_{\nu_{\mathbb H}}\Big|_{V_{\mathcal F'}}
    \widetilde u'_\ell=0
    \quad
    \text{on }V_{\mathcal F'}\cap B_{\mathbb H}(y_1,r_h).
\]
Consequently,
\[
    \partial_{\nu_{\mathbb H}}\Big|_{V_{\mathcal F'}}\widetilde w_\ell
    =
    \partial_{\nu_{\mathbb H}}\Big|_{V_{\mathcal F'}}u_\ell
    \quad
    \text{on }V_{\mathcal F'}\cap B_{\mathbb H}(y_1,r_h).
\]
Hence, on a slightly smaller patch,
\begin{equation}\label{eq:hard-small-neumann-defect}
    \bar h
    \max_{\ell=1,\ldots,m}
    \sup_{x\in V_{\mathcal F'}\cap B_{\mathbb H}(y_1,r_h)}
    \left|
        \partial_{\nu_{\mathbb H}}\Big|_{V_{\mathcal F'}}u_\ell(x)
    \right|
    \leq
    C^{N_1+N_2}
    \bigl(\Psi(\varepsilon)\bigr)^{\beta^{N_1+N_2}} .
\end{equation}
The factor \(\bar h\) comes from the scale \(r_h=c\bar h\) in the derivative estimate.

We now apply the Neumann part of Lemma~\ref{lem:quant-reflection} to \(u_\ell\) with
\[
    V=V_{\mathcal F'},
    \qquad
    p=y_1,
    \qquad
    r=r_h .
\]
The equation holds in \(B_{\mathbb H}(y_1,r_h)\), the \(L^\infty\)-bound is \(E_*\), and the small Neumann datum is given by \eqref{eq:hard-small-neumann-defect}.
After decreasing \(c\) and enlarging \(C\), we obtain
\begin{equation}\label{eq:hard-one-reflection-defect}
    \max_{\ell=1,\ldots,m}
    \|u_\ell-u_\ell\circ I_{V_{\mathcal F'}}\|_{L^\infty
    (B_{\mathbb H}(y_1,c\bar h))}
    \leq
    C^{N_1+N_2}
    \bigl(\Psi(\varepsilon)\bigr)^{\beta^{N_1+N_2+1}} .
\end{equation}
Here \(I_{V_{\mathcal F'}}\) is the hyperbolic reflection with respect to \(V_{\mathcal F'}\).

\smallskip
\noindent
\emph{Step 5. Iteration of reflected continuation.}

We now iterate the propagation-reflection step obtained above.
This is the quantitative version of the reflected-continuation construction used in the uniqueness proof.
At each propagation step we use Lemma~\ref{lem:propagation-chain}.
At each reflection step we use the Neumann part of Lemma~\ref{lem:quant-reflection}.
Reflections across totally geodesic hypersurfaces are hyperbolic isometries, and hence preserve \(\mathcal L_{\lambda_0}\).
The uniform \(L^\infty\)-bound entering all applications is the fixed constant \(E_*\).

{
Fix a reference chain
\(\mathscr B_0,\ldots,\mathscr B_{N_{\rm ref}}\) by adjoining, at the observation end
of the Step~2 chain, a fixed-scale exterior chain to the far annulus chosen
below, and orient it from that annulus toward the first discrepancy.  This
exterior extension has uniformly bounded
length, so
\[
    N_{\rm ref}\leq N_1+C_{\rm geo}\leq C_{\rm geo}\log\frac{2eR}{d}.
\]
Put \(q_0=N_{\rm ref}\).  At stage \(j\), with \(Q_j\) and \(Q_{j+1}\)
defined below, join the local face ball to
\(\mathscr B_{q_{j-1}}\) and follow this same chain toward decreasing indices.
The attachment is chosen so that
\(\mathscr B_{q_{j-1}}^A\subset Q_j\); hence the first enlarged ball meeting
\(I_{V_j}(\mathcal P)\), if it exists, has index strictly smaller than
\(q_{j-1}\).
If the relevant component of \(Q_j\) has not reached the far annulus, let
\(q_j<q_{j-1}\) be the index of the first enlarged ball meeting
\(I_{V_j}(\mathcal P)\).  The intervening balls lie in \(Q_j\); the exterior
cone and face-selection construction joins the last ball before contact to a
regular patch of this first contacted face and, after reflection, joins the new
local face ball to \(\mathscr B_{q_j}\) inside \(Q_{j+1}\).
The new local ball is invariant under \(I_{V_{j+1}}\) and is contained in
\(G\), so it is indeed contained in \(Q_{j+1}\).
}

We first describe one induction step.
Assume that, at some stage \(j\), we have obtained a totally geodesic hypersurface \(V_j\), a point \(y_j\in V_j\), and a scale \(r_j\geq c\bar h\).
{
Set
\[
    \mathcal P_j:=I_{V_j}(\mathcal P),
    \qquad
    Q_j
    :=
    \{x\in G:\ I_{V_j}x\in G\}
    =
    \mathbb B^n\setminus(\mathcal P\cup\mathcal P_j).
\]
}
We consider the reflection residual in the connected component of \(Q_j\) used in the reflected-continuation construction:
\[
    D_{j,\ell}
    :=
    u_\ell-u_\ell\circ I_{V_j},
    \qquad
    \ell=1,\ldots,m .
\]
Assume that
\begin{equation}\label{eq:hard-induction-defect}
    \max_{\ell=1,\ldots,m}
    \|D_{j,\ell}\|_{L^\infty(B_{\mathbb H}(y_j,c r_j))}
    \leq
    C^{K_j}
    \bigl(\Psi(\varepsilon)\bigr)^{\beta^{K_j}} .
\end{equation}
Here, \(K_j\) is the total number of elementary propagation and reflection
steps used up to this stage, with the chain lengths controlled by
\eqref{eq:first-chain-length-bound-hard} and
\eqref{eq:second-chain-length-bound-hard}.
For \(j=1\), this estimate is exactly \eqref{eq:hard-one-reflection-defect}, after setting
\[
    K_1:=N_1+N_2+1
\]
and enlarging \(C\) if necessary.

The function \(D_{j,\ell}\) solves
\[
    \mathcal L_{\lambda_0}D_{j,\ell}=0
\]
in the connected component of \(Q_j\) used in the reflected-continuation construction.
Indeed, \(u_\ell\) solves \(\mathcal L_{\lambda_0}u_\ell=0\) in \(G\), while \(u_\ell\circ I_{V_j}\) solves the same equation at every point \(x\in Q_j\), because \(I_{V_j}x\in G\) and \(I_{V_j}\) is a hyperbolic isometry.
Following the preceding construction, the next reflected face is reached from
\(B_{\mathbb H}(y_j,c r_j)\) by a regular chain of hyperbolic balls contained
in this reflected-continuation region.
The length of this chain is controlled by the admissible geometry and the exterior connectedness modulus.

Applying Lemma~\ref{lem:propagation-chain} to \(D_{j,\ell}\) along this chain gives smallness of \(D_{j,\ell}\) near the next face.
More precisely, if \(V_{j+1}\) is the totally geodesic hypersurface containing the next reflected face and \(y_{j+1}\in V_{j+1}\) is the chosen interior point of that face, then
\begin{equation}
\label{eq:hard-next-defect-small}
    \max_{\ell=1,\ldots,m}
    \|D_{j,\ell}\|_{L^\infty(B_{\mathbb H}(y_{j+1},c r_{j+1}))}
    \leq
    C^{K_{j+1}}
    \bigl(\Psi(\varepsilon)\bigr)^{\beta^{K_{j+1}}} ,
\end{equation}
where \(r_{j+1}\geq c\bar h\), and \(K_{j+1}\) is obtained from \(K_j\) by adding the number of balls in the new chain.

More precisely, the next hypersurface \(V_{j+1}\) is chosen so that
\[
    I_{V_j}
    \left(
        V_{j+1}\cap B_{\mathbb H}(y_{j+1},c r_{j+1})
    \right)
\]
is contained in a \(\mathsf N\)-type totally geodesic face of \(\partial\mathcal P\).
{
Choose the unit hyperbolic normal on the reflected face as the image under
\(dI_{V_j}\) of the unit hyperbolic normal on the original face.  The isometric covariance formula
\eqref{eq:normal-covariance-isometry} then gives, at corresponding points,
\[
    \partial_{\nu_{\mathbb H,V_{j+1}}}(u_\ell\circ I_{V_j})
    =
    \bigl(\partial_{\nu_{\mathbb H,I_{V_j}(V_{j+1})}}u_\ell\bigr)\circ I_{V_j}
    =0.
\]
Thus \(u_\ell\circ I_{V_j}\) satisfies the homogeneous hyperbolic Neumann
condition on the reflected face, with no additional lower-order term.
}
Hence, on a slightly smaller patch of \(V_{j+1}\),
\[
    \partial_{\nu_{\mathbb H}}\Big|_{V_{j+1}}u_\ell
    =
    \partial_{\nu_{\mathbb H}}\Big|_{V_{j+1}}D_{j,\ell}.
\]
Using the local interior \(C^1\)-estimate for \(D_{j,\ell}\), together with \eqref{eq:hard-next-defect-small}, we obtain
\begin{equation}\label{eq:hard-next-neumann-small}
    \bar h
    \max_{\ell=1,\ldots,m}
    \sup_{x\in V_{j+1}\cap B_{\mathbb H}(y_{j+1},c r_{j+1}/2)}
    \left|
        \partial_{\nu_{\mathbb H}}\Big|_{V_{j+1}}u_\ell(x)
    \right|
    \leq
    C^{K_{j+1}}
    \bigl(\Psi(\varepsilon)\bigr)^{\beta^{K_{j+1}}} .
\end{equation}
Here we used \(r_{j+1}\geq c\bar h\), and absorbed the fixed constant \(c\) into \(C\).
The same reflected-continuation geometry also ensures, after decreasing \(c>0\) if necessary, that
\[
    B_{\mathbb H}(y_{j+1},c r_{j+1})\subset G .
\]
Hence \(u_\ell\) solves \(\mathcal L_{\lambda_0}u_\ell=0\) in the ball where Lemma~\ref{lem:quant-reflection} is applied.

We now apply the Neumann part of Lemma~\ref{lem:quant-reflection} to \(u_\ell\) across \(V_{j+1}\).
The \(L^\infty\)-bound is \(E_*\), and the small Neumann datum is given by \eqref{eq:hard-next-neumann-small}.
After decreasing \(c\) and enlarging \(C\), this gives
\begin{equation*}
    \max_{\ell=1,\ldots,m}
    \|u_\ell-u_\ell\circ I_{V_{j+1}}\|_{L^\infty(B_{\mathbb H}(y_{j+1},c r_{j+1}))}
    \leq
    C^{K_{j+1}+1}
    \bigl(\Psi(\varepsilon)\bigr)^{\beta^{K_{j+1}+1}} .
\end{equation*}
This is the induction hypothesis \eqref{eq:hard-induction-defect} at the next stage, after replacing \(K_{j+1}\) by \(K_{j+1}+1\).

{
At every nonterminal stage \(q_j<q_{j-1}\).  Hence the construction stops
after a number \(J\) of reflection stages satisfying
\[
    J\leq N_{\rm ref}+1\leq C_{\rm geo}\log\frac{2eR}{d}.
\]
At each stage the passage from the first hit to a regular face patch uses, as
in Steps~2--3, at most
\[
    L_*:=C_{\rm geo}\left(
        \log\frac{2eR}{d}
        +
        \log\frac{2eR}{\bar h}
    \right)
\]
elementary propagation and reflection steps.
Consequently,
\[
    K_J\leq C_{\rm geo}(N_{\rm ref}+JL_*).
\]
}
Let \(V_J\) be the terminal totally geodesic hypersurface.
{
Every selected face patch lies at controlled distance from the fixed reference
chain; hence \(V_J\) intersects a fixed compact hyperbolic ball depending only
on the a priori data.
}

At the terminal stage we have, after renaming \(C\),
\begin{equation*}
    \max_{\ell=1,\ldots,m}
    \|u_\ell-u_\ell\circ I_{V_J}\|_{L^\infty(B_{\mathbb H}(y_J,c r_J))}
    \leq
    C^{K_J}
    \bigl(\Psi(\varepsilon)\bigr)^{\beta^{K_J}},
\end{equation*}
where \(y_J\in V_J\) lies in a fixed compact hyperbolic ball and \(r_J\geq c\bar h\).

We now propagate this terminal reflection residual to the fixed far annulus
\[
    B_{\mathbb H}(R_*+T_*)\setminus B_{\mathbb H}(R_*).
\]
{
Choose \(R_{\rm geo}\) so that the fixed compact ball met by every hypersurface
produced above is contained in \(B_{\mathbb H}(R_{\rm geo})\).  If
\(z\in V\cap \overline{B_{\mathbb H}(R_{\rm geo})}\) and
\(p\in\mathcal P\subset \overline{B_{\mathbb H}(R)}\), then
\[
    \rho(0,I_Vp)
    \leq \rho(0,z)+\rho(z,I_Vp)
    =\rho(0,z)+\rho(z,p)
    \leq 2R_{\rm geo}+R.
\]
Thus \(R_*>2R_{\rm geo}+R+1\) may be fixed so that, uniformly for all such
\(V\),
\[
    \mathcal P\cup I_V(\mathcal P)
    \subset B_{\mathbb H}(R_*-1),
    \qquad
    \mathbb B^n\setminus B_{\mathbb H}(R_*-1)
    \subset Q_V.
\]
Here \(Q_V:=\{x\in G:\ I_Vx\in G\}\).
}
Hence the terminal reflection residual
\[
    u_\ell-u_\ell\circ I_{V_J}
\]
is well-defined in the far region used below.
{
By the stopping rule, the component of
\[
    Q_J:=\{x\in G:\ I_{V_J}x\in G\}
\]
containing \(B_{\mathbb H}(y_J,c r_J)\) reaches the fixed far annulus: either it
does so before another contact, or the remaining reference balls provide
the connection when no further \(q_j\) is defined.  A regular
chain in that component has length bounded by the same single-stage estimate
displayed above.  If its length is \(L_J\), set
\(M:=K_J+L_J\).  The stage bound therefore gives
\[
    M
    \leq
    C_{\rm geo}
    \log\frac{2eR}{d}
    \left(
        \log\frac{2eR}{d}
        +
        \log\frac{2eR}{\bar h}
    \right).
\]
The exterior inclusion above shows that the exterior end of this chain and
the tubular neighborhood of \(V_J\) in the far annulus, chosen with radius
less than \(1\), lie in the same component of \(Q_J\).
}
Propagating along this last chain gives
\[
    \max_{\ell=1,\ldots,m}
    \|u_\ell-u_\ell\circ I_{V_J}\|_{L^\infty(\mathcal N_J)}
    \leq
    C^M
    \bigl(\Psi(\varepsilon)\bigr)^{\beta^M},
\]
where \(\mathcal N_J\) is a fixed-size tubular neighborhood of
\[
    V_J\cap
    \left(B_{\mathbb H}(R_*+T_*)\setminus B_{\mathbb H}(R_*)\right).
\]
The required uniform \(L^\infty\)-bounds in this far region follow from Proposition~\ref{prop:uniform-decay} for the outgoing corrections and from the explicit form of the incoming Helgason modes.

Applying the interior \(C^1\)-estimate in \(\mathcal N_J\), we obtain
\[
    \max_{\ell=1,\ldots,m}
    \sup_{x\in V_J\cap
    \left(B_{\mathbb H}(R_*+T_*)\setminus B_{\mathbb H}(R_*)\right)}
    \left|
        \partial_{\nu_{\mathbb H}}\Big|_{V_J}
        \bigl(u_\ell-u_\ell\circ I_{V_J}\bigr)(x)
    \right|
    \leq
    C^M
    \bigl(\Psi(\varepsilon)\bigr)^{\beta^M}.
\]
On \(V_J\), the reflection \(I_{V_J}\) fixes points and reverses the hyperbolic normal direction.
Therefore
\[
    \partial_{\nu_{\mathbb H}}\Big|_{V_J}
    \bigl(u_\ell-u_\ell\circ I_{V_J}\bigr)
    =
    2\partial_{\nu_{\mathbb H}}\Big|_{V_J}u_\ell .
\]
Since \(\rho(x)\in[R_*,R_*+T_*]\) on the fixed annulus, the weight
\[
    e^{\frac{n+1}{2}\rho(x)}
\]
is bounded above by a constant depending only on \(R_*\), \(T_*\), and \(n\).
Multiplying also by \(\bar h\leq 2R\), and absorbing these fixed factors into \(C^M\), we obtain
\begin{equation*}
    \bar h
    \max_{\ell=1,\ldots,m}
    \sup_{x\in V_J\cap
    \left(B_{\mathbb H}(R_*+T_*)\setminus B_{\mathbb H}(R_*)\right)}
    e^{\frac{n+1}{2}\rho(x)}
    \left|
        \partial_{\nu_{\mathbb H}}\Big|_{V_J}u_\ell(x)
    \right|
    \leq
    C^M
    \bigl(\Psi(\varepsilon)\bigr)^{\beta^M}.
\end{equation*}
This is \eqref{eq:hard-reflected-continuation-upper}.
{
These constants are independent of \(\mathcal P,\mathcal P'\) and
\(\varepsilon\), with the dependence on the fixed \(h\) as stated above.
}

The proof is complete.
\end{proof}

The following elementary lemma converts the reflected-continuation estimate into the logarithmic stability modulus.  It will be used in both the \(\mathsf N\)-type and \(\mathsf D\)-type parts.
We shall use the logarithmic modulus
\begin{equation}\label{eq:eta-modulus-section-six}
    \eta(t):=
    \exp\!\left(-\sqrt{\log(-\log t)}\right),
    \qquad 0<t<e^{-1}.
\end{equation}
After decreasing the admissible error thresholds, every argument of \(\eta\) appearing below lies in \((0,e^{-1})\).

\begin{lemma}\label{lem:logarithmic-reduction}
Let \(0<\bar h\leq d\leq 2R\), \(A_0:=2eR\), and \(0<s<e^{-1}\).  Assume that
there are constants \(E\geq1\), \(C_0\geq1\), \(C_1\geq1\),
\(\beta_0\in(0,1)\), \(c_0>0\), and \(\sigma\in\{0,1\}\) such that
\begin{equation}
\label{eq:log-lemma-main-ineq}
    c_0\bar h^\sigma
    \leq
    C_0^{\,1+B_M}E^{1-\Gamma_M}s^{\Gamma_M},
\end{equation}
where
\begin{equation}
\label{eq:log-lemma-M-bound}
    M
    \leq
    C_1
    \log\frac{A_0}{d}
    \left(
        \log\frac{A_0}{d}
        +
        \log\frac{A_0}{\bar h}
    \right),
\end{equation}
and
\begin{equation*}
    \Gamma_M\geq\beta_0^M,
    \qquad
    B_M\leq C_1(1+M),
\end{equation*}
{
Then there exist \(C>0\) and \(\alpha>0\), depending only on
\(A_0,C_1,\beta_0\).  There also exists \(s_0>0\), which may depend on
\(C_0,E,c_0\), such that, for every \(0<s\leq s_0\),
}
\begin{equation}
\label{eq:log-lemma-conclusion}
    \bar h\leq C\bigl(\eta(s)\bigr)^\alpha,
\end{equation}
where $\eta(s)$ is defined in \eqref{eq:eta-modulus-section-six}.
\end{lemma}

\begin{proof}
Set
\begin{equation*}
    Y:=\log\frac{A_0}{\bar h}.
\end{equation*}
Since \(0<\bar h\leq 2R\), we have \(Y\geq1\).  Moreover, since
\(\bar h\leq d\), one has
\[
    \log\frac{A_0}{d}\leq Y.
\]
Thus \eqref{eq:log-lemma-M-bound} gives
\begin{equation*}
    M\leq C_2Y^2 .
\end{equation*}
Consequently,
\begin{equation}
\label{eq:gamma-lower-Y}
    \Gamma_M
    \geq
    \beta_0^M
    \geq
    \exp(-C_3Y^2),
\end{equation}
and, using \(E\geq1\),
\begin{equation}
\label{eq:prefactor-Y}
    C_0^{1+B_M}E^{1-\Gamma_M}
    \leq
    \exp(C_4(1+Y^2)).
\end{equation}

Let
\[
    L:=-\log s .
\]
Taking logarithms in \eqref{eq:log-lemma-main-ineq}, and using
\eqref{eq:prefactor-Y} together with
\(\bar h=A_0e^{-Y}\), gives
\begin{equation}
\label{eq:basic-Y-L-ineq}
    \Gamma_M L
    \leq
    C_5(1+Y^2)+C_5Y .
\end{equation}
When \(\sigma=0\), the linear term in \(Y\) is absent; keeping it only enlarges the
right-hand side.

We now show that \(Y\) cannot be too small.  Choose \(c>0\) so small that
\(C_3c^2<1/2\).  If, for arbitrarily large \(L\),
\[
    Y<c\sqrt{\log L},
\]
then \eqref{eq:gamma-lower-Y} gives
\[
    \Gamma_M L
    \geq
    L\exp(-C_3c^2\log L)
    \geq
    L^{1/2}.
\]
On the other hand, the right-hand side of \eqref{eq:basic-Y-L-ineq} is bounded by
\(C\log L\), which is impossible for large \(L\).  Hence, for all sufficiently small
\(s\),
\begin{equation*}  
  Y\geq c\sqrt{\log L}
    =
    c\sqrt{\log(-\log s)} .
\end{equation*}
Therefore
\[
    \bar h
    =
    A_0e^{-Y}
    \leq
    A_0
    \exp\!\left(-c\sqrt{\log(-\log s)}\right)
    =
    A_0(\eta(s))^c .
\]
Renaming the constants gives \eqref{eq:log-lemma-conclusion}.
{
Since the above \(c\) depends only on \(C_3\), the constants \(C_0,E,c_0\)
enter only through the smallness threshold for \(s\).
}
\end{proof}

\medskip

The reflected-continuation estimate above gives an upper bound for the normalized
Neumann trace of the \(\mathsf N\)-type exterior solutions on a terminal totally geodesic hypersurface
\(V_J\).  In order to turn this analytic smallness into geometric stability, we need
a lower bound for the same quantity.  The point is that the terminal hypersurface
produced by the reflected-continuation construction is not arbitrary: as in the
uniqueness proof, it intersects a fixed compact hyperbolic ball.  We denote by
\(R_{\rm geo}>R\) an a priori radius, depending only on the \(\mathsf N\)-type admissible class,
such that every terminal hypersurface arising in the construction satisfies
\[
    V_J\cap \overline{B_{\mathbb H}(R_{\rm geo})}\neq\emptyset .
\]
This radius is only a bookkeeping parameter for the compactness of the family of
possible terminal hypersurfaces; it is fixed once and for all and is independent of
\(\mathcal P,\mathcal P'\), \(h\), and the far-field error.

The next lemma gives the required nondegeneracy estimate on this compact family
of hypersurfaces.  It is the quantitative version of the argument used in the
boundedness of Neumann totally geodesic hypersurfaces in the uniqueness proof:
the admissibility of \(\Xi\) ensures that, for every such \(V\), at least one boundary
label has a non-vanishing normal component at conformal infinity.  Consequently, the
incoming principal part of the corresponding incoming Helgason mode cannot be canceled by
the outgoing correction on a fixed far annulus.

\begin{lemma}\label{lem:hard-lower-bound}
Let \(0<h\leq h_0\), let \(\mathcal P\in\mathcal B_{\mathbb H}^h\), and let
\(\Xi=(\xi_1,\ldots,\xi_m)\) be admissible in the sense of
Definition~\ref{def:hard-directions}.
Set
\[
    U=(u_{\mathcal P,\xi_1},\ldots,u_{\mathcal P,\xi_m}).
\]
{
There exist \(R_*>\max\{R,R_{\rm geo}\}\) and \(T_*>0\), depending only on
the \(\mathsf N\)-type a priori data, on \(n\), on \(\lambda_0\), on
\(R_{\rm geo}\), and on \(a_0(\Xi)\), and a constant \(c_{\mathsf N}>0\), which may
also depend on the fixed \(h\), such that for every totally geodesic
hypersurface \(V\) satisfying
}
\[
    V\cap \overline{B_{\mathbb H}(R_{\rm geo})}\neq\emptyset,
\]
one has
\begin{equation}
\label{eq:hard-lower-bound}
    \max_{\ell=1,\ldots,m}
    \sup_{x\in V\cap
    \left(B_{\mathbb H}(R_*+T_*)\setminus B_{\mathbb H}(R_*)\right)}
    e^{\frac{n+1}{2}\rho(x)}
    \left|
        \partial_{\nu_{\mathbb H}}\Big|_{V}
        u_{\mathcal P,\xi_\ell}(x)
    \right|
    \geq c_{\mathsf N} .
\end{equation}
Here \(\partial_{\nu_{\mathbb H}}\big|_V\) denotes the hyperbolic normal derivative
along \(V\), with either choice of a unit hyperbolic normal to \(V\).
The left-hand side is independent of this choice.
\end{lemma}

\begin{proof}
{
Fix \(R_*>\max\{R,R_{\rm geo}\}\) and \(T_*>0\) once and for all, depending
only on the stated a priori data.
}
Fix a totally geodesic hypersurface \(V\) with
\[
    V\cap \overline{B_{\mathbb H}(R_{\rm geo})}\neq\emptyset .
\]
By admissibility of \(\Xi\) given in Definition~\ref{def:hard-directions}, there is an index \(\ell=\ell(V)\) such that
\begin{equation}
\label{eq:chosen-hard-direction-separated}
    \operatorname{dist}_{\mathbb S^{n-1}}(\xi_\ell,\partial_\infty V)
    \geq a_0(\Xi).
\end{equation}
This separation is the quantitative non-tangency condition at conformal infinity.
We use here the same asymptotic calculation as in the proof of
Lemma~\ref{lem:neumann-set-bounded}, but not the boundedness conclusion of that
lemma.
The point is that \eqref{eq:chosen-hard-direction-separated} gives a uniform
nonzero principal coefficient for the normal derivative of the incoming Helgason mode
\(u^i_{\xi_\ell}=e_{2\lambda_0,\xi_\ell}\) along \(V\) at conformal infinity.

We first record the calculation in the standard case
\[
    V_0=\{x_n=0\}.
\]
Let
\[
    x_\rho=\tanh(\rho/2)\eta\in V_0,
    \qquad
    \eta\in\partial_\infty V_0=\{\eta_n=0\}.
\]
On \(V_0\), the hyperbolic normal derivative is
\[
    \partial_{\nu_{\mathbb H}}\Big|_{V_0}
    =
    \frac{1-|x|^2}{2}\,
    \partial_{x_n}\Big|_{V_0}.
\]
A direct computation from the formula for \(e_{2\lambda_0,\xi}\) gives
\[
    \partial_{\nu_{\mathbb H}}\Big|_{V_0}
    e_{2\lambda_0,\xi}(x_\rho)
    =
    A_{\xi,\eta}(\rho)
    \left(\cosh\frac{\rho}{2}\right)^{-(n+1)-2\mathrm i\lambda_0},
\]
where
\[
    A_{\xi,\eta}(\rho)
    =
    \frac{n-1+2\mathrm i\lambda_0}{2}\,
    \xi_n
    \left|
        \tanh\frac{\rho}{2}\eta-\xi
    \right|^{-(n+1)-2\mathrm i\lambda_0}.
\]
Hence
\[
    e^{\frac{n+1}{2}\rho}
    \partial_{\nu_{\mathbb H}}\Big|_{V_0}
    e_{2\lambda_0,\xi}(x_\rho)
    =
    A(\eta,\xi)e^{-\mathrm i\lambda_0\rho}+O(e^{-\rho}),
\]
where \(A(\eta,\xi)\neq0\) precisely when
\[
    \xi\notin\partial_\infty V_0.
\]
Moreover, if
\[
    \operatorname{dist}_{\mathbb S^{n-1}}(\xi,\partial_\infty V_0)
    \geq a_0(\Xi),
\]
then
\[
    |A(\eta,\xi)|\geq a_1>0,
\]
where \(a_1\) depends only on \(n\), \(\lambda_0\), and \(a_0(\Xi)\) in the standard
model.

We now pass to a general \(V\).
{
The passage is legitimate for the hyperbolic normal derivative because
\eqref{eq:normal-covariance-isometry} transports it exactly under every
hyperbolic isometry; in particular, no zeroth-order term is created when
\(V\) is moved to \(V_0\).
}
The family of totally geodesic hypersurfaces satisfying
\[
    V\cap \overline{B_{\mathbb H}(R_{\rm geo})}\neq\emptyset
\]
is compact in the natural parametrization by signed distance from the origin and boundary direction.
For each such \(V\), the admissibility condition \eqref{eq:chosen-hard-direction-separated} gives an index \(\ell=\ell(V)\) such that \(\xi_\ell\) stays a distance at least \(a_0(\Xi)\) away from \(\partial_\infty V\).
Equivalently, the boundary label is uniformly non-tangent to \(V\) at conformal infinity.

The leading coefficient of
\[
    e^{\frac{n+1}{2}\rho}
    \partial_{\nu_{\mathbb H}}\Big|_V e_{2\lambda_0,\xi_\ell}
\]
along any ray in \(V\) depends continuously on \(V\), on the endpoint of the ray in \(\partial_\infty V\), and on \(\xi_\ell\).
By the standard-model calculation above and by the separation from \(\partial_\infty V\), this coefficient is nonzero.
Using compactness, after decreasing \(a_1>0\) if necessary, we obtain a uniform lower bound for a suitable ray in each \(V\).

Thus, for every such \(V\), we may choose a family of points \(x_r\in V\), defined for all sufficiently large \(r\), such that
\[
    \rho(x_r)=r,
    \qquad
    \frac{x_r}{|x_r|}\to\eta_V\in\partial_\infty V
    \quad\text{as }r\to+\infty .
\]
Along this family one has
\begin{equation*}
    e^{\frac{n+1}{2}r}
    \partial_{\nu_{\mathbb H}}\Big|_{V}
    u^i_{\xi_\ell}(x_r)
    =
    A_Ve^{-\mathrm i\lambda_0 r}
    +
    r^i_V(r),
    \qquad
    |A_V|\geq a_1 .
\end{equation*}
The convergence \(r^i_V(r)\to0\) is uniform in \(V\), by the same compactness and by the separation \eqref{eq:chosen-hard-direction-separated}.

We now treat the outgoing correction.
The outgoing Green representation and the \(C^1\) version of the far-field asymptotics in Subsection~\ref{subsec:radition}, together with the uniform exterior bounds from Proposition~\ref{prop:uniform-decay}, give
\begin{equation*}
    e^{\frac{n+1}{2}r}
    \partial_{\nu_{\mathbb H}}\Big|_{V}
    u^s_{\mathcal P,\xi_\ell}(x_r)
    =
    B_{\mathcal P,V}e^{\mathrm i\lambda_0 r}
    +
    r^s_{\mathcal P,V}(r).
\end{equation*}
{
This uniform remainder follows by inserting the fixed-sphere Green
representation, differentiating the uniform kernel expansion in the
\(x\)-variable once, and using the uniform Cauchy-data bound
\eqref{eq:uniform-Cauchy-data-fixed-sphere}.
}
Here, \(B_{\mathcal P,V}\) is uniformly bounded, and \(r^s_{\mathcal P,V}(r)\to0\) uniformly for all admissible \(\mathcal P\), all admissible \(V\), and the selected index \(\ell(V)\).

For the following asymptotic contradiction, choose
\[
    T^\sharp>\frac{\pi}{2\lambda_0}+1.
\]
By the uniform convergence of the remainders, we may then choose \(R_h^\sharp\)
sufficiently large, larger also than the radius in Proposition~\ref{prop:uniform-decay}, such that
\[
    R_h^\sharp>\max\{R,R_{\rm geo}\}
\]
and
\begin{equation}
\label{eq:hard-remainder-small}
    |r^i_V(r)|+|r^s_{\mathcal P,V}(r)|
    \leq
    \frac{a_1}{8},
    \qquad
    r\in[R_h^\sharp,R_h^\sharp+T^\sharp],
\end{equation}
uniformly for all admissible \(\mathcal P\), all admissible \(V\), and the selected
index \(\ell(V)\).

Choose
\[
    \rho_a:=R_h^\sharp+\frac12,
    \qquad
    \rho_b:=\rho_a+\frac{\pi}{2\lambda_0}.
\]
Since
\[
    T^\sharp>\frac{\pi}{2\lambda_0}+1,
\]
both \(\rho_a\) and \(\rho_b\) belong to
\((R_h^\sharp,R_h^\sharp+T^\sharp)\).
Set
\[
    X:=A_Ve^{-\mathrm i\lambda_0\rho_a},
    \qquad
    Y:=B_{\mathcal P,V}e^{\mathrm i\lambda_0\rho_a}.
\]
At the two radii \(\rho_a\) and \(\rho_b\), the principal parts of the normal derivative of the exterior solution are
\[
    X+Y,
    \qquad
    -\mathrm iX+\mathrm iY.
\]
Since
\[
    |X+Y|^2+|-\mathrm iX+\mathrm iY|^2
    =
    2(|X|^2+|Y|^2),
\]
we have
\[
    \max\{|X+Y|,|-\mathrm iX+\mathrm iY|\}
    \geq
    |X|
    \geq
    a_1.
\]
Using \eqref{eq:hard-remainder-small}, we obtain
\[
    \max_{r\in[R_h^\sharp,R_h^\sharp+T^\sharp]}
    e^{\frac{n+1}{2}r}
    \left|
        \partial_{\nu_{\mathbb H}}\Big|_{V}
        u_{\mathcal P,\xi_\ell}(x_r)
    \right|
    \geq
    \frac{a_1}{2}.
\]
{
Consequently, the selected normal derivative cannot vanish identically on the
portion of \(V\) contained in the fixed annulus in
\eqref{eq:hard-lower-bound}: otherwise analytic continuation along \(V\) would
make it vanish on the above ray, contradicting the two-radius estimate.  Thus
the left-hand side is positive for every admissible pair \((\mathcal P,V)\).
Its continuity follows from Proposition~\ref{prop:direct-stability}, local
\(C^1\)-convergence on the fixed annulus, and smooth convergence of \(V\).
Compactness of
\(\mathcal B_{\mathbb H}^h\) and of the family of such \(V\) therefore gives a
positive minimum \(c_{\mathsf N}\).
}

The proof is complete.
\end{proof}
\subsubsection{Proof of Theorem~\ref{thm:main-hard-farfield-simple}}

\begin{proof}
If \(\varepsilon=0\), then the far-field patterns corresponding to all boundary labels in \(\Xi\) coincide.
Since \(\Xi\) is admissible, it contains an affinely independent subfamily of \(n+1\) boundary labels.
The \(\mathsf N\)-type uniqueness theorem, Theorem~\ref{thm:mainip2}, applied to this subfamily yields $\mathcal P=\mathcal P'$.
Hence the desired stability estimate is immediate, and we may assume from now on that \(\varepsilon>0\).

Set
\[
    d:=d(\mathcal P,\mathcal P'),
    \qquad
    \bar h:=\min\{d,h\}.
\]
Since \(h>0\), the case \(\bar h=0\) is equivalent to \(d=0\).
In this case the modified boundary distance already vanishes, and the corresponding Hausdorff estimate follows from the distance comparison in Proposition~\ref{prop:distance-equivalence-hyperbolic}.
It remains to consider the nontrivial case \(\bar h>0\).

{
Choose \(R_*,T_*\) in Lemma~\ref{lem:hard-lower-bound} so that \(R_*\) also
exceeds the a priori radius required in
Proposition~\ref{prop:hard-quant-reflected-continuation}, and use the same
annulus in that proposition.
}
We now apply the quantitative reflected-continuation estimate in Proposition~\ref{prop:hard-quant-reflected-continuation}.
Denote by \(B_{\mathbb H}(x_0,\widehat\rho)\) the fixed observation ball used in that proposition.
The far-field-to-near-field estimate first converts the far-field discrepancy \(\varepsilon\) into the near-field bound
\[
    \max_{\ell=1,\ldots,m}
    \bigl\|
        u_{\mathcal P,\xi_\ell}
        -
        u_{\mathcal P',\xi_\ell}
    \bigr\|_{L^\infty(B_{\mathbb H}(x_0,\widehat\rho))}
    \leq
    \Psi(\varepsilon),
\]
where \(\Psi\) is the far-field-to-near-field modulus defined in \eqref{eq:Psi-ff-definition-section-six}.
Starting from this fixed observation ball, Proposition~\ref{prop:hard-quant-reflected-continuation} propagates the smallness through the common exterior by chains of hyperbolic balls and then applies the quantitative Neumann reflection across successive totally geodesic faces.
As a result, after possibly interchanging \(\mathcal P\) and \(\mathcal P'\), there is a selected defect
\[
    \widetilde{\mathcal P}\in\{\mathcal P,\mathcal P'\}
\]
with corresponding exterior solutions
\[
    \widetilde u_\ell:=u_{\widetilde{\mathcal P},\xi_\ell},
    \qquad \ell=1,\ldots,m,
\]
and there exist integers \(J,M\geq1\) and a terminal totally geodesic hypersurface \(V_J\), intersecting a fixed compact hyperbolic ball, such that
\begin{equation}
\label{eq:hard-upper-from-reflection}
    \bar h
    \max_{\ell=1,\ldots,m}
    \sup_{x\in V_J\cap
    \left(B_{\mathbb H}(R_*+T_*)\setminus B_{\mathbb H}(R_*)\right)}
    e^{\frac{n+1}{2}\rho(x)}
    \left|
        \partial_{\nu_{\mathbb H}}\Big|_{V_J}\widetilde u_\ell(x)
    \right|
    \leq
    C^M
    \bigl(\Psi(\varepsilon)\bigr)^{\beta^M}.
\end{equation}
{
Here \(J\) is the number of reflection stages, while \(M\) is the total number
of elementary propagation and reflection steps; in particular,
\[
    J\leq C_{\rm geo}\log\frac{2eR}{d},
\]
and \(M\) satisfies
\begin{equation}\label{eq:hard-M-bound-main-proof}
    M
    \leq
    C_{\rm geo}
    \log\frac{2eR}{d}
    \left(
        \log\frac{2eR}{d}
        +
        \log\frac{2eR}{\bar h}
    \right).
\end{equation}
}
The possible interchange of \(\mathcal P\) and \(\mathcal P'\) does not change \(d\), \(\bar h\), or the far-field error \(\varepsilon\).

We now compare the upper bound obtained by reflected continuation with the uniform lower bound for \(\mathsf N\)-type exterior solutions.
Recall that the terminal hypersurface \(V_J\) produced by Proposition~\ref{prop:hard-quant-reflected-continuation} intersects a fixed compact hyperbolic ball depending only on the a priori data.
Since \(\widetilde{\mathcal P}\in\mathcal B_{\mathbb H}^h\) and the family \(\Xi\) is admissible, Lemma~\ref{lem:hard-lower-bound} applies to the selected exterior solutions \(\widetilde u_\ell=u_{\widetilde{\mathcal P},\xi_\ell}\).
It yields
\begin{equation}\label{eq:hard-lower-in-proof}
    \max_{\ell=1,\ldots,m}
    \sup_{x\in V_J\cap
    \left(B_{\mathbb H}(R_*+T_*)\setminus B_{\mathbb H}(R_*)\right)}
    e^{\frac{n+1}{2}\rho(x)}
    \left|
        \partial_{\nu_{\mathbb H}}\Big|_{V_J}\widetilde u_\ell(x)
    \right|
    \geq
    c_{\mathsf N} .
\end{equation}
{
For fixed \(h\), this lower bound is uniform in
\(\widetilde{\mathcal P}\) and the far-field error.
}
Combining \eqref{eq:hard-upper-from-reflection} and \eqref{eq:hard-lower-in-proof}, we obtain
\begin{equation}
\label{eq:hard-log-input-main-proof}
    c_{\mathsf N}\bar h
    \leq
    C^M
    \bigl(\Psi(\varepsilon)\bigr)^{\beta^M}.
\end{equation}

We apply Lemma~\ref{lem:logarithmic-reduction} to \eqref{eq:hard-log-input-main-proof}.
We take
\[
    s:=\Psi(\varepsilon),
    \qquad
    \sigma=1.
\]
In the notation of Lemma~\ref{lem:logarithmic-reduction}, we take
\[
    \Gamma_M=\beta^M,
    \qquad
    B_M=M.
\]
{
Here \(C_0=C\), \(C_1=C_{\rm geo}\), \(c_0=c_{\mathsf N}\), and \(E=1\).
}
The estimate \eqref{eq:hard-M-bound-main-proof} gives the required bound for \(M\).
After decreasing the admissible error threshold, Lemma~\ref{lem:logarithmic-reduction} yields
\begin{equation}\label{eq:hard-far-field-modified-distance-eta}
    \bar h
    =
    \min\{d(\mathcal P,\mathcal P'),h\}
    \leq
    C\left(\eta(\Psi(\varepsilon))\right)^\alpha .
\end{equation}

Since
\[
    \Psi(t)
    =
    C
    \exp\!\left(-c(-\log t)^{1/2}\right),
\]
the composition \(\eta(\Psi(t))\) is bounded, for \(t\) sufficiently small, by a double-logarithmic modulus.
Thus, after reducing the small-error regime and renaming \(C\) and \(\alpha\), \eqref{eq:hard-far-field-modified-distance-eta} gives
\begin{equation}
\label{eq:hard-far-field-modified-distance-clean}
    \bar h
    \leq
    C
    \bigl(\log\log(1/\varepsilon)\bigr)^{-\alpha}.
\end{equation}

After decreasing the admissible threshold \(\varepsilon_{\mathsf N}(h)\) if necessary, the right-hand side of \eqref{eq:hard-far-field-modified-distance-clean} is smaller than \(h\).
Hence \(\bar h=d(\mathcal P,\mathcal P')\), and therefore
\begin{equation}\label{eq:hard-modified-distance-clean}
    d(\mathcal P,\mathcal P')
    \leq
    C
    \bigl(\log\log(1/\varepsilon)\bigr)^{-\alpha}.
\end{equation}
This is the modified-distance estimate \eqref{eq:main-hard-modified-distance-stability} in Theorem~\ref{thm:main-hard-farfield-simple}.

Finally, by Proposition~\ref{prop:distance-equivalence-hyperbolic},
\[
    d_{\mathcal H}^{\mathbb H}(\mathcal P,\mathcal P')
    \leq
    \delta^{-1}\bigl(d(\mathcal P,\mathcal P')\bigr).
\]
Using \eqref{eq:hard-modified-distance-clean} and the monotonicity of \(\delta^{-1}\), we obtain
\begin{equation*}
    d_{\mathcal H}^{\mathbb H}(\mathcal P,\mathcal P')
    \leq
    \delta^{-1}
    \left(
        C
        \bigl(\log\log(1/\varepsilon)\bigr)^{-\alpha}
    \right).
\end{equation*}
This is the Hausdorff stability estimate \eqref{eq:main-hard-final-stability} in Theorem~\ref{thm:main-hard-farfield-simple}.

The constants \(C>0\) and \(\alpha>0\) depend only on \(n\), on \(\lambda_0\), on the \(\mathsf N\)-type a priori data in Definition~\ref{def:hard-admissible-tg-defects}, and on \(\Xi\) through \(m\) and \(a_0(\Xi)\).
They are independent of \(\mathcal P,\mathcal P'\), \(h\), and \(\varepsilon\).

This proves Theorem~\ref{thm:main-hard-farfield-simple}.

\end{proof}

\subsection{\texorpdfstring{Stable determination of \(\mathsf D\)-type totally geodesic defects}{Stable determination of D-type totally geodesic defects}}\label{subsec:stable-soft-tg-defects}

We now prove Theorem~\ref{thm:main-soft-farfield}.  Let \(0<h\leq h_0\), and let
\[
    \mathcal P,\mathcal P'\in\mathcal A_{\mathbb H}^h
\]
be two admissible \(\mathsf D\)-type totally geodesic defects.  
Fix a boundary label \(\xi\in\mathbb S^{n-1}\), and write
\begin{equation}\label{eq:soft-exterior-solutions-section-six}
    u^i_\xi:=e_{2\lambda_0,\xi},
    \qquad
    u:=u_{\mathcal P,\xi}=u^i_\xi+u^s_{\mathcal P,\xi},
    \qquad
    u':=u_{\mathcal P',\xi}=u^i_\xi+u^s_{\mathcal P',\xi}.
\end{equation}
Define the far-field error by
\begin{equation}\label{eq:soft-farfield-error-section-six}
    \varepsilon
    :=
    \bigl\|
        u_{\infty,\mathcal P,\xi}
        -
        u_{\infty,\mathcal P',\xi}
    \bigr\|_{L^2(\mathbb S^{n-1})}.
\end{equation}
We assume \(0<\varepsilon\leq \varepsilon_{\mathsf D}(h)\).

Set
\begin{equation*}
    G:=\mathbb B^n\setminus\mathcal P,
    \qquad
    G':=\mathbb B^n\setminus\mathcal P',
\end{equation*}
and let \(H\) be the unbounded connected component of \(G\cap G'\).  
As in the \(\mathsf N\)-type case, we use the modified boundary distance
\[
    d:=d(\mathcal P,\mathcal P'),
    \qquad
    \bar h:=\min\{d,h\}.
\]
If \(d=0\), then the modified boundary discrepancy has already vanished, so there is nothing to prove.  
Hence we assume \(d>0\).

We use the same fixed observation ball
\[
    B_{\mathbb H}(x_0,\widehat\rho)
    \subset
    \mathbb B^n\setminus \overline{B_{\mathbb H}(R)},
\]
which lies in \(H\), since both defects are contained in \(\overline{B_{\mathbb H}(R)}\).

\medskip

The next proposition is the \(\mathsf D\)-type analogue of
Proposition~\ref{prop:hard-quant-reflected-continuation}.  It starts directly from
the far-field error \eqref{eq:soft-farfield-error-section-six}.  The proof is the
same reflected-chain construction as in the \(\mathsf N\)-type case, except that the Dirichlet
condition produces an odd reflection residual and no scale factor \(\bar h\) appears on
the left-hand side.

\begin{proposition}\label{prop:soft-quant-reflected-continuation}
Let \(0<h\leq h_0\), and let
\(\mathcal P,\mathcal P'\in\mathcal A_{\mathbb H}^h\).
Let \(d=d(\mathcal P,\mathcal P')\) be the modified boundary distance and set
\[
    \bar h:=\min\{d,h\}.
\]
Assume that \(\bar h>0\).
Let \(u,u'\) be defined by \eqref{eq:soft-exterior-solutions-section-six}.
Let \(\varepsilon\) be defined by \eqref{eq:soft-farfield-error-section-six}.

There exist constants
\[
    \begin{gathered}
        R_*>R,\qquad T_*>0,\qquad C\geq1,\qquad C_{\rm geo}\geq1,\\
        \beta\in(0,1),\qquad \varepsilon_0\in(0,e^{-1}).
    \end{gathered}
\]
{
The geometric exponent \(\beta\) and the constants in the bounds for \(J\) and
\(M\) depend only on the stated a priori data; \(\varepsilon_0\) may also
depend on \(h\).
}
They have the following property.
{
The far annulus may be prescribed in advance: the same conclusion holds for
every fixed \(T_*>0\) and every \(R_*\) larger than an a priori radius, with
the constants also allowed to depend on this prescribed annulus.
}

If \(0<\varepsilon\leq\varepsilon_0\), then, after possibly interchanging
\((\mathcal P,u)\) and \((\mathcal P',u')\), there exist integers \(J,M\geq1\), an
exterior solution
\[
    \widetilde u\in\{u,u'\},
\]
and a totally geodesic hypersurface \(V_J\), intersecting a fixed compact
hyperbolic ball depending only on the a priori data, such that
\begin{equation}
\label{eq:soft-reflected-continuation-upper}
    \sup_{x\in V_J\cap
    \left(B_{\mathbb H}(R_*+T_*)\setminus B_{\mathbb H}(R_*)\right)}
    e^{\frac{n-1}{2}\rho(x)}
    |\widetilde u(x)|
    \leq
    C^M\bigl(\Psi(\varepsilon)\bigr)^{\beta^M}.
\end{equation}
{
Here \(J\) is the number of reflection stages, whereas \(M\) is the total
number of elementary propagation and reflection steps, and
\[
    J\leq C_{\rm geo}\log\frac{2eR}{d}.
\]
}
Moreover,
\begin{equation}
\label{eq:M-bound-soft-reflected-continuation}
    M
    \leq
    C_{\rm geo}
    \log\frac{2eR}{d}
    \left(
        \log\frac{2eR}{d}
        +
        \log\frac{2eR}{\bar h}
    \right).
\end{equation}
\end{proposition}

\begin{proof}
The proof follows the same reflected-chain construction as
Proposition~\ref{prop:hard-quant-reflected-continuation}.
The only differences are caused by the boundary condition on each exposed
face.
In the \(\mathsf D\)-type case the reflection residual is odd rather than even,
the small boundary data are the Dirichlet trace and tangential gradient rather
than the Neumann trace, and no scale factor \(\bar h\) is lost in passing from
smallness near the face to the reflection residual.
For completeness, we indicate these changes.

The far-field-to-near-field conversion is supplied by
Lemma~\ref{lem:quantitative-rellich}, applied to
\[
    u-u'=u^s_{\mathcal P,\xi}-u^s_{\mathcal P',\xi}
\]
in the exterior of \(B_{\mathbb H}(R)\).
Thus the initial smallness on the observation ball is bounded by
\(\Psi(\varepsilon)\).
The uniform \(L^\infty\)-bound used in every application of
Lemma~\ref{lem:propagation-chain} and Lemma~\ref{lem:quant-reflection} is provided by
Proposition~\ref{prop:uniform-local-bound}, exactly as in the \(\mathsf N\)-type case.

The propagation to the first boundary discrepancy and then to a regular face patch uses the same regular-chain construction and gives the same bound for the number of balls.
On the exposed face of \(\partial\mathcal P'\), however, the exterior solution \(u'\) satisfies the homogeneous Dirichlet condition.
{
As in Step~4 of the proof of
Proposition~\ref{prop:hard-quant-reflected-continuation}, the Dirichlet part of
Lemma~\ref{lem:one-sided-reflection-extension} allows us to reflect \(u'\)
oddly across the totally geodesic hypersurface containing that face.
}
The resulting reflection residual solves \(\mathcal L_{\lambda_0}w=0\) in the corresponding local ball, agrees with \(u-u'\) on the common exterior side, and is therefore small near the face after a bounded-length propagation.
The interior \(C^1\)-estimate then gives small Dirichlet-type Cauchy data for \(u\) on the face patch, namely small \(u\) and small tangential hyperbolic gradient.

We then apply the Dirichlet, or odd-reflection, part of Lemma~\ref{lem:quant-reflection}.
This gives a small odd reflection residual across the corresponding totally geodesic hypersurface.
No Neumann trace has to be estimated, and therefore no factor \(\bar h\) is lost.
Iterating the same propagation-reflection procedure as in the \(\mathsf N\)-type case gives the terminal hypersurface \(V_J\) and the estimate \eqref{eq:soft-reflected-continuation-upper}.
{
The same strictly decreasing first-hit index gives
\(J\leq C_{\rm geo}\log(2eR/d)\), while summing the elementary steps over these stages,
including the terminal chain to the far annulus, gives
\eqref{eq:M-bound-soft-reflected-continuation}.
}
\end{proof}

\medskip

The reflected-continuation estimate above gives an upper bound for the normalized
size of a \(\mathsf D\)-type exterior solution on a terminal totally geodesic hypersurface.  We now need
the corresponding lower bound.  As in the \(\mathsf N\)-type case, the terminal hypersurface
intersects a fixed compact hyperbolic ball; let \(R_{\rm geo}>R\) be an a priori
radius such that every terminal hypersurface produced in the construction satisfies
\[
    V_J\cap \overline{B_{\mathbb H}(R_{\rm geo})}\neq\emptyset .
\]
The next lemma says that a nontrivial incoming Helgason mode cannot be canceled by an outgoing correction on the whole far part of such a hypersurface.

\begin{lemma}\label{lem:soft-lower-bound}
Let \(0<h\leq h_0\), let
\(\mathcal P\in\mathcal A_{\mathbb H}^h\), and fix
{
\(\xi\in\mathbb S^{n-1}\).  Let \(u=u_{\mathcal P,\xi}\).  There exist
\(R_*>\max\{R,R_{\rm geo}\}\) and \(T_*>0\), depending only on the \(\mathsf D\)-type a priori
data, on \(n\), on \(\lambda_0\), on \(R_{\rm geo}\), and on \(\xi\), and a
constant \(c_{\mathsf D}>0\), which may also depend on the fixed \(h\), such that for every
totally geodesic hypersurface \(V\) satisfying
\(V\cap \overline{B_{\mathbb H}(R_{\rm geo})}\neq\emptyset\),
}
\begin{equation}\label{eq:soft-lower-bound}
    \sup_{x\in V\cap
    \left(B_{\mathbb H}(R_*+T_*)\setminus B_{\mathbb H}(R_*)\right)}
    e^{\frac{n-1}{2}\rho(x)}
    |u_{\mathcal P,\xi}(x)|
    \geq c_{\mathsf D} .
\end{equation}
\end{lemma}

\begin{proof}
{
Fix \(R_*>\max\{R,R_{\rm geo}\}\) and \(T_*>0\) once and for all, depending
only on the stated a priori data.
}
Fix \(V\) with
\[
    V\cap \overline{B_{\mathbb H}(R_{\rm geo})}\neq\emptyset .
\]
By compactness of this family of totally geodesic hypersurfaces, there exists
\(\kappa_{\rm sep}>0\), depending only on \(R_{\rm geo}\), such that for every such \(V\) we can choose
\[
    \eta_V\in\partial_\infty V
\]
with
\[
    |\eta_V-\xi|\geq \kappa_{\rm sep} .
\]
Thus the boundary label \(\xi\) stays uniformly away from the endpoint \(\eta_V\).

Choose a family of points \(x_r\in V\), defined for all sufficiently large \(r\), such that
\[
    \rho(x_r)=r,
    \qquad
    \frac{x_r}{|x_r|}\to\eta_V
    \quad\text{as }r\to+\infty .
\]
From the explicit formula for \(u^i_\xi=e_{2\lambda_0,\xi}\), along this family one has
\begin{equation*}
    e^{\frac{n-1}{2}r}
    u^i_\xi(x_r)
    =
    A_Ve^{-\mathrm i\lambda_0 r}
    +
    r^i_V(r),
    \qquad
    |A_V|\geq a_s>0 .
\end{equation*}
Here \(a_s\) depends only on \(n\), \(\lambda_0\), \(R_{\rm geo}\), and \(\xi\).
The convergence \(r^i_V(r)\to0\) is uniform for the above family of \(V\), because the endpoint \(\eta_V\) stays a fixed positive distance away from \(\xi\).

The outgoing correction satisfies the outgoing radiation condition.
The outgoing Green representation and the \(C^0\) far-field asymptotics in Subsection~\ref{subsec:radition}, together with the uniform exterior bounds from Proposition~\ref{prop:uniform-decay}, give
\begin{equation*}
    e^{\frac{n-1}{2}r}
    u^s_{\mathcal P,\xi}(x_r)
    =
    B_{\mathcal P,V}e^{\mathrm i\lambda_0 r}
    +
    r^s_{\mathcal P,V}(r).
\end{equation*}
{
This follows from the fixed-sphere Green representation, the uniform kernel
expansion, and \eqref{eq:uniform-Cauchy-data-fixed-sphere}.
}
Here \(B_{\mathcal P,V}\) is uniformly bounded, and \(r^s_{\mathcal P,V}(r)\to0\) uniformly for all admissible \(\mathcal P\) and all admissible \(V\).

For the following asymptotic contradiction, choose
\[
    T^\sharp>\frac{\pi}{2\lambda_0}+1 .
\]
By the uniform convergence of the remainders, we may then choose
\(R_h^\sharp\) sufficiently large, larger also than the radius in
Proposition~\ref{prop:uniform-decay}, such that
\[
    R_h^\sharp>\max\{R,R_{\rm geo}\}
\]
and
\begin{equation}\label{eq:soft-remainder-small}
    |r^i_V(r)|+|r^s_{\mathcal P,V}(r)|
    \leq
    \frac{a_s}{8},
    \qquad
    r\in[R_h^\sharp,R_h^\sharp+T^\sharp],
\end{equation}
uniformly for all admissible \(\mathcal P\) and all admissible \(V\).

Choose
\[
    \rho_a:=R_h^\sharp+\frac12,
    \qquad
    \rho_b:=\rho_a+\frac{\pi}{2\lambda_0}.
\]
Since
\[
    T^\sharp>\frac{\pi}{2\lambda_0}+1,
\]
both \(\rho_a\) and \(\rho_b\) belong to
\((R_h^\sharp,R_h^\sharp+T^\sharp)\).
Set
\[
    X:=A_Ve^{-\mathrm i\lambda_0\rho_a},
    \qquad
    Y:=B_{\mathcal P,V}e^{\mathrm i\lambda_0\rho_a}.
\]
At the two radii \(\rho_a\) and \(\rho_b\), the principal parts of the exterior solution are
\[
    X+Y,
    \qquad
    -\mathrm iX+\mathrm iY.
\]
Since
\[
    |X+Y|^2+|-\mathrm iX+\mathrm iY|^2
    =
    2(|X|^2+|Y|^2),
\]
we have
\[
    \max\{|X+Y|,|-\mathrm iX+\mathrm iY|\}
    \geq
    |X|
    \geq
    a_s .
\]
Using \eqref{eq:soft-remainder-small}, we obtain
\[
    \max_{r\in[R_h^\sharp,R_h^\sharp+T^\sharp]}
    e^{\frac{n-1}{2}r}
    |u_{\mathcal P,\xi}(x_r)|
    \geq
    \frac{a_s}{2}.
\]
{
Consequently, \(u_{\mathcal P,\xi}\) cannot vanish identically on the portion
of \(V\) contained in the fixed annulus in
\eqref{eq:soft-lower-bound}: otherwise analytic continuation along \(V\) would
make it vanish on the above ray, contradicting the two-radius estimate.  Thus
the left-hand side is positive for every admissible pair \((\mathcal P,V)\).
Continuity follows from Proposition~\ref{prop:direct-stability}, local
\(C^0\)-convergence on the fixed annulus, and smooth convergence of \(V\).
Compactness of \(\mathcal A_{\mathbb H}^h\) and of the family of such \(V\)
then gives a positive minimum \(c_{\mathsf D}\).
}

The proof is complete.
\end{proof}

\subsubsection{Proof of Theorem~\ref{thm:main-soft-farfield}}

\begin{proof}
Let
\[
    \varepsilon
    :=
    \bigl\|
        u_{\infty,\mathcal P,\xi}
        -
        u_{\infty,\mathcal P',\xi}
    \bigr\|_{L^2(\mathbb S^{n-1})}.
\]
Set
\[
    u:=u_{\mathcal P,\xi},
    \qquad
    u':=u_{\mathcal P',\xi}.
\]
If \(\varepsilon=0\), then Theorem~\ref{thm:mainip1} gives \(\mathcal P=\mathcal P'\).
Hence the desired estimate is immediate.
We therefore assume that \(\varepsilon>0\).

Set
\[
    d:=d(\mathcal P,\mathcal P'),
    \qquad
    \bar h:=\min\{d,h\}.
\]
Since \(h>0\), the case \(\bar h=0\) is equivalent to \(d=0\).
In this case the modified boundary distance vanishes, and the Hausdorff estimate follows from Proposition~\ref{prop:distance-equivalence-hyperbolic}.
Thus it remains to consider the case \(\bar h>0\).

{
Choose \(R_*,T_*\) in Lemma~\ref{lem:soft-lower-bound} so that \(R_*\) also
exceeds the a priori radius required in
Proposition~\ref{prop:soft-quant-reflected-continuation}, and use the same
annulus in that proposition.
}
We apply Proposition~\ref{prop:soft-quant-reflected-continuation}.
This proposition starts from the far-field discrepancy \(\varepsilon\), converts it into near-field smallness through the modulus \(\Psi(\varepsilon)\), and then propagates this smallness by the Dirichlet reflected-continuation argument.
After possibly interchanging \((\mathcal P,u)\) and \((\mathcal P',u')\), we obtain a selected exterior solution
\[
    \widetilde u\in\{u,u'\},
\]
and integers \(J,M\geq1\) and a terminal totally geodesic hypersurface \(V_J\), intersecting a fixed compact hyperbolic ball, such that
\begin{equation}
\label{eq:soft-upper-from-reflection}
    \sup_{x\in V_J\cap
    \left(B_{\mathbb H}(R_*+T_*)\setminus B_{\mathbb H}(R_*)\right)}
    e^{\frac{n-1}{2}\rho(x)}
    |\widetilde u(x)|
    \leq
    C^M
    \bigl(\Psi(\varepsilon)\bigr)^{\beta^M}.
\end{equation}
{
Here \(J\) is the number of reflection stages and
\[
    J\leq C_{\rm geo}\log\frac{2eR}{d};
\]
the total number \(M\) of elementary propagation and reflection steps satisfies
\begin{equation}
\label{eq:soft-M-bound-main-proof}
    M
    \leq
    C_{\rm geo}
    \log\frac{2eR}{d}
    \left(
        \log\frac{2eR}{d}
        +
        \log\frac{2eR}{\bar h}
    \right).
\end{equation}
}
The possible interchange does not change \(d\), \(\bar h\), or \(\varepsilon\).

We now use the lower bound for \(\mathsf D\)-type exterior solutions.
Since \(\widetilde u\) is the exterior solution associated with an admissible \(\mathsf D\)-type defect and with the fixed boundary label \(\xi\), Lemma~\ref{lem:soft-lower-bound} gives
\begin{equation}
\label{eq:soft-lower-in-proof}
    \sup_{x\in V_J\cap
    \left(B_{\mathbb H}(R_*+T_*)\setminus B_{\mathbb H}(R_*)\right)}
    e^{\frac{n-1}{2}\rho(x)}
    |\widetilde u(x)|
    \geq
    c_{\mathsf D} .
\end{equation}
Combining \eqref{eq:soft-upper-from-reflection} and \eqref{eq:soft-lower-in-proof}, we obtain
\begin{equation}
\label{eq:soft-log-input-main-proof}
    c_{\mathsf D}
    \leq
    C^M
    \bigl(\Psi(\varepsilon)\bigr)^{\beta^M}.
\end{equation}
This is the \(\mathsf D\)-type analogue of the logarithmic input.
Here no factor \(\bar h\) appears on the left-hand side, because the reflected-continuation step uses Dirichlet data and odd reflection rather than a Neumann trace estimate.

We apply Lemma~\ref{lem:logarithmic-reduction} to \eqref{eq:soft-log-input-main-proof}.
We take
\[
    s:=\Psi(\varepsilon),
    \qquad
    \sigma=0.
\]
In the notation of Lemma~\ref{lem:logarithmic-reduction}, we take
\[
    \Gamma_M=\beta^M,
    \qquad
    B_M=M.
\]
{
Here \(C_0=C\), \(C_1=C_{\rm geo}\), \(c_0=c_{\mathsf D}\), and \(E=1\).
}
The estimate \eqref{eq:soft-M-bound-main-proof} gives the required control of \(M\).
Moreover, \(\bar h\leq d\) holds by definition.
After decreasing the admissible error threshold, Lemma~\ref{lem:logarithmic-reduction} yields
\begin{equation}
\label{eq:soft-far-field-modified-distance-eta}
    \bar h
    =
    \min\{d(\mathcal P,\mathcal P'),h\}
    \leq
    C
    \left(\eta(\Psi(\varepsilon))\right)^\alpha .
\end{equation}

Since
\[
    \Psi(t)
    =
    C
    \exp\!\left(-c(-\log t)^{1/2}\right),
\]
the composition \(\eta(\Psi(t))\) is bounded, for \(t\) sufficiently small, by a double-logarithmic modulus.
Thus, after reducing the small-error regime and renaming \(C\) and \(\alpha\), \eqref{eq:soft-far-field-modified-distance-eta} gives
\begin{equation}
\label{eq:soft-far-field-modified-distance-clean}
    \bar h
    \leq
    C
    \bigl(\log\log(1/\varepsilon)\bigr)^{-\alpha}.
\end{equation}

After decreasing the admissible threshold \(\varepsilon_{\mathsf D}(h)\) if necessary, the right-hand side of \eqref{eq:soft-far-field-modified-distance-clean} is smaller than \(h\).
Hence \(\bar h=d(\mathcal P,\mathcal P')\), and therefore
\begin{equation}
\label{eq:soft-modified-distance-clean}
    d(\mathcal P,\mathcal P')
    \leq
    C
    \bigl(\log\log(1/\varepsilon)\bigr)^{-\alpha}.
\end{equation}
This is the modified-distance estimate \eqref{eq:main-soft-modified-distance-stability}.

Finally, Proposition~\ref{prop:distance-equivalence-hyperbolic} gives
\[
    d_{\mathcal H}^{\mathbb H}(\mathcal P,\mathcal P')
    \leq
    \delta^{-1}\bigl(d(\mathcal P,\mathcal P')\bigr).
\]
Using \eqref{eq:soft-modified-distance-clean} and the monotonicity of \(\delta^{-1}\), we obtain
\begin{equation*}
    d_{\mathcal H}^{\mathbb H}(\mathcal P,\mathcal P')
    \leq
    \delta^{-1}
    \left(
        C
        \bigl(\log\log(1/\varepsilon)\bigr)^{-\alpha}
    \right).
\end{equation*}
This is the Hausdorff stability estimate \eqref{eq:main-soft-final-stability}.

The constants \(C>0\) and \(\alpha>0\) depend only on \(n\), on \(\lambda_0\), on the \(\mathsf D\)-type a priori data in Definition~\ref{def:admissible-tg-defects}, and on the boundary label \(\xi\).
They are independent of \(\mathcal P,\mathcal P'\), \(h\), and \(\varepsilon\).

This proves Theorem~\ref{thm:main-soft-farfield}.
\end{proof}

\bibliographystyle{plain}
\bibliography{ref}

\end{document}